\documentclass[10pt,a4paper, leqno]{article}
\usepackage{geometry}
\usepackage[utf8]{inputenc}
\usepackage{amsmath}
\usepackage{amsfonts}
\usepackage{amssymb}
\usepackage{comment}
\usepackage{hyperref}
\usepackage{enumitem,pstricks,xy,pst-node}
\xyoption{all}
\usepackage{amsthm}
\usepackage{booktabs} 
\usepackage[color]{changebar} 
\cbcolor{red} 
\usepackage{mathrsfs} 
\usepackage{verbatim} 
\usepackage[disable] 
{todonotes}
\usepackage[normalem]{ulem} 
\usepackage{caption} 
\usepackage{nicefrac} 
\usepackage{enumitem} 
\usepackage{tikz} 
\usepackage{bbm}

\usepackage{graphicx}
\usepackage{todonotes}

\usepackage{placeins}

\newtheorem{theorem}{Theorem}[section]
\newtheorem{lemma}[theorem]{Lemma}
\newtheorem{proposition}[theorem]{Proposition}
\newtheorem{conjecture}[theorem]{Conjecture}
\newtheorem{corollary}[theorem]{Corollary}
\newtheorem{definition}[theorem]{Definition}
\newtheorem{question}[theorem]{Problem}
\theoremstyle{definition}

\newtheorem{remark}[theorem]{Remark}
\newtheorem*{ack}{Acknowledgements}
\newtheorem*{fact}{Fact}
\theoremstyle{remark}
\newtheorem{example}[theorem]{Example}

\newcommand{\PP}{\mathbb{P}}
\newcommand{\FF}{\mathbb{F}}

\def\pires{{\pi_{\rm res}}}
\def\QH{{\rm QH}}
\def\BQH{{\rm BQH}}
\def\na{{\rm na}}
\def\wt{{\rm wt}}
\def\supp{{\rm Supp}}

\def\cX{{\mathcal X}}

\def\cU{{\mathcal U}}
\def\cM{{\mathcal M}}
\def\cS{{\mathcal S}}
\def\cA{{\mathcal A}}
\def\cE{{\mathcal E}}
\def\cJ{{\mathcal J}}

\def\XX{\mathbb{X}}
\def\JJ{\mathbb{J}}
\def\TT{\mathbb{T}}
\def\GG{\mathbb{G}}
\def\gg{\mathbbm{g}}
\def\PP{\mathbb{P}}
\def\AA{\mathbb{A}}
\def\RR{\mathbb{R}}
\def\CC{\mathbb{C}}
\def\HH{\mathbb{H}}
\def\hh{\mathbbm{h}}
\def\OO{\mathbb{O}}
\def\QQ{\mathbf{Q}}
\def\QQQ{\mathbb{Q}}
\def\ZZ{\mathbb{Z}}
\def\NN{\mathbb{N}}
\def\cA{{\mathcal A}}
\def\cE{{\mathcal E}}

\def\cO{{\mathcal O}}
\def\cF{{\mathcal F}}

\def\ff{\mathfrak{f}}
\def\fg{\mathfrak{g}}
\def\fb{\mathfrak{b}}
\def\fn{\mathfrak{n}}

\def\fh{\mathfrak{h}}

\def\fp{\mathfrak{p}}

\def\fgl{\mathfrak{gl}}

\def\af1{\mathbf{aff}_1}
\def\cI{{\mathcal I}}

\DeclareMathOperator{\haut}{ht}
\DeclareMathOperator{\rk}{Rk}
\DeclareMathOperator{\im}{Im}
\DeclareMathOperator{\coker}{Coker}

\DeclareMathOperator{\stab}{Stab}

\DeclareMathOperator{\Aut}{Aut}

\DeclareMathOperator{\Hom}{Hom}

\DeclareMathOperator{\codim}{codim}
\DeclareMathOperator{\Pic}{Pic}

\DeclareMathOperator{\ad}{ad}

\DeclareMathOperator{\HHH}{H}

\DeclareMathOperator{\Ker}{Ker}
\DeclareMathOperator{\Gr}{Gr}
\DeclareMathOperator{\IGr}{IGr}

\DeclareMathOperator{\OGr}{OGr}
\DeclareMathOperator{\OF}{OF}
\DeclareMathOperator{\IF}{IF}
\DeclareMathOperator{\GL}{GL}
\DeclareMathOperator{\Sp}{Sp}
\DeclareMathOperator{\PGL}{PGL}
\DeclareMathOperator{\SL}{SL}
\DeclareMathOperator{\SO}{SO}
\DeclareMathOperator{\Spin}{Spin}

\DeclareMathOperator{\pt}{pt}
\DeclareMathOperator{\modulo}{mod}
\DeclareMathOperator{\id}{Id}

\newcommand{\Hna}{{{\HHH(Y)_\na}}}
\newcommand{\Ha}{{{\HHH(Y)_{\rm a}}}}
\newcommand{\Ham}{{{\HHH(Y)_{\rm a}^{\dim Y}}}}
\newcommand{\ladi}{\begin{lastadd}}
\newcommand{\ladf}{\end{lastadd}}
\newcommand{\lrei}{\begin{lastrem}}
\newcommand{\lref}{\end{lastrem}}
\newenvironment{lastadd}
{\cbstart\color{red}}
{\todo{red to remove}\cbend}
\newenvironment{lastrem}
{\cbstart\color{yellow}}
{\cbend}

\author{Vladimiro Benedetti\thanks{Institut de Math\'ematiques de Bourgogne, UMR CNRS 5584, Universit\'e
de Bourgogne et Franche-Comt\'e, 9 Avenue Alain Savary, BP 47870,
21078 Dijon Cedex, France}, Nicolas Perrin\thanks{Laboratoire de Math\'ematiques de Versailles, UVSQ, CNRS, Universit\'e Paris-Saclay, 78035 Versailles, France}}
\title{Cohomology of hyperplane sections of (co)adjoint varieties}
\begin{document}
\maketitle

\begin{abstract}
 In this paper we study general hyperplane sections of adjoint and coadjoint varieties. We show that these are the only sections of homogeneous varieties such that a maximal torus of the automorphism group of the ambient variety stabilizes them. We then study their geometry, provide formulas for their classical cohomology rings in terms of \emph{Schubert} classes and compute the quantum Chevalley formula. This allows us to obtain results about the semi-simplicity of the (small) quantum cohomology, analogous to those holding for (co)adjoint varieties. 
\end{abstract}

\section{Introduction}

Many Fano varieties are obtained as linear sections or more generally zero loci of general sections of vector bundles over a projective rational homogeneous space. Furthermore, varieties having an action of a torus (with finitely many fixed points) have a rather simple cohomological description.

In this paper we consider $G$ a connected reductive group and $X \subset \PP(V)$ a projective $G$-homogeneous space $G$-equivariantly embedded in the projective space of a $G$-representation $V$.

\paragraph{$T$-general pairs.} The pair $X \subset \PP(V)$ is called $T$-general if a general hyperplane section of $X$ is stable under a maximal torus of $G$ (see Definitions \ref{def:T-gen} and \ref{def-t-gen}). We fully answer the following.

\begin{question}
Classify the $T$-general pairs $X \subset \PP(V)$.
\end{question}

To state our results we recall some definitions (see Subsection \ref{subsection-min-qmin-adj}). Let $G$ be a connected reductive group whose Lie algebra $\fg$ is simple. We call {\bf $G$-adjoint variety} the unique closed $G$-orbit $X \subset \PP(\fg)$. In this case $V = \fg$ is the highest weight representation of highest weight $\Theta$, the highest root of $G$. Let $\theta$ be the highest short root of $G$ (for $G$ simply laced, we have $\theta = \Theta$) and let $\nabla_\theta$ be the highest weight $G$-representation of highest weight $\theta$. The unique closed $G$-orbit $X \subset \PP(\nabla_\theta)$ is called {\bf $G$-quasi-minuscule variety}. It turns out that for $X$ adjoint or quasi-minuscule both embeddings $X \subset \PP(\fg)$ and $X \subset \PP(\nabla_\theta)$ are $T$-general pairs. We state our results for $V$ irreducible (see Lemma \ref{lemm:irr-suffit} to extend to the general situation). 

\begin{theorem}[see Theorem \ref{thm:sections-t-stables}]
  \label{thm1:intro}
  The pair $X \subset \PP(V)$ is $T$-general if and only if $G$ acts via a unique simple factor and $X = \PP(V)$, $X \subset \PP(\fg)$ is $G$-adjoint or $X \subset \PP(\nabla_\theta)$ is $G$-quasi-minuscule.
\end{theorem}

\paragraph{Linear sections of $T$-general pairs.} Let $Y$ be a general hyperplane section of $X \subset \PP(V)$ a $T$-general pair. Then $Y$ is a Fano variety and we devote the rest of the paper to study some of the geometric properties of $Y$ for $X \subsetneq \PP(V)$. We start with a description of the automorphism group of $Y$ and its infinitesimal deformations. Let $\rk(G)$ be the semi-simple rank of $G$.

\begin{proposition}[see Theorem \ref{thm:sections-t-stables}]
  Let $X \subsetneq \PP(V)$ be a $T$-general pair and let $Y$ be a general hyperplane section of $X$.
  \begin{enumerate}
    \item  If $X$ is adjoint, then the deformation space of $Y$ has dimension $\rk(G) - 1$ and we have $\Aut(Y)^0 = T^{\rm ad}$ the image of a maximal torus $T$ in the adjoint group $G^{\rm ad}$ of $G$.
\item If $X$ is quasi-minuscule and not adjoint, then the dimension of the deformation space of $Y$ and the connected component of its automorphism group are given as follows 
\FloatBarrier
\begin{table}[ht]
   \begin{tabular}{cccccccc}
     Type & $X$ & $\Aut^0(Y)$ & $\dim \HHH^1(Y,T_Y)$ \\
     \hline
      $B_n$ & $\QQ_{2n-1}$ &  $\SO_{2n}$ & $0$ \\
      $C_n$ & $\IGr(2,2n)$ & $(\SL_2)^n$ & $n - 3$ \\
      $F_4$ & $F_4/P_4$ & $\SO_8$ & $0$ \\
      $G_2$ & $\QQ_{5}$ & $\SO_6$ & $0$ \\
   \end{tabular}
   \medskip
   \centering
\end{table}
\FloatBarrier
  \end{enumerate}
  Here $\QQ_{n}$ is a smooth $n$-dimensional quadric, $\IGr(2,2n)$ is the grassmannian of line in $\CC^{2n}$ isotropic for a symplectic form and $F_4/P_4$ is the the homogeneous variety obtained as the quotient of the group of type $F_4$ by the maximal parabolic associated to the fourth fundamental weight in Bourbaki's notation \cite{bourbaki}.
\end{proposition}

An interesting feature of the proof of the above results for $X$ quasi-minuscule and $G$ non-simply laced is the use of Jordan algebras. In particular, we compute all the above invariants in terms of an associated simple Jordan algebra $\JJ$ (see Subsection \ref{subsection-coadjoint}).

\paragraph{Cohomology of $Y$.} We now consider cohomology and quantum cohomology of linear section of $T$-general pairs. For $Z$ a complex variety, we denote by $\HHH^*(Z)$ the cohomology of $Z$ with coefficients in the fiels $\QQQ$ of rational numbers. As in our situation there will be not odd-cohomology, we call a class $\eta \in \HHH^*(Z)$ of degree $d$ if $\eta \in H^{2b}(Z,\QQQ)$. If $T$ is a torus acting on $Z$, we write $\HHH^*_T(Z)$ for the $T$-equivariant cohomology with coefficients in $\QQQ$.

Let $Y \subset X$ as above and let $T \subset G$ be a maximal torus stabilising $Y$. Since $X$ has only finitely many $T$-fixed points and $T$-stable curves, the same holds true for $Y$. In particular, we have a Bia{\l}ynicki-Birula decomposition leading to two cohomology basis $(\sigma_\alpha)_{\alpha \in \aleph}$ and $(\sigma_\alpha^-)_{\alpha \in \aleph}$. Note that the indexing set $\aleph$ is the set of $T$-fixed points in $X$ and $Y$ and is the set of long (resp. short) roots of $G$ for $X$ adjoint (resp. quasi-minuscule). The class $\sigma_\alpha$ has middle degree \emph{i.e.} $\deg(\sigma_\alpha) = \dim Y$ if and only if $\alpha$ or $-\alpha$ is simple. Set $\varpi = \Theta$ (resp. $\varpi = \theta$) for $X$ adjoint (resp. $X$ quasi-minuscule). We have $\sigma_\varpi = 1$ is the unit in $\HHH^*(Y)$ and $\sigma_{-\varpi} = [\pt]$ is the class of a point. We write $h$ for the hyperplane class in $\HHH^*(Y)$ or $\HHH^*_T(Y)$ the rational and equivariant rational cohomology of $Y$. 

We prove an equivariant Chevalley formula for multiplying with $h$. Let us fix some notations. Let $\Phi$ be the set of simple roots and set $\Phi_\aleph = \Phi \cap \aleph$. For $x = \sum_{\alpha \in \Phi}x_\alpha \alpha$ a linear combination, set $\supp(x) = \{\alpha \in \Phi \ | \ x_\alpha \neq 0\}$, set $|x| = |\supp(x)|$ and set $\|x\| = \max\{|x_\alpha| \ | \ \alpha \in \Phi \}$.

\begin{theorem}[see Theorem \ref{thm:equi-chev}]
  Let $Y \subset X$ be a general hyperplane section as above.
  \begin{enumerate}
\item Let $\alpha \in \aleph$, then we have $h \cup \sigma_\alpha = (\varpi - \alpha) \sigma_\alpha + \sum_{\beta \in \aleph} a_\alpha^\beta \sigma_\beta$ in $\HHH_T^*(Y)$.
\item For $\alpha,\beta \in \aleph$, then $a_\alpha^\beta = 0$ unless ($\alpha \geq \beta$, $\beta \neq - \alpha$ and $|\alpha - \beta| = 1$) or ($\alpha \geq \beta$, $|\alpha - \beta| \in \{2,3\}$, $\supp(\alpha - \beta)$ is connected and $\supp(\alpha - \beta) \cap \{\alpha , -\beta \} \neq \emptyset$).
\item Assume that $\alpha,\beta \in \aleph$ satisfy the above condition, then
\begin{enumerate}
\item If $|\alpha - \beta| = 1$, then $a_\alpha^\beta = \| \alpha - \beta \|$.
\item If $\alpha,-\beta \in \Phi_\aleph$ are simple, then $a_\alpha^\beta = -\alpha$.
\item If $\alpha \in \Phi_\aleph$ and $-\beta \not\in \Phi_\aleph$, then $a_\alpha^\beta = \left\{\begin{array}{ll}
\| \beta \| & \textrm{ if $|\alpha - \beta| = 2$} \\
0 & \textrm{ if $|\alpha - \beta| =3$. } \\
\end{array}\right.$
\item If $\alpha \not\in \Phi_\aleph$ and $-\beta \in \Phi_\aleph$, then $a_\alpha^\beta = \left\{\begin{array}{ll}
\| \alpha \| & \textrm{ if $|\alpha - \beta| = 2$ and $\supp(\alpha - \beta) \not\subset \aleph$} \\
\|\alpha - \beta \| & \textrm{ otherwise. } \\
\end{array}\right.$
\end{enumerate}
\end{enumerate}
\end{theorem}

The class $h$ generates $\HHH^*_T(Y)$ so that this result gives a full description of the algebra structure of $\HHH_T^*(Y)$. However, it is hard to express the classes $(\sigma_\alpha)_{\alpha \in \aleph}$ as polynomials in $h$ and the multiplication rule is not easy to compute, so more geometric methods are also interesting.

On the other hand if $j : Y \to X$ is the inclusion, most classes in $\HHH^*(Y)$ can be expressed as elements in $j^*\HHH^*(X)$. In fact only the middle cohomology in $Y$ is not obtained this way. If $(\sigma_{\alpha,X})_{\alpha \in \aleph} \in \HHH^*(X)$ is the Schubert basis, we use the previous result to give an explcit formula for $j^*\sigma_{\alpha,X}$ in terms of the basis $(\sigma_{\alpha})_{\alpha \in \aleph}$. To get a useful multiplication rule in $\HHH^*(Y)$ we compute the intersection form on $\HHH^{\dim Y}(Y)$. 
Note that $(\sigma_\alpha)_{\pm \alpha \in \Phi_\aleph}$ is a basis of $\HHH^{\dim Y}(Y)$. Let $I_\aleph$ be the identity matrix of size $|\Phi_\aleph|$, let $C_\aleph = (\langle \alpha^\vee,\beta \rangle)_{\alpha,\beta \in \Phi_\aleph}$ be the submatrix of the Cartan matrix corresponding to $\Phi_\aleph$ and let $J$ be the matrix of the map $\iota : \Phi_\aleph \to \Phi_\aleph, \alpha \mapsto -w_0(\alpha)$ where $w_0 \in W$ is the longest element of the Weyl group $W$ of $G$. We prove the following result.

\begin{theorem}[see Theorem \ref{thm_inter_product_middle_cohom}]
We have the formula
  $$(\sigma_\alpha \cup \sigma_\beta)_{\alpha,\beta \in \Phi_\aleph \cup - \Phi_\aleph} = \frac{1}{4I_\aleph - C_\aleph} \left(\begin{array}{cc}
    J + (\sqrt{-1})^{\dim Y}(C_\aleph - 3I_\aleph)^2 & J + (\sqrt{-1})^{\dim Y}(C_\aleph - 3I_\aleph) \\
    J + (\sqrt{-1})^{\dim Y}(C_\aleph - 3I_\aleph)   &  J + (\sqrt{-1})^{\dim Y}I_\aleph) \\
  \end{array}\right).$$
\end{theorem}

From this result we deduce an explicit presentation of the algebra $\HHH^*(Y)$, see Theorem \ref{thm:presH}.

\paragraph{Quantum cohomology.} Next we turn to quantum cohomology. We write $\QH(Z)$ for the small quantum cohomology ring and $\BQH(Z)$ for the big quantum cohomology ring of a Fano variety $Z$. Let $Y \subset X$ as above. To understand $\QH(Y)$ and $\BQH(Y)$, we need to study the moduli space $\overline{M}_{0,n}(Y,d)$ of $n$-pointed genus $0$ stable maps of degree $d \in H_2(Y,\ZZ)$ to $Y$.

Note that for $X$ as above, we have $\Pic(X) = \ZZ$ except for $G$ of type $A_n$ with $n \geq 2$ in which case $\Pic(X) = \ZZ^2$. We also obtain result in type $A_n$ but to simplify the exposition in the introduction, we state some results in the case $\Pic(X) = \ZZ$. After proving general results, we focus on degree $1$ and $2$ curves in $Y$. We prove the following.

\begin{proposition}[see Proposition \ref{prop:m(Y,1)} and Theorem \ref{thm:m(Y,2)}]
  Let $Y \subset X$ as above with $\Pic(X) = \ZZ$.
  \begin{enumerate}
  \item $\overline{M}_{0,n}(Y,1)$ is smooth irreducible of expected dimension.
  \item If $X$ is not the adjoint variety of type $G_2$, then $\overline{M}_{0,n}(Y,2)$ is irreducible of expected dimension.
  \item If $X$ is the adjoint variety of type $G_2$, then $\overline{M}_{0,n}(Y,2)$ has two irreducible components, both of expected dimension, one of which is formed by degenerate conics.
  \end{enumerate}
\end{proposition}

Using these results, we are able to compare Gromov-Witten invariants in $Y$ to Gromov-Witten invariants in $X$ and we obtain quantum Chevalley formulae.

\begin{theorem}[see Theorems \ref{thm-qch-ad} and \ref{thm-qch-qm}]
  Let $Y \subset X$ be as above with $\Pic(X) = \ZZ$.
  \begin{enumerate}
    \item Assume that $X$ is adjoint. Let $\alpha \in \aleph$, then in $\QH(Y)$ we have:
  \begin{enumerate}
  \item If $\alpha > 0$ and $\deg(\sigma_\alpha) < c_1(Y) - 1$, then $h \star \sigma_\alpha = h \cup \sigma_\alpha$.
    \item If $\alpha > 0$ and $\deg(\sigma_\alpha) = c_1(Y) - 1$, then $h \star \sigma_\alpha = h \cup \sigma_\alpha + \|\alpha\| q\delta_{\alpha_0 \leq \alpha}$.
  \item If $\alpha$ is simple, then $h \star \sigma_\alpha = h \cup \sigma_\alpha + q h \delta_{\alpha_0 \leq \alpha}$.
  \item If $\alpha < 0$ and $|\Theta + \alpha| \geq 2$, then $h \star \sigma_\alpha = h \cup \sigma_\alpha + |\langle \Theta^\vee , \alpha \rangle| q \sigma_{s_\Theta(\alpha)}.$
  \item  If $\alpha < 0$ and $|\Theta + \alpha| =1$, then $h \star \sigma_\alpha = h \cup \sigma_\alpha + q (\sigma_{\alpha_0} + \sigma_{-\alpha_0} + \sum_{\beta \in \Phi_\aleph, \langle \beta^\vee , \alpha_0 \rangle < 0} \sigma_{-\beta}) + 2q^2$.
    
  \item  If $\alpha = - \Theta$, then $h \star \sigma_\alpha = 2q^2 h + q \sum_{\gamma \in \Phi, \langle \gamma^\vee , \alpha_0 \rangle   < 0} |\langle \gamma^\vee , \alpha_0 \rangle| \sigma_{-s_\gamma(\alpha_0)}$.
  \end{enumerate}
\item Assume that $X$ is quasi-minuscule not adjoint. Let $\alpha \in \aleph$, then in $\QH(Y)$ we have:
  \begin{enumerate}
  \item If $\alpha > 0$ is not simple, then $h \star \sigma_\alpha = h \cup \sigma_\alpha$.
  \item If $\alpha$ is simple, then $h \star \sigma_\alpha = h \star \sigma_{-\alpha}$.
  \item If $\alpha < 0$, then $h \star \sigma_\alpha = h \cup \sigma_\alpha + \delta_{\alpha \leq \theta - \Theta} q j^*\sigma_{\alpha + \Theta,X}$. We thus have

\noindent
    $\displaystyle{h \star \sigma_\alpha = \left\{\begin{array}{ll}
  \displaystyle{h \cup \sigma_\alpha + \delta_{\alpha \leq \theta - \Theta} q \bigg(\sigma_{\alpha + \Theta} + \sigma_{-\alpha-\Theta}+\sum_{\beta\in \Phi_\aleph,\ \beta+\alpha+\Theta \in \Delta}\sigma_{-\beta}}\bigg) & \textrm{for $\alpha + \Theta$ simple} \\
  \displaystyle{h \cup \sigma_\alpha + \delta_{\alpha \leq \theta - \Theta} q \sigma_{\alpha + \Theta}} & \textrm{otherwise.} \\

  \end{array}
    \right.}$
  \end{enumerate}
  \end{enumerate}
\end{theorem}

The following is a general problem for the quantum cohomology of a Fano variety $Z$.

\begin{question}
  For which Fano varieties $Z$ are $\BQH(Z)$ and $\QH(Z)$ semi-simple?
\end{question}

Note that the semi-simplicity of $\QH(Z)$ implies that $\BQH(Z)$ is semi-simple but the converse is false in general. This question was answered in \cite{adjoint} and \cite{maxim} for $X$ adjoint or quasi-minuscule.

\begin{theorem}[see \cite{adjoint,maxim}]
  Let $X$ be as above.
  \begin{enumerate}
  \item The algebra $\BQH(X)$ is semi-simple.
  \item If $X$ is quasi-minuscule, then $\QH(X)$ is not semi-simple.
  \item If $X$ is adjoint not quasi-minuscule, then $\QH(X)$ is semi-simple.
  \end{enumerate}
\end{theorem}

We conjecture that general linear sections $Y \subset X$ as above have the same behaviour as $X$. In particular we make the following conjecture.

\begin{conjecture}
  \label{conjecture:ks}
 Let $Y \subset X$ be as above.
 \begin{enumerate}
   \item The algebra $\BQH(Y)$ is semi-simple.
   \item The algebra $\QH(Y)$ is semi-simple if and only if $\QH(X)$ is semi-simple.
  \end{enumerate}
\end{conjecture}

We prove some partial results in this direction. In particular, we obtain the following results.

\begin{theorem}[see Propositions \ref{prop-non-ss-qmin-non-adj}, \ref{prop-non-ss-qmin+adj} and \ref{thm:ss}]
  Let $Y \subset X$ be as above with $\Pic(X) = \ZZ$.
  \begin{enumerate}
  \item If $X$ is quasi-minuscule not of type $C_n$ or $D_n$, then $\QH(Y)$ is not semi-simple.
  \item If $X$ is adjoint not quasi-minuscule, then $\QH(Y)$ is semi-simple.
  \end{enumerate}
\end{theorem}

In type $A_n$, a more interesting question is the restriction $\QH(X)_{\rm can}$ and $\QH(Y)_{\rm can}$ of $\QH(X)$ and $\QH(Y)$ to the canonical curve. This means to set $q_1 = q_2$ where $q_1$ and $q_2$ are the quantum parameters. The algebra $\QH(X)_{\rm can}$ is semi-simple if and only if $n$ is even and has a singularity of type $A_n$ for $n$ odd. We give a complete Chevaley formula in type $A_n$ and make the following conjecture.

\begin{conjecture}
  \label{conjecture:ks-an}
For $X$ adjoint of type $A_n$ and $Y \subset X$ a general hyperplane section,  the algebra $\QH(Y)_{\rm can}$ is semi-simple if and only if $n$ is even.
\end{conjecture}

In the appendix we prove this conjecture when $n$ is even.

\begin{ack}
The first author is partially supported by the EIPHI Graduate
School (contract ANR-17-EURE-0002). The second author was partially supported by the ANR project CATORE, grant ANR-18-CE40-0024. Both authors acknowledge support from the ANR project FanoHK, grant ANR-20-CE40-0023. 
\end{ack}

\setcounter{tocdepth}{1}
\tableofcontents

\newpage

\paragraph{Notations}

Let $G$ be a connected reductive group and $\mathfrak{g}$ its Lie algebra. Then $\mathfrak{g}$ is a $G$-representation: the adjoint representation. Let $T$ be a maximal torus of $G$ and let $\mathfrak{h}\subset \mathfrak{g}$ be its Lie algebra. Denote by $\Delta$ the set of roots and by $W$ the Weyl group associated to $(G,T)$. We write $\fg_\alpha$ for the weight space associated to the root $\alpha \in \Delta$. Fix $B$ a Borel subgroup containing $T$ and let $\fb \subset \fg$ be its Lie algebra. Let $\Delta_+$ be the set of positive roots defined as the set of root $\alpha \in \Delta$ such that $\fg_\alpha \subset\fb$. Denote by $\Phi \subset \Delta$ the set of simple root and by $\Delta_-$ the set of negative roots.

The finite dimensional representation theory of reductive Lie algebras is well understood in terms of weights. All representations are completely reducible, and isomorphism classes of irreducible representations are in a one-to-one correspondance with non-negative characters of $T$. By non-negative characters we mean elements $w\in \chi(T)$, where $\chi(T)$ is the set of group homomorphisms $T\to \CC^*$, such that $w(\alpha)\geq 0$ for all $\alpha\in \Phi$. Indeed, if $V$ is a finite $\fg$-representation via $\rho:\fg\to \fgl(V)$, it is also a $\fh$-representation. Since $\fh$ is composed of semi-simple commuting operators, one can diagonalize its action on $V$ and decompose it in eigenspaces $V = \oplus_{\gamma \in \chi(T)} V_\gamma$, with $\chi(T)\subset \fh^\vee$. If $V$ is irreducible, there exists a unique weight $\lambda$, called the highest weight, such that $\rho(\fn)\cdot V_\lambda = 0$, where $\fn$ is a maximal nilpotent subalgebra of $\fb$. Furthermore $V_\lambda$ is one-dimensional and $\langle \alpha^\vee , \lambda \rangle \geq 0$ for all $\alpha\in \Delta_+$. Vice versa, starting from such a $\lambda$, one can construct an irreducible representation, which we will denote by $\nabla_\lambda$, whose highest weight is $\lambda$.

Denote by $\Theta$ the highest (long) root of $\Delta$, and by $\theta$ the highest short root of $\Delta$. If $G$ is simply laced, then $\Theta = \theta$. If $G$ is semi-simple, then $\mathfrak{g}$ is the unique irreducible $G$-representation whose highest weight is $\Theta$ \emph{i.e.} $\fg = \nabla_\Theta$. In general for $\fg$ reductive, we have $\nabla_\Theta = [\fg,\fg]$. For $I \subset \Phi$ a set of simple roots, denote by $P_I$ the parabolic subgroup whose simple roots are given by $\Phi \setminus I$: for $\alpha \in \Phi$, we have $\fg_{-\alpha} \subset \fp_I = {\rm Lie}(P_I) \Leftrightarrow \alpha \in \Phi \setminus I$. For example $P_\Phi = G$ and $P_\emptyset = B$.

Recall that if $G$ is simply laced, all roots are conjugated by the Weyl group and have the same length, while if $G$ is not simply laced, there are two conjugacy classes of roots of different lengths, long roots and short roots. In the former case we will consider the roots to be long and short, and in both cases we define $\Delta_l=\{ \alpha\in \Delta \mid \alpha\mbox{ is a long root} \}$ and $\Delta_s=\{ \alpha\in \Delta \mid \alpha\mbox{ is a short root} \}$. For $x = \sum_{\alpha \in \Phi}x_\alpha \alpha$ a linear combination of simple roots, we define $\supp(x) = \{\alpha \in \Phi \ | \ x_\alpha \neq 0\}$. Set $|x| = |\supp(x)|$ and set $\|x\| = \max\{|x_\alpha| \ | \ \alpha \in \Phi \}$. We define the height of $x$ as $\haut(x) = \sum_\alpha x_\alpha$.

Recall that, for any vector space $V$, there exists an exponential map $\exp:\fgl(V)\to \GL(V)$, which is the usual exponential of matrices. The exponential map $\exp:\fg \to G$ exists in general for any linear algebraic group, and it allows to lift any $\fg$-representation $\rho:\fg\to \fgl(V)$ to a $G$-representation (beware: for this, we need to consider the ``simply connected version" of $G$) satisfying
$$\exp(g)\cdot v=\sum_{i\geq 0}\frac{1}{i!}\rho(g)^i \cdot v.$$
From this, one sees that the induced representation $\rho:G\to \GL(V)$ is compatible with the exponential maps (of $\fg$ and $\fgl(V)$).

Let $\mathfrak{g}$ be a simple Lie algebra, and let $\mathfrak{g}=\mathfrak{h} \oplus \bigoplus_{\alpha\in \Delta}\mathfrak{g}_\alpha$ be the Cartan decomposition of $\mathfrak{g}$. Recall that the restriction of the Killing form to $\mathfrak{h}$ is non degenerate and induces a duality between $\mathfrak{g}_\alpha$ and $\mathfrak{g}_{-\alpha}$ for $\alpha\in\Delta$. Let $e_\alpha\in\mathfrak{g}_\alpha$. Then there exists $h_\alpha\in\mathfrak{h}$ and $e_{-\alpha}\in\mathfrak{g}_{-\alpha}$ such that $[e_\alpha,e_{-\alpha}] = h_\alpha$, $[h_\alpha,e_{\alpha}] = 2e_{\alpha}$, $[h_\alpha,e_{-\alpha}] = - 2e_{-\alpha}$.

We have $\ad(e_{\alpha})^{3}(e_{-\alpha})=0$ and for $\alpha \neq \beta$ positive roots, we have $\ad(e_{\gamma})^{-\langle \gamma^\vee , \alpha \rangle + 1}(e_{\alpha})=0$. For the exponential map $\exp:\mathfrak{g}\to G$, we have $\exp(\mathfrak{b})\subset B$ for $\mathfrak{b}$ the Borel subalgebra corresponding to $B$, and $\exp(\mathfrak{h})\subset T$. Any $G$-representation $V$ is also a $\mathfrak{g}$-representation and these representations are compatible with the exponential map.

\section{Torus stable hyperplane sections of homogeneous spaces}

Let $G$ be a connected reductive group and $X \subset \PP(V)$ be a $G$-homogeneous projective variety $G$-equivariantly embedded.

\begin{definition}
  \label{def:T-gen}
  The pair $X \subset \PP(V)$ is called $T$-general if a general hyperplane section is stable under a maximal torus of $G$.
\end{definition}

Note that this definition is stable under central extensions of $G$ since maximal tori will act via their quotient by the center of $G$. Any $X$ as above is of the form $X = G/P$ where $P$ is a parabolic subgroup of $G$. The variety $X$ can be $G$-equivariantly embedded in a projective space in many ways. Any such embedding is of the form $\PP(V)$ where $V$ is a $\widetilde{G}$-representation for $\widetilde{G}$ the simply connected cover of $G$. By the above remark we may therefore assume that $V$ is a $G$-representation. If the embedding $X \subset \PP(V)$ is non degenerate \emph{i.e.} if $X$ is not contained in any proper linear subspace of $\PP(V)$, then $V = \nabla_\lambda$ is the highest weight representation of highest weight $\lambda$, and if $v \in V_\lambda$ is a highest weight vector and $[v] \in \PP(V)$ is the class of $v$, we have $P = \stab_G([v])$ and $X = G \cdot [v]$.

\begin{lemma}
  \label{lemm:irr-suffit}
  Let $X \subset \PP(V)$ be a $G$-homogeneous space $G$-equivariantly embedded in $\PP(V)$. Assume that $V$ is not irreducible as representation.
  \begin{enumerate}
  \item Then there exists an unique irreducible representation $\nabla \subset V$ such that $X \subset \PP(\nabla) \subset \PP(V)$.
  \item The pair $X \subset \PP(V)$ is $T$-general if and only if the pair $X \subset \PP(\nabla)$ is $T$-general.
  \end{enumerate}
\end{lemma}

\begin{proof}
1. Since $X$ is projective and homogeneous, we have $X = G \cdot [v]$ with $\stab_G([v])$ a parabolic subgroup. In particular $v$ is a highest weight vector and if we define $\nabla$ as the $G$-submodule generated by $v$, the result follows.

2. Linear section are given by elements in $V^\vee$ and $\nabla^\vee$. Furthermore the inclusion $\nabla \subset V$ induces a surjective open morphism $\pi : V^\vee \to \nabla^\vee$. If $\mathcal{U}_\nabla \subset \nabla^\vee$ and $\mathcal{U}_V \subset V^\vee$ are the sets of hyperplane sections of $X \subset \PP(\nabla)$ and $X \subset \PP(V)$ stable under a maximal torus of $G$, then $\mathcal{U}_V = \pi^{-1}(\mathcal{U}_\nabla)$ proving the result.
\end{proof}

In particular, in what follows we may and will assume that the representation $V$ is irreducible. We are thus interested in general hyperplane sections $Y$ of $X$ in $\PP(V)$ with $V$ irreducible which are stable under the action of a maximal torus $T$. Let $V^\vee$ be the dual representation of $V$. Hyperplane sections are indexed by elements $[h] \in \PP(V^\vee)$ with $h \in V^\vee$ non-zero via $[h] \mapsto Y_{h} = \{ [x] \in X \ | \ h(x) = 0 \}$.

\begin{lemma}
  The hyperplane section $Y_h$ is $T$-stable if and only if $h \in V^\vee$ is a $T$-eigenvector.
\end{lemma}

\begin{proof}
  If $h$ is a $T$-eigenvector of weight $\lambda$, we have, for any $[w] \in Y_h$ and any $t \in T$, the equality $h(t \cdot w) = (t^{-1} \cdot h)(w) = \lambda(t)^{-1} h(w) = 0$ thus $t \cdot [w] \in Y_h$.

  Conversely, since $V$ is irreducible, the restriction map $\HHH^0(\PP(V),\cO_{\PP(V)}(1)) \to \HHH^0(X,\cO_X(1))$ is an isomorphism, and we have that, for $t \in T$, the equality $Y_{t \cdot h} = t \cdot Y_h = Y_h$ implies $[t \cdot h] = [h]$ and $h$ is a weight vector.
\end{proof}

We now consider the situation where the group $G = G_1 \times G_2$ is a product of two reductive groups. Note that, up to a finite central extension, this is equivalent to the fact that the same holds at the Lie algebra level: $\fg = \fg_1 \times \fg_2$. We let $T = T_1 \times T_2$ be a maximal torus with $T_1$ and $T_2$ maximal tori of $G_1$ and $G_2$. Let $X = G/P \simeq G_1/P_1 \times G_1/P_2$ where $P = P_1 \times P_2$ is a product of parabolic subgroups of $G_1$ and $G_2$. Note also that a $G$-representation $V$ is irreducible if and only if $V = V_1 \otimes V_2$ with $V_1$ and $V_2$ irreducible representations of $G_1$ and $G_2$ (see \cite[\S 12]{steinberg}).

\begin{lemma}
  \label{lemma:no-product}
 Let $G = G_1 \times G_2$ and $X \subset \PP(V)$ with $V = V_1 \otimes V_2$ as above. The pair $X \subset \PP(V)$ is $T$-general if and only if, up to reordering, we have $P_2 = G_2$, the representation $V_2$ is  trivial and $X \subset \PP(V_1)$ is a $T_1$-stable pair for $G_1$
\end{lemma}

\begin{proof}
  Assume that both $V_1$ and $V_2$ are non-trivial. In particular any $T$-weight space of $V_1$ and $V_2$ or equivalently of $V_1^\vee$ and $V_2^\vee$ are proper subspaces. Let $f \in V^\vee = V_1^\vee \otimes V_2^\vee$. If $f$ is a $T$-eigenvector then $f \in V_{1,\lambda_1}^\vee \otimes V_{2,\lambda_2}^\vee$ for weights $\lambda_1$ and $\lambda_2$. But this implies that the rank of $f$ as a tensor is at most $\min(\dim(V_{1,\lambda_1}^\vee),\dim(V_{2,\lambda_2}^\vee)) < \min(\dim V_1,\dim V_2)$, which is the rank of a general tensor. In particular $f$ is not general. This implies that $V_1$ or $V_2$ must be trivial. Assuming that $V_2$ is trivial, the result follows. 
\end{proof}

We may and will therefore assume that $\fg$ is a simple Lie algebra.

\subsection{Minuscule, quasi-minuscule and adjoint representations}
\label{subsection-min-qmin-adj}

Let $T$ be a maximal torus of $G$. Many aspects of the geometry of $X$ can be interpreted using weights of the $T$-action on $V$ and on $X$. Denote by $\Pi(V)$ the set of $T$-weights of $V$.

\begin{definition}
  Let $V = \nabla_\lambda$ be an irreducible representation of highest weight $\lambda$. 
\begin{enumerate}
\item The representation $V$ or its highest weight $\lambda$ is called \textbf{minuscule} if $\Pi(V) = W \cdot \lambda$.
\item The representation $V$ or its highest weight $\lambda$ is called \textbf{quasi-minuscule} if $\Pi(V) \setminus \{0\} = W \cdot \lambda$.
\end{enumerate}
\end{definition}

Let $\Delta$ be the root system of $G$ and recall that $\Theta$ is the highest (long) root.

\begin{definition}  
A representation $V$ is called \textbf{adjoint} if it is irreducible of highest weight $\Theta$. The highest weight $\Theta$ is called \textbf{adjoint} weight.
\end{definition}

Let $G$ be a reductive group and $G^\vee$ its dual Langlands group. The Dynkin diagram of $G^\vee$ is obtained from that of $G$ by reversing the arrows. The fundamental weights of $G$ are denoted by $(\varpi_i)_{i \in [1,n]}$ while the fundamental weights of $G^\vee$ are denoted by $(\varpi_i^\vee)_{i \in [1,n]}$. We keep the same numbering of simple roots and fundamental weights for $G$ and $G^\vee$. For a weight $\varpi = \sum_i a_i \varpi_i$, we set $\varpi^\vee = \sum_i a_i \varpi_i^\vee$.

\begin{definition}
  Recall the definition of minuscule and adjoint representations.
  \begin{enumerate}
  \item A representation $V$ is called \textbf{cominuscule} if it is irreducible of highest weight $\varpi$ such that $\varpi^\vee$ is a minuscule weight for $G^\vee$. The weight $\varpi$ is called \textbf{cominuscule}.
 \item A representation $V$ is called \textbf{coadjoint} if it is irreducible of highest weight $\varpi$ such that $\varpi^\vee$ is an adjoint weight for $G^\vee$. The weight $\varpi$ is called \textbf{coadjoint}.
  \end{enumerate}
\end{definition}

Quasi-minuscule, (co)minuscule or (co)adjoint representations are well known. See \cite{seshadri} for a classification of minuscule and quasi-minuscule ones. We list them in Table \ref{table_rep} using notations as in \cite{bourbaki} for roots systems. For example, recall that $\theta$ is the highest short root (for simply laced groups, we have $\theta = \Theta$), then we have the following well known observation.

\begin{lemma}
A representation is quasi-minuscule if and only if it is irreducible of highest weight~$\theta$.
\end{lemma}

\FloatBarrier
\begin{table}[ht]
     \begin{tabular}{c|c|c|c|c|c}
     Type of $G$ & minuscule  & cominuscule & quasi-minuscule & adjoint & coadjoint\\
\hline
$A_n$ & $(\varpi_i)_{i \in [1,n]}$ & $(\varpi_i)_{i \in [1,n]}$ &
$\varpi_1 + \varpi_n$ & $\varpi_1 + \varpi_n$ & $\varpi_1 + \varpi_n$ \\
$B_n$ & $\varpi_n$  & $\varpi_1$ & $\varpi_1$ & $\varpi_2$ & $2\varpi_1$ \\
$C_n$ & $\varpi_1$  & $\varpi_n$ & $\varpi_2$ & $2\varpi_1$ & $\varpi_2$\\
$D_n$ & $\varpi_1$, $\varpi_{n-1}$, $\varpi_n$  & $\varpi_1$, $\varpi_{n-1}$, $\varpi_n$  &
$\varpi_2$ & $\varpi_2$ & $\varpi_2$ \\
$E_6$ & $\varpi_1$, $\varpi_{6}$ & $\varpi_1$, $\varpi_{6}$ &
$\varpi_2$ & $\varpi_2$ & $\varpi_2$ \\
$E_7$ & $\varpi_7$ & $\varpi_7$ 
& $\varpi_1$ & $\varpi_1$ & $\varpi_1$ \\
$E_8$ & None & None & $\varpi_8$ & $\varpi_8$ & $\varpi_8$ \\
$F_4$ & None & None & $\varpi_4$ & $\varpi_1$ & $\varpi_4$ \\
$G_2$ & None & None & $\varpi_1$ & $\varpi_2$ & $\varpi_1$ 
  \medskip
     \end{tabular}
     \centering
   \caption{\label{table_rep} Quasi-minuscule, (co)minuscule and (co)adjoint weights}
\end{table}
\FloatBarrier

\begin{remark}
Notice that quasi-minuscule and coadjoint weights coincide except for type $B_n$ for which the quasi-minuscule weight is $\varpi_1$ while the coadjoint weight is $2\varpi_1$.
\end{remark}

\begin{definition}
  Let $V$ be an irreducible representation, let $v \in V$ an highest weight vector and $X = G \cdot [v] \subset \PP(V)$ be the closed orbit. The variety $X$ is called \textbf{minuscule, cominuscule, quasi-minuscule, adjoint} or \textbf{coadjoint} if $V$ is minuscule, cominuscule, quasi-minuscule, adjoint or coadjoint.    
\end{definition}

The list of adjoint, coadjoint and quasi-minuscule varieties is given in Table \ref{table_adjoint}

\FloatBarrier
\begin{table}[ht]
   \begin{tabular}{c||c|c||c|c|c}
      G & Parabolic & Adjoint variety & Parabolic & Coadjoint & Quasi-minuscule \\
\hline
$A_n$ & $P_{1,n}$ & ${\rm Fl}(1,n ; n+1)$ & $P_{1,n}$ & ${\rm Fl}(1,n;n+1)$ & ${\rm Fl}(1,n;n+1)$ \\ 
$B_n$ & $P_2$ & $\OGr(2,2n+1)$ & $P_1$ & $v_2(\QQ_{2n-1})$ & $\QQ_{2n-1}$ \\
$C_n$ & $P_1$ & $v_2(\PP^{2n-1})$ & $P_2$ & $\IGr(2,2n)$ & $\IGr(2,2n)$ \\
$D_n$ & $P_2$ & $\OGr(2,2n)$ & $P_2$ & $\OGr(2,2n)$ & $\OGr(2,2n)$ \\
$E_6$ & $P_2$ & $E_6/P_2$ & $P_2$ & $E_6/P_2$ & $E_6/P_2$ \\
$E_7$ & $P_1$ & $E_7/P_1$ & $P_1$ & $E_7/P_1$ & $E_7/P_1$ \\
$E_8$ & $P_8$ & $E_8/P_8$ & $P_8$ & $E_8/P_8$ & $E_8/P_8$ \\
$F_4$ & $P_1$ & $F_4/P_1$ & $P_4$ & $F_4/P_4$ & $F_4/P_4$\\
$G_2$ & $P_2$ & $G_2/P_2$ & $P_1$ & $G_2/P_1$ & $G_2/P_1$ 
   \end{tabular}
\medskip
  \centering
  \caption{\label{table_adjoint} Adjoint, coadjoint and quasi-minuscule varieties}
\end{table}
\FloatBarrier

\begin{remark}
Notice that in the $B_n$ case for the coadjoint variety and in the $C_n$ case for the adjoint variety, the embeddings $\QQ_{2n-1} \simeq v_2(\QQ_{2n-1}) \subset \PP((S^2\CC^{2n+1})_0)$ and $\PP^{2n-1} \simeq v_2(\PP^{2n-1})\subset \PP(\mathfrak{sp_{2n}})\cong \PP(S^2(\CC^{2n}))$ are the Veronese embeddings given by the line bundle $\cO(2)$.
\end{remark}

\subsection{$T$-general representations}

It is easy to produce hyperplace sections $Y_h$ of $X$ which are $T$-stable. However, these hyperplane sections will be very special and singular in general. Our aim is to find a representation $V$ as above such that a \emph{general} hyperplane section is $T$-stable. Let us be more precise in the following definition.

\begin{definition}
  \label{def-t-gen}
  Let $V$ be an irreducible $G$-representation and $\PP(V^\vee)^T \subset \PP(V^\vee)$ the set of $T$-fixed points. The representation $V$ is called \textbf{$T$-general} if $G \cdot \PP(V^\vee)^T$ is dense in $\PP(V^\vee)$.
\end{definition}

\begin{lemma}
Let $X \subset \PP(V)$ be as before. Then $V$ is a $T$-general representation if and only if a general hyperplane section $Y_h\subset X$ is $G$-conjugate to a $T$-stable hyperplane section.
\end{lemma}

\begin{proof}
Let $Y_h$ be a hyperplane section with $h\in V^\vee$ general. If $Y_h$ is $G$-conjugate to a $T$-stable section, it means that there exists $g\in G$ such that $g\cdot Y_h=Y_{g\cdot h}$ is $T$-stable. Hence $g\cdot h$ is a $T$-eigenvector, thus showing that $V$ is $T$-general. The converse follows from the definitions.
\end{proof}

\begin{proposition}
  \label{prop-t-stable}
  Let $V$ be a $T$-general representation of $G$.
  \begin{enumerate}
  \item For $G$ not of type $A_n$ or $C_n$, then $V = \nabla_\Theta$ or $V = \nabla_\theta$.
  \item If $G$ is of type $A_n$, then $V = \nabla_\Theta$, $V = \nabla_{\varpi_1}$ or $V = \nabla_{\varpi_n}$.
     \item If $G$ is of type $C_n$, then $V = \nabla_\Theta$, $V = \nabla_\theta$ or $V = \nabla_{\varpi_1}$.
  \end{enumerate}
\end{proposition}

\begin{proof}
  For the special cases of type $A_n$ and $C_n$, the corresponding representations have a unique non-zero orbit and therefore are obviously $T$-general. The case of type $A_1$ is also easily verified so we assume that $\rk(G) \geq 2$.

  We now prove that adjoint representations are $T$-general. For coadjoint and quasi-minuscule representations, this will be addressed in Section \ref{subsection-coadjoint} using Jordan algebras. Note that adjoint representations are self-dual. The crucial point for the adjoint case is Kostant's Theorem stating the equality of quotients: $\mathfrak{g} \mathbin{/\mkern-6mu/} G \cong \mathfrak{h} \mathbin{/} W$. In particular any general $G$-orbit in $\mathfrak{g}$ contains an element of $\mathfrak{h}$ and therefore a $T$-stable element $h$.

  In the sequel we consider $V$ an irreducible representation not listed in the statement of the proposition. Let $V^\vee$ be its dual and $\lambda$ be the highest weight of $V^\vee$. We prove that $V$ cannot be $T$-general. Assume that $G \cdot \PP(V^\vee)^T$ is dense. This means that there exists a weight $\mu$ such that $G \cdot \PP(V^\vee_\mu)$ is dense, where $V_\mu^\vee$ is the weight space associated to the weight $\mu$. Up to acting by an element of $G$, we may assume that $\mu$ is dominant. Note that for $\mu = \lambda$, the space $\PP(V_\mu^\vee)$ is a point and $G \cdot \PP(V_\mu^\vee)$ is the closed orbit in $\PP(V^\vee)$ and is not dense (since $X$ is not a projective space). Therefore we must have $\mu \neq \lambda$. In particular $V^\vee$ (and therefore $V$) cannot be minuscule. The set of weights $\Pi(V^\vee)$ contains at least two $W$-orbits: $W \cdot \lambda$ the orbit of the highest weight $\lambda$ and $W \cdot \mu$. We will now compare dimensions to prove that $V$ is not $T$-general. To simplify notation, set $\cO_{\lambda,\mu}$ be the set of $W$-orbits in $\Pi(V^\vee)$ different from $W \cdot \lambda$ and $W \cdot \mu$.

Recall that $T$ fixes $\PP(V^\vee_\mu)$, therefore we have
  $$\dim G \cdot \PP(V_\mu^\vee) \leq \dim(V_\mu^\vee) - 1 + \dim G - \rk(G) = \dim(V_\mu^\vee) - 1 + |\Delta|,$$
where $\Delta$ is the root system of $G$. On the other hand, because of the description of $\Pi(V^\vee)$, we have $\dim \PP(V^\vee) \geq |W \cdot \lambda| + |W \cdot \mu|  \dim(V_\mu^\vee) + \sum_{\nu \in \cO_{\lambda,\mu}}
|W \cdot \nu| - 1$. If $G \cdot \PP(V_\mu^\vee)$ is dense in $\PP(V^\vee)$, we get $|W \cdot \lambda| + |W \cdot \mu| \dim(V_\mu^\vee) + \sum_{\nu \in \cO_{\lambda,\mu}}
|W \cdot \nu| - 1 \leq \dim \PP(V) = \dim G \cdot \PP(V_\mu^\vee) \leq \dim(V_\mu^\vee) - 1 + \dim G - \rk(G) = \dim(V_\mu^\vee) - 1 + |\Delta|$ and therefore the inequality
\begin{equation}
  \label{eq:ineg}
  |\Delta| \geq |W \cdot \lambda| + (|W \cdot \mu| - 1)\dim(V_\mu^\vee) + \sum_{\nu \in \cO_{\lambda,\mu}}
  |W \cdot \nu|.
\end{equation}
We deduce the following fact.

\begin{fact}
  For $\nu \in \Pi(V^\vee)$ with $\nu \neq \mu$ or $\nu \neq 0$, we have $|\Delta| > |W \cdot \nu|$.
\end{fact}

\begin{proof}
  Assume first that $\nu \neq \mu$. If $\mu \neq 0$, we have $|W \cdot \mu| - 1 > 0$ and the result follows. If $\mu = 0$, remark that since $V$ is not quasi-minuscule, so is $V^\vee$ and therefore there exists $\nu' \in \Pi(V^\vee)$ with $\nu' \neq \lambda$ and $\nu' \neq 0$. We get $ |\Delta| \geq |W \cdot \lambda| + |W \cdot \nu'|$. We may choose $\nu'$ such that $\nu$ is either equal to $\lambda$ or $\nu'$, proving the result.

  If $\mu = \nu \neq 0$, then $|W \cdot \mu| \geq 2$ (recall that $\rk(G) \geq 2$) and we thus have $|\Delta| \geq |W \cdot \lambda| + |W \cdot \mu|$ proving the result.
\end{proof}

We need to understand the size of $W$-orbits on weights. Note that $|W \cdot \nu| = |W/W_\nu|$ where $W_\nu$ is the stabiliser of $\nu$. The group $W_\nu$ is the Coxeter subgroup of $W$ generated by all simple roots $\alpha$ such that $\langle \alpha^\vee , \nu \rangle = 0$. Recall that $\Phi$ denotes the set of simple roots and set $\Phi_\nu = \{ \alpha \in \Phi \ | \ \langle \alpha^\vee , \nu \rangle \neq 0\}$. It is now easy to check the following fact.

\begin{fact}
The condition $|\Delta| > |W \cdot \nu|$ implies $|\Phi_\nu| = 1$.
\end{fact}

\begin{proof}
With no loss of generality, we will show that $|\Delta| \leq |W \cdot \nu|$ when $|\Phi_\nu|=2$. We will identify simple roots with the corresponding nodes in the Dynkin diagram (nodes are ordered following Bourbaki's convention). 

Suppose that $G$ is of type $A_n$ and $\Phi_{\nu}=\{i,j\}$. Then $|\Delta|=n(n+1)$, $|W|=(n+1)!$ and $|W_\nu|=i!(j-i)!(n+1-j)!\leq (n-1)!$ for any $i,j$, thus showing that $|\Delta| \leq |W \cdot \nu|$. If $G$ is of type $B_n$ or $C_n$ and $\Phi_{\nu}=\{i,j\}$, then $|\Delta|=2n^2$, $|W|=2^{n}(n)!$ and $|W_\nu|=i!(j-i)!2^{n-j}(n-j)!\leq 2^{n-2}(n-2)!$, thus giving the claim in this case as well. If $G$ is of type $D_n$ and $\Phi_{\nu}=\{i,j\}$, then $|\Delta|=2n(n-1)$, $|W|=2^{n-1}(n)!$ and $|W_\nu|=i!(j-i)!2^{n-j-1}(n-j)!\leq 2^{n-3}(n-2)!$ if $n-1\notin \{i,j\}$, while $|W_\nu|=(n-1)!$ if $\{n-1,n\}= \{i,j\}$; once again from these computations the claim follows. In the exceptional cases, one can proceed with a case by case analysis.
\end{proof}

The previous fact is in particular true for $\nu = \lambda$, therefore $\lambda = m\varpi$ is a multiple of a fundamental weight $\varpi$. Let $\alpha$ be the simple root such that $\langle \alpha^\vee , \varpi \rangle = 1$.

Assume first that we have $m = 1$. Recall that $V$ and therefore $V^\vee$ are not minuscule and that we assume that $V$ and therefore $V^\vee$ are neither adjoint nor quasi-minuscule. It is easy to check that the above inequality \eqref{eq:ineg} is never satisfied. Indeed, following a similar argument as the one appearing in the proof of the previous \textbf{Fact}, one can show that $|\Delta| > |W \cdot \lambda|$ implies that $\lambda$ is either minuscule, adjoint or quasi-minuscule, or $\lambda = \varpi = \varpi_n$ in type $C_n$; this last case is excluded by a further inspection of inequality \eqref{eq:ineg}.

Assume now that $m \geq 2$. The weight $\nu = \lambda - \alpha = m\varpi - \alpha$ is in $\Pi(V^\vee)$. Furthermore, this weight is dominant and non trivial since if $\beta$ is a simple root with $\langle \beta^\vee , \alpha \rangle \neq 0$ (recall that $\rk(G) \geq 2$), we have $\langle \beta^\vee , \nu \rangle = - \langle \beta^\vee , \alpha \rangle > 0$. By the above facts, we thus have $|\Delta| > |W \cdot \nu|$ and $|\Phi_\nu| = 1$. This implies that $\nu$ is a multiple $\nu = m' \varpi_\beta$ of the fundamental weight dual to the simple coroot $\beta^\vee$. Therefore we must have $\langle \alpha^\vee , \nu \rangle = 0$ which implies $m = 2$. Furthermore, there must be a unique simple root $\beta$ such that $\langle \beta^\vee , \alpha \rangle \neq 0$ so that $\alpha$ is a simple root at an end of the Dynkin diagram. Note that since $\nu \neq 0$, we have $|\Delta| \geq |W \cdot \lambda| + |W \cdot \nu|$. Indeed, for $\mu \neq \nu$, this is obvious from the inequality \eqref{eq:ineg}. For $\mu = \nu$ we have $\mu \neq 0$, thus $|W \cdot \mu| \geq 2$ and the result follows again from inequality \eqref{eq:ineg}. The only possibilities are $\lambda = 2 \varpi_1$ in types $B_n$ and $C_n$. We recover the adjoint and coadjoint non fundamental weights.

To finish the proof, we need to check that in type $B_n$, the representation $V  = \nabla_{2\varpi_1} \simeq V^\vee$ of highest weight $2\varpi_1$ is not $T$-general. We have $\Pi(V^\vee) = W\cdot\{2\varpi_1,\varpi_2,0\}$ with $\dim(V_{2\varpi_1}) = \dim(V_{\varpi_2}) = 1$ and $\dim(V_0) = n$. Furthermore it is easy to check that $G \cdot \PP(V_{2\varpi_1})$ and $G \cdot \PP(V_{\varpi_2})$ are not dense. Now we have $\dim(G \cdot \PP(V_0)) \leq \dim G - \rk(G) + \dim \PP(V_0)  = 2n^2 + n - 1 < 2n^2 + 3n - 1 = \dim \PP(V)$ so that $V$ is not $T$-general.
\end{proof}

Let $G$ be a reductive group and let $\rk(G)$ be its rank.
We obtain the following result. 

\begin{theorem}
  \label{thm:sections-t-stables}
  Let $X \subset\PP(V)$ be a projective rational homogeneous space. The general hypersurface $Y$ of $X$ of multidegree $\mathbf{d}$ is $T$-stable for some maximal torus $T \subset G$ if and only if one of the following conditions is satisfied
  \begin{enumerate}
  \item $X = \PP^n$, $\mathbf{d} = 1$ and $G$ is of type $A_n$;
  \item $X = \PP^{2n-1}$, $\mathbf{d} = 1$ and $G$ is of type $C_n$;
  \item $X$ is an adjoint or a quasi-minuscule variety and $\mathbf{d} = 1$.
  \end{enumerate}
  
  For $X$ adjoint and $G$ not of type $C_n$, the deformation space of $Y$ has dimension $\rk(G) - 1$ and we have $\Aut(Y)^0 = T^{\rm ad}$ the image of a maximal torus $T$ in the adjoint group $G^{\rm ad}$ of $G$.

  For $X$ quasi-minuscule and not adjoint, the dimension of the deformation space of $Y$ and the connected component of its automorphism group are given in Table \ref{table_coadjoint-def}.
\end{theorem}

\FloatBarrier
\begin{table}[ht]
   \begin{tabular}{cccccccc}
     Type & Weight & $X$ & $\Aut^0(Y)$ & $\dim \HHH^1(Y,T_Y)$ \\
     \hline
      $B_n$ & $\varpi_1$ & $\QQ_{2n-1}$ &  $\SO_{2n}$ & $0$ \\
      $C_n$ & $\varpi_2$ & $\IGr(2,2n)$ & $(\SL_2)^n$ & $n - 3$ \\
      $F_4$ & $\varpi_4$ & $F_4/P_4$ & $\SO_8$ & $0$ \\
      $G_2$ & $\varpi_1$ & $\QQ_{5}$ & $\SO_6$ & $0$ \\
   \end{tabular}
   \medskip
   \centering
   \caption{\label{table_coadjoint-def} Hyperplane section of coadjoint varieties}
\end{table}
\FloatBarrier

\begin{remark}
  Note that the unique coadjoint non quasi-minuscule variety $X = v_2(\QQ_{2n-1}) \subset \PP((S^2\CC^{2n+1})_0)$ does not appear in the above result. Indeed, a general hyperplane section will be an even dimensional complete intersection of two quadrics and is not $T$-stable. However, we will see in the next sections that many computations also hold for this variety. We give the deformation information in this case (see Proposition \ref{prop-deform-coadjoint}).
  \medskip
  
  \centerline{
   \begin{tabular}{cccccccc}
     Type & Weight & $X$ & $\Aut^0(Y)$ & $\dim \HHH^1(Y,T_Y)$ \\
     \hline
      $B_n$ & $2\varpi_1$ & $\QQ_{2n-1}$ &  $1$ & $2n-3$ \\
   \end{tabular}}
\end{remark}

\begin{remark}
  In \cite{manben}, the possible automorphism groups of smooth hyperplane sections $Y=Y_h$ of adjoint varieties have been classified, including the component group $\Aut(Y)/\Aut^0(Y)$. It turns out that, in the non simply laced case, there exist \emph{special} smooth hypersurfaces $Y$ for which the automorphism group does not contain any maximal torus of $G$. In the simply laced case the situation is simpler: $h$ is a semi-simple element in $V^\vee\cong \mathfrak{g}$ if and only if $Y$ is smooth, and in such a situation $\Aut^0(Y)=T^{\rm ad}$; moreover, except in type $A$, $\Aut(Y)/\Aut^0(Y)$ can be identified with the subgroup of the Weyl group of $G$ stabilizing $[h]\in \PP(\mathfrak{g})$.
\end{remark}

\begin{remark}
  Let $X$ be a $G$-homogeneous variety, $E$ an irreducible homogeneous vector bundle on $X$ such that $\HHH^0(X,E)$ is a $T$-general representation. From Proposition \ref{prop-t-stable} one directly recovers that the zero locus $Y$ inside $X$ of a general section of $E$ is stabilized by a maximal torus of $G$. This applies for instance to the case $X=\IGr(k,2n)$, $E=\wedge^2 \cU^\vee$, where $\cU$ is the rank $k$ tautological bundle; in such situation $Y$ is called a \emph{bisymplectic} Grassmannian. The same bundle over $X=\OGr(k,m)$ gives as $Y$ the so-called \emph{orthosymplectic} Grassmannians.
\end{remark}

\medskip

\begin{proof}
  The first part of the result follows from Proposition \ref{prop-t-stable}. We  need to prove the last assertions on automorphism groups and deformation spaces.

  In the adjoint case not of type $A$, using Bott's Theorem, we have $\HHH^0(Y,T_X|_Y)=\HHH^0(X,T_X)$ so that any automorphism of $Y$ in the connected component of the identity lifts to an automorphism of $X$. Indeed, there exists a contact structure on adjoint varieties (\cite{Bcontact}), i.e. a non-split exact sequence of vector bundles
  \begin{equation}\label{contact_sequence}
  0\to F \to T_X \to \cO(1)\to 0
  \end{equation}
  such that the differential $df:\wedge^2 F \to \cO(1)$ of the map $f:T_X\to\cO(1)$ is everywhere non-degenerate, thus giving an isomorphism $F^\vee\cong F(-1)$. Since $\cO(-1)$ is acyclic, the dual of \eqref{contact_sequence} gives that $\HHH^i(F^\vee)\cong \HHH^i(\Omega^1_X)$ for all $i\in \ZZ$; in particular, $\HHH^1(F^\vee)\cong \CC$ is the only non-zero cohomology (here we used the fact that $ \HHH^1(\Omega^1_X)\cong \CC$, which is not true in type $A$). As a consequence, twisting \eqref{contact_sequence} by $\cO(-1)$, one gets that $T_X(-1)$ is acyclic, which in turn implies that $\HHH^0(Y,T_X|_Y)=\HHH^0(X,T_X)$ by a direct application of the Koszul complex of $Y\subset X$.
  
  For $G \neq \Sp_{2n}$ we have $\Aut(X) = G^{\rm ad}$ the adjoint form of $G$ and $\Aut(Y) = C_{G^{\rm ad}}(h)$ where $h \in \fh = {\rm Lie}(T)$ is such that $Y = Y_h$. Since $h$ is general, it is regular and $C_{G^{\rm ad}}(h) = T^{\rm ad}$. Note that in type $C_n$, $X=\PP^{2n-1}$ and $\Aut(X) = \PGL_{2n}$ while $Y$ is a quadric hypersurface of dimension $2n$ and $\Aut(Y) = \SO_{2n-1}$. To compute the dimension of the deformation space, we use the exact sequence
  $$0 \to \HHH^0(Y,T_Y) \to \HHH^0(Y,T_X\vert_Y)  = \mathfrak{g} \to \HHH^0(Y,\mathcal{O}_Y(1)) = \mathfrak{g}/\langle h \rangle \to \HHH^1(Y,T_Y) \to 0,$$
  where the map $\mathfrak{g} \to \mathfrak{g}/\langle h \rangle \simeq h^\perp$ (this last isomorphism is induced by the Killing form) is given by $\ad_h$. We therefore get $\HHH^0(Y,T_Y) = \mathfrak{c}_\mathfrak{g}(h) = \mathfrak{h}$ with $\mathfrak{h} = {\rm Lie}(T)$ and $\HHH^1(Y,T_Y) = h^\perp \cap \mathfrak{h} \simeq \mathfrak{h}/\langle h \rangle$.

  In the coadjoint but non-adjoint case and in type $A$, the automorphisms and the deformation spaces are studied via Jordan algebras in Subsection \ref{subsection-coadjoint}. Note that in type $G_2$, the variety $X$ is a $5$-dimensional quadric and $\Aut(X) = \SO_{7}$ while $Y$ is a $4$-dimensional quadric hypersurface and $\Aut^0(Y) = \SO_{6}$.

  For the unique quasi-minuscule non-adjoint and non-coadjoint case, we have $X = \QQ_{2n-1} \subset \PP^{2n}$, $G = \SO_{2n+1}$ and $Y = \QQ_{2n-2}$. Therefore $Y$ has no deformation and $\Aut^0(Y) = \SO_{2n}$.
\end{proof}

\begin{remark}
  We obtain a non-homogeneous rigid Fano variety $Y$ with Picard number $1$ as the hyperplane section of $F_4/P_4$. We have $\dim(Y) = 14$, $c_1(Y) = 10$ and $\Aut^0(Y) = \SO_8$.
\end{remark}

\subsection{Coadjoint varieties and Jordan algebras}
\label{subsection-coadjoint}

In this subsection, we consider $X$ a quasi-minuscule or a coadjoint non-adjoint variety. In particular, the variety $X$ is homogeneous for the action of a non simply laced group $G$. Let $V$ be the irreducible coadjoint $G$-representation and let $v$ be a highest weight vector, then $X = G \cdot [v] \subset \PP(V)$. Recall that $\theta$ is the highest short root of $G$ and that $\varpi_1$ is its first fundamental weight. Recall that $\nabla_\varpi$ is the highest weight representation of highest weight $\varpi$ of $G$. The list of the quasi-minuscule or coadjoint non adjoint varieties $X$ is reported in Table \ref{table_coadjoint}.

\FloatBarrier
\begin{table}[ht]
   \begin{tabular}{c|c|c|c}
G & $V$ & Parabolic & $X$ \\
    \hline
    $B_n$ & $\nabla_{\varpi_1}$ 
    & $P_1$ & $\QQ_{2n-1}$ \\
    $B_n$ & $\nabla_{2\varpi_1}$ 
    & $P_1$ & $v_2(\QQ_{2n-1})$ \\
    $C_n$ & $\nabla_\theta$ 
    & $P_2$ & $\IGr(2,2n)$ \\
    $G_2$ & $\nabla_\theta$ 
    & $P_1$ & $\QQ_5$ \\
    $F_4$ & $\nabla_\theta$ 
    & $P_4$ & $F_4/P_4$ \\
   \end{tabular}
   \medskip
   \centering
   \caption{\label{table_coadjoint} Coadjoint varieties }
\end{table}
\FloatBarrier

We will interpret the representation $V$ occuring in Table \ref{table_coadjoint} in terms of Jordan algebras.

\subsubsection{Jordan algebras}

In this subsection, we recall the definition and few facts on Jordan algebras that we will need in the sequel. We refer to \cite{springer-jordan} for definitions and results.

\begin{definition}
A Jordan algebra $A$ is a $\CC$-vector space with a bilinear commutative product satisfying the following rule: $x(x^2y) = x^2(xy)$ for all $x,y \in A$. A Jordan algebra is called simple if it contains no non-trivial ideal.
\end{definition}

\begin{example}
  We describe all possible simple Jordan algebras in the following examples.
  \begin{enumerate}
  \item If $A$ is an associative algebra, then $A$ with the product $x \bullet y = \frac{1}{2}(xy + yx)$ is a Jordan algebra. Almost all semi-simple Jordan algebras are obtained this way.
    \item Let $\AA$ be a complex composition algebra \emph{i.e.} $\AA = \AA_\RR \otimes_\RR \CC$, where $\AA_\RR$ is a real division algebra, namely $\AA_\RR = \RR, \CC,\HH, \OO$, the real numbers, the complex numbers, the quaternions or the octonions. For $a \in \AA$, denote by $\bar a$ its conjugate and set
      $J_n(\mathbb{A}) = \{ M \in M_n(\AA) \ | \ M^t = \bar M\}.$
      Then $J_n(\AA)$ with the product $M \bullet N = \frac{1}{2}(MN + NM)$ is a Jordan algebra with the restriction $n = 3$ for $\AA = \OO$. These Jordan algebras are denoted by $\cS_n$, $\cM_n$, $\cA_n$ and $\cE_3$ for $\AA_\RR = \RR$, $\CC$, $\HH$ and $\OO$.
    \item Let $V$ be a vector space of dimension $n$ endowed with a non-degenerate bilinear form $B$. Let $e \in V$ such that $B(e,e) = 1$. Then the product
      $x \bullet y = x B(y,e) + y B(x,e) - e B(x,y)$
induces a Jordan algebra structure on $V$ denoted by $\cO_{2,n}$.
  \end{enumerate}
\end{example}

The following is a construction of any simple Jordan algebra. Let $\Gamma$ be a connected reductive group and fix a maximal torus $T$ in $\Gamma$. Let $\Pi \subset \Gamma$ be a parabolic subgroup containing $T$ whose unipotent radical is abelian and which is conjugate to its opposite $\Pi^-$ with respect to $T$. For $\Gamma$ simple, the condition on the unipotent radical imposes that $\Pi$ is a maximal parabolic subgroup associated to a cominuscule fundamental weight $\varpi$. The possible triples $({\rm Type}(\Gamma),\varpi,\Pi)$ are as follows: $(A_{2n-1},\varpi_n,P_n)$ ; $(B_n,\varpi_1,P_1)$ ; $(C_n,\varpi_n,P_n)$ ; $(D_n,\varpi_1,P_1)$ ; $(D_{2n},\varpi_{2n},P_{2n})$ ; $(D_{2n},\varpi_{2n-1},P_{2n-1})$ ; $(E_7,\varpi_7,P_7)$.

Let $(\Gamma,\Pi)$ be as above, let $\JJ$ be the Lie algebra of the unipotent radical of $\Pi$ and let $f \in \JJ^-$ be a general element in the Lie algebra $\JJ^-$ of the unipotent radical of $\Pi^-$. Define a product on $\JJ$ as follows: for $x,y \in \JJ$ set
$$x \bullet y = \frac{1}{2} [x,[f,y]].$$

Let $e \in \JJ$ such that, setting $\alpha^\vee = [e,f]$, we have $[\alpha^\vee,e] = 2e$ and $[\alpha^\vee,f] = - 2f$. Then $e$ is the unit for $\JJ$ and we may define invertible elements and the inverse $x^{-1}$ of $x \in \JJ$. Let $j : \JJ \dasharrow \JJ$ be the rational map defined by $j(x) = x^{-1}$. We define two subgroups of $\GL(\JJ)$.

\begin{definition}
The structure group is defined by $\GG = \{ g \in \GL(\JJ) \ | \ \exists h \in \GL(\JJ), g \circ j = j \circ h \}$.
\end{definition}

\begin{proposition}[See {\cite[Chapter~2 and~14]{springer-jordan}}]
  Let $\Gamma$, $\Pi$ and $\JJ$ as above.
  \begin{enumerate}
  \item The space $\JJ$ with the above product is a simple Jordan algebra.
    \item The Jordan algebra does not depend (up to isomorphism) on the choice of $f \in \JJ^-$.
    \item Any simple Jordan algebra is obtained this way.
      \item The $\GG$-representation $\JJ$ is irreducible.
  \end{enumerate}
\end{proposition}

\begin{lemma}[{See \cite[Pages 10-11]{springer-jordan}}]
For $g$ in $\GG$, there is a unique $h$ such that $g \circ j = j \circ h$. Furthermore, the map $\sigma : \GG \to \GG, g \mapsto \sigma(g) = h$ is a group automorphism of $\GG$.
\end{lemma}

\begin{lemma}[{See \cite[Proposition 12.2, Corollary 12.4 and Theorem 14.27]{springer-jordan}}]
For a simple Jordan algebra, the group $\GG$ is reductive and $\GG \cdot e$ is dense in $\JJ$.
\end{lemma}

\begin{definition}
The automorphism group of $\JJ$ is defined by $G = \{ g \in \GG \ | \ \sigma(g) = g \textrm{ and } g \cdot e = e \}$.
\end{definition}

\begin{lemma}[{See \cite[Proposition 4.6 and Theorem 14.27]{springer-jordan}}]
We have $G = \stab_\GG(e)$ and for a simple Jordan algebra, the orbit $\GG \cdot e \simeq \GG/G$ is the dense $\GG$-orbit in $\JJ$. Furthermore $G$ is a semi-simple group.
\end{lemma}

In Table \ref{para-jordan} we list the simple Jordan algebras, together with the groups $\Gamma$, $\GG$ and $G$, the maximal parabolic subgroup $\Pi \subset \Gamma$ and the $\GG$-representation $\JJ$ (see \cite[Sections 14.25-14.31]{springer-jordan}).

\FloatBarrier
\begin{table}[ht]
   \begin{tabular}{c|c|c|c|c|c}
Type of $\Gamma$ & Parabolic $\Pi$ & $\GG/Z(\GG)$ & $G^\circ$ & $\JJ$ & Jordan algebra\\
    \hline
$A_{2n-1}$ & $P_n$ & $\PGL_n(\CC) \times \PGL_n(\CC)$ & $\PGL_n(\CC)$ & $M_n(\CC)$ & $\cM_n$ \\
$B_{n+1}$ & $P_1$ & ${\rm PO}_{2n}(\CC)$ & ${\rm O}_{2n-1}(\CC)$ & $\CC^{2n}$ & $\cO_{2,2n}$ \\
$C_{n+1}$ & $P_{n+1}$ & $\PGL_n(\CC)$ & ${\rm O}_n(\CC)$ & $S^2(\CC^n)$ & $\cS_n$ \\
    $D_{n+1}$ & $P_1$ & ${\rm PO}_{2n+1}(\CC)$ & ${\rm O}_{2n}(\CC)$ & $\CC^{2n+1}$ & $\cO_{2,2n+1}$ \\
$D_{2n}$ & $P_{2n-1}$ or $P_{2n}$ & $\PGL_{2n}(\CC)$ & ${\rm PSp}_{2n}(\CC)$ & $\Lambda^2\CC^{2n}$ & $\cA_{n}$ \\
    $E_7$ & $P_7$ & $E_6$ & $F_4$ & $\CC^{27}$ & $\cE_3$ \\
   \end{tabular}
\centering
   \caption{\label{para-jordan} Simple Jordan algebras}
\end{table}
\FloatBarrier

Note that $\GG^\sigma$ contains $G$. The following is a consequence of \cite[Sections 14.25-14.31]{springer-jordan}.

\begin{lemma}
The group $G$ has finite index in $\GG^\sigma$.
\end{lemma}

Let ${\rm Lie}(\Gamma)$, $\gg$ and $\fg$ be the Lie algebras of $\Gamma$, $\GG$ and $G$. As $\GG$-representation, we have
$${\rm Lie}(\Gamma) = \JJ^\vee \oplus \gg \oplus \JJ.$$

This is also the decomposition with respect to weights of $\alpha$. If ${\rm Lie}(\Gamma)(i)$ is the weight space of weight $i$ for $\alpha^\vee$, we have ${\rm Lie}(\Gamma)(-2) = \JJ^\vee$, ${\rm Lie}(\Gamma)(0) = \gg = [\gg,\gg] \oplus \CC\alpha^\vee$ and ${\rm Lie}(\Gamma)(2) = \JJ$. We have $e \in \JJ$ and $f \in \JJ^\vee$ and both generate a dense $\GG$-orbit.

\begin{lemma}[{See \cite[Chapter 14]{springer-jordan}}] For a simple Jordan algebra, the Lie algebras $[\gg,\gg]$ and $\fg$ are simple. Furthermore $\gg = \mathfrak{c}(\gg) \oplus [\gg,\gg]$ and $\mathfrak{c}(\gg) = \CC \alpha^\vee$ is one-dimensional.
\end{lemma}

\subsubsection{Symmetric spaces and $T$-general representations}

The group $G$ is obtained as a subgroup of finite index in the invariant subgroup of the above group involution $\sigma$ of $\GG$. We may therefore apply results concerning symmetric spaces to the situation. In this subsection, we recall few facts on symmetric spaces and refer to \cite[Section 26]{timashev} for more details.

The group $(\GG,\GG)$ admits maximal $\sigma$-stable tori. There are two important types of maximal $\sigma$-stable tori $\TT$: those for which $\TT^\sigma$ is of maximal dimension and the split maximal tori for which the dimension of $\TT_{-1} = \{t \in \TT \ | \ \sigma(t) = t^{-1} \}$ is maximal. We will use both types of $\sigma$-stable tori.

For any $\sigma$-stable torus $\TT$, the involution $\sigma$ induces an involution on roots and we have $\sigma(\gg_\beta) = \gg_{\sigma(\beta)}$ for any root $\beta$ of $(\GG,\TT)$. We consider the eigenspace decomposition with respect to $\sigma$ and set $E_1 = \{ x \in E \ | \ \sigma(x) = x\}$ and $E_{-1} = \{ x \in E \ | \ \sigma(x) = -x\}$ for any subspace $E \subset \gg$. Denote by $\hh$ the Lie algebra of $\TT$. We have the decompositions $[\gg,\gg] = [\gg,\gg]_1 \oplus [\gg,\gg]_{-1}$ and $\hh = \hh_1 \oplus \hh_{-1}$.

\begin{lemma}
We have $\gg_1 = [\gg,\gg]_1 = \fg$ as Lie algebras and $\gg_{-1} = [\gg,\gg]_{-1} \oplus \mathfrak{c}(\gg)$ is a $\fg$-module.
\end{lemma}

\begin{proof}
The equality $\gg_1 = \gg^\sigma$ implies that $\gg_1$ is a Lie algebra. Furthermore, the subgroup $G$ has finite index in $\GG^\sigma$ so $\fg = \gg_1 = [\gg,\gg]_1 \oplus \mathfrak{c}(\gg)_1$ and since $\fg$ is simple we must have $\fg = [\gg,\gg]_1$. This Lie algebra acts on $\gg_{-1}$ via the restriction of the adjoint representation of $\gg$ on itself and the previous argument implies that $\gg_{-1} = [\gg,\gg]_{-1} \oplus \mathfrak{c}(\gg)$.
\end{proof}

Set $J = [\gg,\gg]_{-1}$. Recall that $\theta$ is the highest short root of $G$ and that $\varpi_1$ is its first fundamental weight. Recall that $\nabla_\varpi$ is the highest weight representation of highest weight $\varpi$ of $G$.

\begin{lemma}
  As $G$-representations, we have the following results: 
  \begin{enumerate}
  \item We have $\gg = \fg \oplus J \oplus \CC\alpha^\vee$ and $\JJ \simeq J \oplus \CC e \simeq J \oplus \CC f \simeq \JJ^\vee$.
  \item The representation $J$ is the following irreducible representation of $G$
    \FloatBarrier
\begin{table}[ht]
   \begin{tabular}{c|c|c|c|c|c|c}
Type of $\JJ$ & $\cM_n$ & $\cO_{2,2n}$ & $\cS_n$ & $\cO_{2,2n+1}$ & $\cA_n$ & $\cE_3$ \\
    \hline
$G^\circ$ & $\PGL_n(\CC)$ & ${\rm O}_{2n-1}(\CC)$ & ${\rm O}_n(\CC)$ & ${\rm O}_{2n}(\CC)$ & ${\rm PSp}_{2n}(\CC)$ & $F_4$ \\
    $J$ & $\nabla_\theta$ & $\nabla_\theta$ & $\nabla_{2\varpi_1}$ & $\nabla_{\varpi_1}$ & $\nabla_\theta$ & $\nabla_\theta$ \\
   \end{tabular}
\centering
\end{table}
\FloatBarrier
  \end{enumerate}
\end{lemma}

\begin{remark}
  \label{rem-quasi-min-J}
  Remark that, except for type $G_2$, any quasi-minuscule non-adjoint representation occurs as the representation $J$ above as cases $\cO_{2,2n}$, $\cA_n$ and $\cE_3$.
\end{remark}

\begin{proof}
  The second statement is an easy case by case check: decompose the Lie algebra $[\gg,\gg]$ as $\fg$-representation. This statement also follows from the description of the weight space decompositions of $\fg$ and $J$ with respect to a maximal $\sigma$-stable torus $\TT$ such that $\TT^\sigma$ has maximal dimension (in that case $(\TT^\sigma)^\circ$ is a maximal torus of $G$). Since $G$ fixes $e$ and $f$, the maps $\ad_e$ and $\ad_f$ are morphisms of $G$-representations and we therefore get the decompositions $\JJ = J  \oplus \CC e$ and $\JJ^\vee = J \oplus \CC f$. 
\end{proof}

Let $\TT$ be a split maximal torus. We will decribe the decomposition of $\fg$ and $J$ in terms of root spaces for $\TT$. Note that for a root $\beta$ of $\gg$, $\sigma(\beta)$ is again a root. We have three possible cases (see \cite[Section 26.4]{timashev}):
  \begin{itemize}
  \item[-] [Complex roots] $\sigma(\beta) \neq \pm \beta$, then both $(\gg_\beta\oplus\gg_{\sigma(\beta)})_1$ and $(\gg_\beta\oplus\gg_{\sigma(\beta)})_{-1}$ are one-dimensional.
  \item[-] [Real roots] $\sigma(\beta) = -\beta$, then both $(\gg_\beta\oplus\gg_{\sigma(\beta)})_1$ and $(\gg_\beta\oplus\gg_{\sigma(\beta)})_{-1}$ are one-dimensional.
     \item[-] [Imaginary compact roots] $\sigma(\beta) = \beta$ and $\sigma\vert_{\gg_{\beta}} = \id$, then $(\gg_\beta)_1 = \gg_\beta$ and $(\gg_\beta)_{-1} = 0$.
  \end{itemize}
  Note that for such tori, there are no imaginary non-compact roots (roots $\beta$ for which $\sigma(\beta) = \beta$ and $\sigma\vert_{\gg_{\beta}} = -\id$) and $\sigma$ maps positive complex or real roots to negative roots.

  \begin{lemma}
    If $\TT$ is a split maximal torus, we have the following weight space decompositions $\fg = \hh_1 \oplus_{\sigma(\beta) = \beta} \gg_\beta \oplus_{\sigma(\beta) \neq \beta} (\gg_\beta \oplus \gg_{\sigma(\beta)})_1$ and $J = \hh_{-1} \oplus_{\sigma(\beta) \neq \beta} (\gg_\beta \oplus \gg_{\sigma(\beta)})_{-1}$.
  \end{lemma}

  \begin{proof}
Follows from the above description of roots and the fact that $\TT$ is $\sigma$-stable, therefore the weight space descomposition is compatible with $\sigma$.
  \end{proof}

  Let $h \in \hh_{-1}$ be a general element. We view $h \in J \subset \JJ$ as an element of ${\rm Lie}(\Gamma)(2)$. We therefore have a map $\ad_h : \fg \to J$.

  \begin{lemma}
    \label{lem-ker-coker}
We have $\Ker\ad_h = \hh_1 \oplus_{\sigma(\beta) = \beta} \gg_\beta$ and $\coker\ad_h = \hh_{-1}$.
\end{lemma}

\begin{proof}
  An element $h \in \hh_{-1}$ acts (via $\ad_h$) by scalar multiplication on the weight spaces $\gg_\beta$ and trivially on $\hh$. In particular $\hh_1 \subset \Ker\ad_h$. Furthermore, the space $\hh_{-1}$ acts on $\gg$ via all weights $\lambda\in \hh_{-1}^\vee$ such that $\sigma(\lambda) = -\lambda$. These weights are orthogonal to roots $\beta$ such that $\sigma(\beta) = \beta$ so that $\gg_\beta \subset \Ker\ad_h$ in that case. Finally, since $h$ is general, its weight on all weight spaces $\gg_\beta \oplus \gg_{\sigma(\beta)}$ is non-trivial if $\sigma(\beta) \neq \beta$
  proving that the restriction of $\ad_h$ on $(\gg_\beta \oplus \gg_{\sigma(\beta)})_1 \to (\gg_\beta \oplus \gg_{\sigma(\beta)})_{-1}$ is an isomorphism. The result follows from this observation.
\end{proof}

We describe general hyperplane sections in $\PP(J)$.

\begin{proposition}
  For $\TT$ a maximal split torus, the orbit $G \cdot \PP(\hh_{-1})$ is dense in $\PP(J)$ or equivalently a general $G$-orbit in $J$ contains an element $h \in \hh_{-1}$.
\end{proposition}

\begin{proof}
  By Lemma \ref{lem-ker-coker}, we have $J = \ad_h(\fg) \oplus \hh_{-1}$ proving that $G \cdot \PP(\hh_{-1})$ is dense in $\PP(J)$.
\end{proof}

Note that if the maximal split torus $\TT$ is also a maximal $\sigma$-stable torus such that $\TT^\sigma$ has maximal dimension, then $T = (\TT^\sigma)^\circ$ is a maximal torus of $G$ and $\PP(\hh_{-1}) \subset \PP(J)^T$. In particular we obtain the following result.

\begin{corollary}
  \label{coro-T-gen}
If there exists a maximal split torus $\TT$ for which $\TT^\sigma$ has maximal dimension, then $J$ is a $T$-general representation.
\end{corollary}

\begin{remark}
  It is easy from the classification to compute the dimension of $\TT_{-1}$ when $\TT$ is a maximal split torus or when $\TT$ is such that $\TT^\sigma$ has maximal dimension. If both dimensions agree, then there are maximal split tori $\TT$ such that $\TT^\sigma$ has maximal dimension. We describe this in the following table. The case for which $\TT^\sigma$ has maximal dimension is described by
  $T = (\TT^\sigma)^\circ$ while the other one is the maximal split case. Note that the dimension of $\TT_{-1}$ in the first case equals $\rk(\GG) - \rk(G)$ while in the second case it is the dimension of the restricted root system. We refer to \cite[Section 26]{timashev} for descriptions of maximal tori for symmetric spaces.

\FloatBarrier
\begin{table}[ht]
     \begin{tabular}{c|c|c|c}
       Jordan algebra & $\dim\TT_{-1}$ for $T = (\TT^\sigma)^\circ$ & $\dim\TT_{-1}$ for maximal split & $J$ \\
\hline
$\cO_{2,2n}$ & $1$ & $1$ & $\nabla_\theta$ \\
$\cO_{2,2n+1}$ & $0$ & $1$ & $\nabla_{\varpi_1}$ \\
$\cS_n$ & $\lfloor \frac{n}{2} \rfloor$
& $n - 1$ & $\nabla_{2 \varpi_1}$ \\
$\cM_n$ & $n-1$ & $n - 1$ & $\nabla_\theta$ \\
$\cA_n$ & $n - 1$ & $n - 1$ & $\nabla_\theta$ \\
$\cE_3$ & $2$ & $2$ & $\nabla_\theta$ 
  \medskip
     \end{tabular}
     \centering
   \caption{\label{table-tori} When do the two types of maximal $\sigma$-stable tori coincide?}
\end{table}
\FloatBarrier

In particular, we see (recall Remark \ref{rem-quasi-min-J}) that, except for type $G_2$, all quasi-minuscule non-adjoint representations come from a Jordan algebra for which the two types of $\sigma$-stable maximal tori agree.
\end{remark}

\begin{proposition}
  Quasi-minuscule representations are $T$-general.
\end{proposition}

\begin{proof}
  Assume first that $G$ is not of type $G_2$. Then $V = J$ for a Jordan algebra for which the two type of $\sigma$-stable maximal tori agree. The result follows from Corollary \ref{coro-T-gen}. If $G = G_2$, then $V$ is the standard representation of $\SO_7(\CC)$ and $V^T$ is one-dimensional. It is easy to check that the Lie algebra of the stabiliser of any element in $V^T$ contains the maximal torus and $\mathfrak{g}_\alpha$ for any long root $\alpha$. The stabiliser has therefore dimension $8$ and $G \cdot V^T$ has dimension $14 + 1 - 8 = 7$ and is dense in $V$, the result follows.
\end{proof}

\subsubsection{Deformation of general hyperplane sections}

From our discussion on Jordan algebras, the following result is easy to check.

\begin{lemma}
  \label{lemm-def-XX}
Let $X$ be a quasi-minuscule or coadjoint variety for the group $G$, then $X$ is a general hyperplane section of a variety $\XX$ homogeneous for the group $\GG$.
\end{lemma}

We describe the varieties $X$ and $\XX$ in the following table.

\FloatBarrier
\begin{table}[ht]
     \begin{tabular}{c|c|c|c|c}
  Jordan algebra & $\GG/Z(\GG)$ & $G^\circ$ & $\XX$ & $X$ \\
\hline
$\cO_{2,n}$ & ${\rm PO}_{n}(\CC)$ & ${\rm O}_{n-1}(\CC)$ & $Q_{n-2}$ & $Q_{n-3}$ \\
$\cS_n$ & $\PGL_n(\CC)$ & ${\rm O}_n(\CC)$ & $v_2(\PP^{n-1})$ & $v_2(Q_{n-2})$ \\
$\cM_n$ & $\PGL_n(\CC) \times \PGL_n(\CC)$ & $\PGL_n(\CC)$ & $\PP^{n-1} \times (\PP^{n-1})^\vee$ & ${\rm Fl}(1,n-1;n)$ \\
$\cA_n$ & $\PGL_{2n}(\CC)$ & ${\rm PSp}_{2n}(\CC)$ & ${\rm Gr}(2,2n)$ & ${\rm IGr}(2,2n)$ \\
$\cE_3$ & $E_6$ & $F_4$ & $E_6/P_6$ & $F_4/P_4$
  \medskip
     \end{tabular}
     \centering
     \caption{\label{table-sect-hyp} Quasi-minuscule and coadjoint varieties as hyperplane sections of homogeneous spaces}
\end{table}
\FloatBarrier

We use this description of $X$ as a hyperplane section to compute the local deformation of the general hyperplane section $Y$ of $X$.

\begin{proposition}
  \label{prop-deform-coadjoint}
  Let $X$ be quasi-minuscule or coadjoint but not adjoint for $G$ and let $Y$ be a general hyperplane section defined by $h \in J = J^\vee$.
  \begin{enumerate}
    \item We have $\HHH^0(Y,T_Y) = \Ker\ad_h$ and $\HHH^1(Y,T_Y) = (\hh_{-1}\oplus\CC e)/\langle e,h,h^2 \rangle$. In particular, we have $h^1(Y,T_Y) = \max(0,\rk(\JJ) - 3)$.
\item The connected component of the automorphism group $\Aut^0(Y)$ of $Y$ and the dimension of local deformations $h^1(Y,T_Y)$ are given in Table \ref{table_deformations}
  \end{enumerate}
  \FloatBarrier
  \begin{table}[ht]
   \begin{tabular}{ccccc}
      Jordan algebra & $\XX$ & $X$ & $\Aut^0(Y)$ & $h^1(Y,T_Y)$ \\
      \hline
$\cO_{2,n}$ & $Q_{n-2}$ & $Q_{n-3}$ & ${\rm PO}_{n-2}(\CC)$ & $0$ \\
$\cS_{n}$ & $v_2(\PP^{n-1})$ & $v_2(Q_{n-2})$ & $1$ & $n - 3$ \\
$\cM_n$ & $\PP^{n-1} \times (\PP^{n-1})^\vee$ & ${\rm Fl}(1,n-1;n)$ & $T$ & $n - 3$ \\
$\cA_n$ & $\Gr(2,2n)$ & $\IGr(2,2n)$ & $(\SL_2(\CC))^n$ & $n - 3$ \\
$\cE_3$ & $E_6/P_6$ & $F_4/P_4$ & ${\rm SO}_8(\CC)$ & $0$ \\
   \end{tabular}
\medskip
\centering
\caption{\label{table_deformations} Automorphisms and deformations}
  \end{table}
  \FloatBarrier
\end{proposition}

\begin{proof}
  Since $Y$ is obtained as complete intersection of two hyperplane sections in $\XX$, we have the following exact sequence
  $$0 \to \HHH^0(Y,T_Y) \to \HHH^0(Y,T_\XX\vert_Y) \to \HHH^0(Y,\cO_Y(1))^2 \to \HHH^1(Y,T_Y) \to \HHH^1(Y,T_\XX\vert_Y).$$
  Furthermore, via Borel-Weil-Bott, we may check the following equalities: $\HHH^1(Y,T_\XX\vert_Y) = 0$ and $\HHH^0(Y,T_\XX\vert_Y) = \HHH^0(\XX,T_\XX) = [\gg,\gg] = \fg \oplus J$. Furthemore, we have an identification $\HHH^0(Y,\cO_Y(1))^2 = (J/\CC h)^2$ and the map $\HHH^0(Y,T_\XX\vert_Y) \to \HHH^0(Y,\cO_Y(1))^2$ identifies with $(\ad_e,\ad_h) : \fg \oplus J \to (J/\CC h)^2$. The spaces $\HHH^0(Y,T_Y)$ and $\HHH^1(Y,T_Y)$ are the kernel and cokernel of this map.

  Since $\ad_e$ vanishes on $\fg$ (recall that $G$ fixes $e$), we get the following commutative diagram
  $$\xymatrix{ 0 \ar[r] & \Ker\ad_h \ar[r] \ar[d] & \fg \ar[rr]^-{\ad_h} \ar[d] && J/\CC h \ar[r] \ar[d] & \hh_{-1}/\CC h \ar[d] \ar[r] & 0 \\
    0 \ar[r] & \HHH^0(Y,T_Y) \ar[r] \ar[d] & \fg \oplus J \ar[rr]^-{(\ad_e,\ad_h)} \ar[d] && (J/\CC h)^2 \ar[r] \ar[d] & \HHH^1(Y,T_Y) \ar[r] & 0 \\
  0 \ar[r] & \CC [f,h] \ar[r] & J \ar[rr]^-{\ad_e} && J/\CC h \ar[r] & 0 & .\\}$$ 
The exact sequence in the last row comes from the fact that $[e,[f,h]]=-[h,[e,f]]-[f,[h,e]]=-[h,\alpha^\vee]=2h$. Note that $\hh_{-1}/\CC h = (\hh_{-1} \oplus \CC e)/\langle e,h \rangle$. Applying The Snake Lemma, we get an exact sequence
  $$0 \to \Ker\ad_h \to \HHH^0(Y,T_Y) \to \CC [f,h] \stackrel{\ad_h}{\longrightarrow} (\hh_{-1} \oplus \CC e)/\langle e,h \rangle \to \HHH^1(Y,T_Y) \to 0.$$
  Now $\ad_h([f,h]) = 2  h \bullet h = 2 h^2$ for the product in $\JJ$. Since $h$ is general in $\hh_{-1}$, the elements $e$, $h$ and $h^2$ are linearly independent so that the map $\ad_h$ in the last exact sequence is injective and we get $\HHH^0(Y,T_Y) = \Ker\ad_h$ and $\HHH^1(Y,T_Y) = (\hh_{-1} \oplus \CC e)/\langle e,h,h^2 \rangle$. To compute the automorphism group we can use the fact that it is the adjoint group of the reductive group $L$ generated by compact roots in $\GG$ (see Lemma \ref{lem-ker-coker}). 
\end{proof}

\begin{remark}
  For $\JJ$ a Jordan algebra of rank $n$, there is a natural degree $n$ polynomial defined over $\JJ$ and called the norm $N$ of $\JJ$ (this is simply the determinant in cases $\cS_n$ and $\cM_n$ and the Pfaffian for $\cA_n$). The group $\GG$ is the subgroup of $\GL(\JJ)$ stabilising $N$.

  We recover the fact that the dimension of the deformation space of $Y$ is always $n-3$ as the dimension of the moduli space of $n$ points over $\PP^1$. It turns out that to any smooth $Y$ one can associate a point in such moduli space. Indeed, $Y$ is the intersection of two hyperplane sections inside $\XX$ and therefore it defines a generic line inside $\PP(\JJ)$. The $n$ points on $\PP^1$ are given by the intersection of this line with the set $\{ N=0 \}\subset \PP(\JJ)$.
\end{remark}

\section{Geometry of hyperplane sections}

Let $X\subset \PP(V)$ be adjoint or quasi-minuscule. Recall that we have $V = \mathfrak{g} = \nabla_\Theta$ or $V = \nabla_\theta$, the highest weight representations of highest weights $\Theta$ and $\theta$, the highest root and the highest short root respectively. Recall that $\Delta_l$ and $\Delta_s$ denote the set of long roots and short roots respectively. We have $\Pi(V) = \Delta \cup \{ 0\} = \Delta_l \cup \Delta_s \cup \{ 0\}$ or $\Pi(V) = \Delta_s \cup \{ 0 \}$ if $V = \nabla_\Theta$ or $V = \nabla_\theta$. Let $\aleph$ be the $W$-orbit of the highest weight in $V$. We have $\aleph = \Delta_l$ or $\aleph = \Delta_s$. Recall that for $\alpha \in \Pi(V) \setminus \{0 \}$, we have $\dim V_\alpha = 1$. In that case, we choose $e_\alpha \in V_\alpha \setminus \{0\}$ and set $x_\alpha = [e_\alpha]$.

Let $[h] \in \PP(V^\vee)$ be general and $T$-invariant. Recall that this implies that $h \in V_0$ has weight $0$ \emph{i.e.} that $h$ is $T$-invariant. Let $Y = Y_h$ be the corresponding hyperplane section, which is general and therefore smooth. The variety $Y$ is $T$-stable and therefore inherits the $T$-action from $X$. The $T$-fixed point of $X$ are given by $X^T = \{ x_\alpha \ | \ \alpha \in \aleph \}$.

\begin{lemma}
We have $Y^T = X^T$.
\end{lemma}

\begin{proof}
We only need to prove the inclusion $X^T \subset Y^T$. Let $\alpha \in \aleph$ and $e_\alpha \in V_\alpha$. For $t \in T$, we have $\langle h,e_\alpha \rangle = \langle t \cdot h,e_\alpha \rangle = \langle h,t \cdot e_\alpha \rangle = \alpha(t) \langle h,e_\alpha \rangle$. Since $\alpha \neq 0$, there exists $t \in T$ with $\alpha(t) \neq 1$ and we get $\langle h,e_\alpha \rangle = 0$, therefore $x_\alpha = [e_\alpha] \in Y$.
\end{proof}

\subsection{$T$-stable curves}
\label{secinvariantcurves}

Let $Z = G/P \subset \PP(V)$ be a projective rational $G$-homogeneous space embedded in its minimal embedding $\PP(V)$. We recall few facts on $T$-stable curves in $Z$ before specialising to our situation (see \cite{Brion} or \cite{fulton-woodward} for more details).

Let $Z^T = \{ z_\alpha \ | \ \alpha \in \aleph \}$ be the finite set of $T$-fixed points in $Z$. Any $T$-invariant curve in $Z$ contains exactly two distinct $T$-fixed points. Furthermore, a pair of distinct fixed points $(z_\alpha,z_\beta)$ is connected by a $T$-stable curve if and only if there exists a root $\gamma$ such that $s_\gamma(z_\alpha) = z_\beta$. For such a root $\gamma$, denote by $\SL_2(\gamma) = \langle \exp(\fg_\gamma) , \exp(\fg_{-\gamma}) \rangle$ the subgroup generated by $\exp(\fg_\gamma)$ and $\exp(\fg_{-\gamma})$ (this subgroup is isomorphic to $\SL_2(\CC)$ or $\PGL_2(\CC)$). Then the $T$-stable curve passing through $(z_\alpha,z_\beta)$ is $C_{\alpha,\beta} = \SL_2(\gamma) \cdot z_\alpha = \SL_2(\gamma) \cdot z_\beta$. It is isomorphic to $\PP^1$. Note that the root $\gamma$ above is determined up to sign.

Recall that $A_1(Z)$ the group of one-cycles on $Z = G/P$ is isomorphic to the quotient $Q^\vee/Q^\vee_P$ where $Q^\vee$ is the coroot lattice and $Q^\vee_P$ the coroot lattice of $P$. Let $\alpha,\beta \in \aleph$ and $\gamma \in \Delta$ as above and let $w \in W$ such that $w(z_\alpha) = P/P \in Z$. The following result is proved in \cite[Lemma 3.4]{fulton-woodward}.

\begin{lemma}
  \label{lem-classe-courbe}
  We have $[C_{\alpha,\beta}] = [\delta^\vee]$, where $\delta$ is the unique positive root in  $\{\pm w(\gamma)\}$.
\end{lemma}

Recall that the Picard group $\Pic(Z)$ is isomorphic to $\Lambda^P = \{ \lambda \in \Lambda \ | \ \langle \delta^\vee , \lambda \rangle = 0 \textrm{ for } \delta^\vee \in Q_P^\vee\}$, where $\Lambda$ is the set of characters of $T$. The isomorphism $\Lambda^P \to \Pic(Z), \lambda \mapsto L_\lambda$ is defined as follows: $L_\lambda$ is the equivariant line bundle whose fiber at $P/P$, the $B$-fixed point, has weight $\lambda$. We have the following result.

\begin{lemma}[Lemma 3.2 \cite{fulton-woodward}]
  We have $\langle{L_\lambda,C_{\alpha,\beta}} \rangle = \langle \delta^\vee , \lambda \rangle$ with $\delta$ as in Lemma \ref{lem-classe-courbe}.
\end{lemma}

Assume now that $Z = X \subset \PP(V)$ is an adjoint or quasi-minuscule variety and is minimally embedded. Let $L_\varpi = \cO_{\PP(V)}(1)$. Then $\varpi = \Theta$ or $\varpi = \theta$ and $\aleph = \Delta_l$ or $\aleph = \Delta_s$ according to the case adjoint or quasi-minuscule. Let $V_0$ be the weight space of weight $0$ in $V$ and let $E$ be its $T$-stable complement.

\begin{lemma}
  \label{lemm-vanish}
Any linear form $h \in V_0^\vee$ vanishes on $E$.
\end{lemma}

\begin{proof}
Let $\alpha$ be a weight of $E$ and $t \in T$. We have $\langle h,v_\alpha \rangle = \langle t \cdot h,v_\alpha \rangle = \langle h,t \cdot v_\alpha \rangle = \alpha(t) \langle h,v_\alpha \rangle$. Since $\alpha \neq 0$, there exists $t \in T$ with $\alpha(t) \neq 1$ and $\langle h,v_\alpha \rangle = 0$. The result follows.
\end{proof}

Let $\alpha,\beta \in \aleph$ such that there exists $\gamma \in \Delta$ with $s_\gamma(x_\alpha) = x_\beta$. We have $\beta = s_\gamma(\alpha)$.

\begin{lemma}
The degree of the curve $C_{\alpha,\beta}$ in $\PP(V)$ is equal to $|\langle \gamma^\vee , \alpha \rangle|$.
\end{lemma}

\begin{proof}
The $B$-fixed point in $X$ is $x_\varpi$. Let $w \in W$ with $w(\alpha) = \varpi$. The degree is
  $\langle \delta^\vee , \varpi \rangle = |\langle w(\gamma^\vee) , \varpi \rangle| = |\langle \gamma^\vee , w^{-1}(\varpi) \rangle| = |\langle \gamma^\vee , \alpha \rangle|$.
\end{proof}

In particular, the degree of a $T$-stable curve in $X \subset \PP(V)$ is equal to $1$, $2$ or $3$ and there are degree $3$ curves only in the adjoint variety of type $G_2$.
For later purposes, we define two types of $T$-stable curves in $X$.

\begin{definition}
  \label{def-t-stable}
  Let $\alpha,\beta \in \aleph$, the curve $C_{\alpha,\beta}$ and $\gamma \in \Delta$ as above.
  \begin{enumerate}
  \item The curve $C_{\alpha,\beta}$ is a \textbf{root-conic} if $\beta = -\alpha$. In that case $\gamma = \pm\alpha$ and the curve has degree $2$.
  \item The curve $C_{\alpha,\beta}$ is called \textbf{plain} if $\beta \neq -\alpha$.
  \item \textbf{The weight} $\wt(C_{\alpha,\beta})$
    of $C_{\alpha,\beta}$ is  $\wt(C_{\alpha,\beta}) = \gamma$.
  \end{enumerate}
\end{definition}

Let $\rho : \fg \to \fgl(V)$ be the representation of the Lie algebra on $V$.

\begin{lemma}
  Let $\alpha,\beta \in \aleph$ such that there exists a $T$-stable curve $C_{\alpha,\beta}$ joining $x_\alpha$ and $x_\beta$. Let $\gamma \in \Delta$ such that $s_\gamma(\alpha) = \beta$.
  \begin{enumerate}
  \item If $C_{\alpha,\beta}$ is plain, then $C_{\alpha,\beta} \subset \PP(E) \subset \PP(V)$.
  \item If $C_{\alpha,\beta}$ is a root conic, then $C_{\alpha,\beta}$ is a conic in  $\PP(\langle v_\alpha , \rho(e_{-\alpha})(v_\alpha),v_{-\alpha} \rangle) \subset \PP(V)$.
  \end{enumerate}
\end{lemma}

\begin{proof}
Assume that $\langle \gamma^\vee , \alpha \rangle > 0$. Then $x_\alpha$ is  stable under the action of $\exp(\fg_\gamma)$ and $C_{\alpha,\beta}$ is the closure of $\exp(\fg_{-\gamma}) \cdot x_\alpha$. In particular $C_{\alpha,\beta}$ is contained in $\PP(\langle \exp(\fg_{-\gamma}) \cdot v_\alpha \rangle)$.

In case 1., recall that we have the vanishing $\ad(e_{\gamma})^{|\langle \gamma^\vee , \alpha \rangle| + 1}(e_{\alpha})=0$. We therefore have the inclusion $\exp(\CC e_{-\gamma}) \cdot v_\alpha \in \langle v_\alpha,\rho(e_{-\gamma})(v_\alpha),\cdots,\rho(e_{-\gamma})^{|\langle \gamma^\vee , \alpha \rangle |}(v_\alpha) \rangle$. Since all these vectors have non-zero weights, the first assertion follows.

In case 2., we have $\gamma = \alpha$. Recall that we have the vanishing $\ad(e_{-\alpha})^3(e_\alpha) = 0$. We therefore have the inclusion $\exp(\CC e_{-\gamma}) \cdot v_\alpha \subset \langle v_\alpha,\rho(e_{-\alpha})(v_\alpha),v_{-\alpha} \rangle$. 
\end{proof}

We now describe the intersection of $T$-stable curves with $Y \subset X$ a general $T$-stable hyperplane section.

\begin{proposition}
\label{lemTeqcurves}
The plain $T$-stable curves are contained in $Y$ while the root-conics are not contained in $Y$.
\end{proposition}

\begin{proof}
  Let $h \in V_0^\vee$ general, then $h$ vanishes on $E$ by Lemma \ref{lemm-vanish}. The result for plain $T$-stable curves follows from this.

  For root-conics, recall that $h \in V_0^\vee$ is general. In particular, we may assume that $h$ does not vanish on any vector $\rho(e_{\-\alpha}) \cdot v_\alpha \in V_0$. Therefore $Y$ intersects the plane $\PP(\langle v_\alpha , \rho(e_{-\alpha})(v_\alpha),v_{-\alpha} \rangle) \subset \PP(V)$ along a line, while the root conic $C_{\alpha,-\alpha}$ is an irreducible conic in that plane. This proves the result for root-conics.
\end{proof}

\subsection{Bia\l ynicki-Birula decomposition}

We recall few facts on the Bia\l ynicki-Birula decomposition \cite{BBdec} before applying them to our situation. Let $Z$ be a smooth projective variety acted on by a one dimensional torus $\top$ such that $Z^\top = \{ z_\alpha \ | \ \alpha \in \aleph \}$ is finite. We identify $\top$ with $\CC^*$ and embed it in $\PP^1$. For $z \in Z^\top$ define $\Omega_{\alpha,Z}^+ = \{ z \in Z \ | \ \lim_{t \to 0}t \cdot z = z_\alpha\}$ and $\Omega_{\alpha,Z}^- = \{ z \in Z \ | \ \lim_{t \to \infty}t \cdot z = z_\alpha\}$. Let $Z_{\alpha}^\pm$ be the Zariski closure in $Z$ of  $\Omega_{\alpha,Z}^\pm$ and set $\sigma_{\alpha,Z}^\pm = [Z_{\alpha}^\pm] \in \HHH^*(Z,\ZZ)$.

\begin{theorem}[Bia\l ynicki-Birula decomposition \cite{BBdec}]
We have the following results.
  \begin{enumerate}
  \item We have two cellular decompositions $Z = \coprod_{\alpha\in\aleph} \Omega_{\alpha,Z}^+ = \coprod_{\alpha\in\aleph} \Omega_{\alpha,Z}^-$.
  \item The cells $(\Omega_{\alpha,Z}^\pm)_{\alpha \in \aleph}$ are affine spaces and their closures $(Z_{\alpha}^\pm)_{\alpha \in \aleph}$ are irreducible.
  \item The classes $(\sigma_{\alpha,Z}^+)_{\alpha \in \aleph}$ and $(\sigma_{\alpha,Z}^-)_{\alpha \in \aleph}$ form two basis of $H^*(Z,\ZZ)$.
  \end{enumerate}
\end{theorem}

\begin{definition}
  The Bia\l ynicki-Birula decomposition satisfies the \textbf{inclusion condition} if the following implications are satisfied: ($z_\alpha \in Z_\beta^+$ $\Rightarrow$ $Z_\alpha^+ \subset Z_\beta^+$) and ($z_\alpha \in Z_\beta^-$ $\Rightarrow$ $Z_\alpha^- \subset Z_\beta^-$).
\end{definition}

\begin{lemma}
If the Bia\l ynicki-Birula decomposition satisfies the inclusion condition, then the families $(\sigma_{\alpha,Z}^+)_{\alpha \in \aleph}$ and $(\sigma_{\alpha,Z}^-)_{\alpha \in \aleph}$ form two Poincar\'e dual basis of $H^*(Z,\ZZ)$.
\end{lemma}

\begin{proof}
  For $\gamma \in \aleph$, by \cite{BBdec}, we have $\dim Z_\gamma^+ + \dim Z_\gamma^- = \dim Z$. Furthermore, we have $Z_\gamma^+ \cap Z_\gamma^- = \{ z_\gamma \}$ and the intersection is transverse.

  Let $\alpha$ and $\beta$ in $\aleph$ such that $\dim Z_\alpha^+ + \dim Z_\beta^- = \dim Z$. If $Z_\alpha^+ \cap Z_\beta^-$ is non empty, then it contains $z_\gamma$ for some $\gamma \in \aleph$. We get $Z_\gamma^+ \subset Z_\alpha^+$ and $Z_\gamma^- \subset Z_\beta^-$. In particular $\dim Z_\gamma^+ \leq \dim Z_\alpha^+$ and $\dim Z_\gamma^- \leq \dim Z_\beta^-$. This implies $\dim Z = \dim Z_\gamma^+ + \dim Z_\gamma^- \leq \dim Z_\alpha^+ + \dim Z_\beta^- = \dim Z$. We must have equality in all previous inequalities. This implies that $Z_\alpha^+ = Z_\gamma^+$ and $Z_\beta^- = Z_\gamma^-$ thus $\alpha = \gamma = \beta$ and $Z_\alpha^+ \cap Z_\beta^- = \{z_\gamma\}$ is transverse. 
\end{proof}

To simplify notation we set $\Omega_{\alpha,Z} = \Omega_{\alpha,Z}^+$, $Z_{\alpha} = Z_{\alpha}^+$ and $\sigma_{\alpha,Z} = \sigma_{\alpha,Z}^+$.

\begin{remark}
  Note that in general a Bia\l ynicki-Birula decomposition does not satisfy the inclusion condition. We will see that the general $T$-stable hyperplane section $Y$ of an adjoint or quasi-minuscule variety $X$ gives an example of such a decomposition.

  We give two examples where the inclusion condition is satisfied.
  \begin{enumerate}
  \item If $Z$ is a rational projective homogeneous space under the action of a reductive group $G$ and $\top$ is a general one-parameter subgroup in $G$, then the Bia\l ynicki-Birula decomposition is given by the Schubert cells and varieties and their opposites. They satisfy the inclusion condition. This is mainly because the decomposition identifies with the Bruhat decomposition (see \cite{BB2dechom}[Book II, example $4.2$]) and is given by orbits of Borel subgroups, see below for more details.
  \item If $Z$ is a (non-homogeneous) smooth $G$-horospherical variety of Picard rank one for some reductive group $G$ and $\top$ is a general one parameter subgroup of $G$, then it is easy to check that the decomposition satisfies the inclusion condition (again because the cells are described using Borel subgroups). We recover the Poincar\'e dual basis description in \cite[Proposition 1.10]{horospherical}.
  \end{enumerate}
\end{remark}

Assume that $Z = G/P$ is a projective rational $G$-homogeneous space. The fixed point set $Z^T = \{ z_\alpha \ | \ \alpha \in \aleph \}$ is finite. Let $\top \subset T$ be a regular anti-dominant one parameter subgoup. In this case, the Bia\l inicky-Birula decomposition coincides with the Bruhat decomposition.

\begin{lemma}
We have $\Omega_{\alpha,Z} = B^- \cdot z_\alpha$, $\Omega_{\alpha,Z}^- = B \cdot x_\alpha$ and $Z_\alpha = \overline{B^- \cdot x_\alpha}$, $Z_\alpha^- = \overline{B \cdot x_\alpha}$.
\end{lemma}

\begin{proof}
  Let $\alpha \in \aleph$. Note that for $b \in B^-$, we have $\lim_{t \to 0}(t b t^{-1}) \in T$.
  From this we obtain
  $\lim_{t \to 0}(t \cdot b \cdot z_\alpha) = \lim_{t \to 0}(tbt^{-1} t \cdot z_\alpha) = \lim_{t \to 0}(t bt^{-1}) \cdot z_\alpha = z_\alpha$. In particular $B^- \cdot z_\alpha$ is contained in the Bia\l ynicki-Birula cell $\Omega_{\alpha,Z}$.
 
  Since cells are irreducible and contain one $\top$-fixed point, we get the equality $\Omega_{\alpha,Z} = B^- \cdot z_\alpha$. The last assertions follow by taking closures.
\end{proof}

\begin{corollary}
  For a projective rational homogeneous space $Z$, the Bia\l ynicki-Birula decomposition satisfies the inclusion condition.
\end{corollary}

\begin{proof}
If $z_\alpha \in Z_\beta$, since the cells and their closures are $B^-$-stable, we get $\Omega_\alpha = B^- \cdot z_\alpha \subset Z_\beta$ and $Z_\alpha \subset Z_\beta$. The same argument works for opposite cells.
\end{proof}

We therefore get Poincar\'e dual basis $(\sigma_{\alpha,Z})_{\alpha \in \aleph}$ and $(\sigma_{\alpha,Z}^-)_{\alpha \in \aleph}$ for $H^*(Z,\ZZ)$. Actually for a projective rational homogeneous space $Z$, these two basis coincide so that we get a Poincar\'e self-dual basis. Recall that $W$ acts on $\aleph$ via $w \cdot z_\alpha = z_{w(\alpha)}$ for $w \in W$ and $\alpha \in \aleph$.

\begin{lemma}
  \label{lem-dual-Z}
If $w_0 \in W$ is the longest element, we have $\sigma_{\alpha,Z}^\vee = \sigma_{\alpha,Z}^- = \sigma_{w_0(\alpha),Z}$.
\end{lemma}

\begin{proof}
We have $B = w_0B^-w_0^{-1}$ and $w_0 \cdot x_\alpha = x_{w_0(\alpha)}$. We therefore have $\Omega_{\alpha,Z}^- = B \cdot x_\alpha = w_0 B^- w_0 \cdot x_\alpha = w_0 B^- \cdot x_{w_0(\alpha)}$ and $Z_\alpha^- =  w_0 Z_{w_0(\alpha)}$. Since $w_0 \in G^\circ$ we get $\sigma_{\alpha,Z}^\vee = \sigma_{\alpha,Z}^- = \sigma_{w_0(\alpha),Z}$.
\end{proof}

For $Z = X$ an adjoint or quasi-minuscule variety, the basis is indexed by long or short roots: $\aleph = \Delta_l$ or $\aleph = \Delta_s$. This partially recovers results from \cite[Proposition 2.9]{adjoint}. Recall that $\varpi = \Theta$ or $\theta$ and let $\rho$ be the half sum of positive roots. Recall the order $\leq$ on roots defined by $\alpha \leq \beta$ $\Leftrightarrow$ ($\beta - \alpha$ is a sum of simple roots), and the support of $\beta=\sum_{\alpha_i\in\Phi}b_i \alpha_i$ defined by $\supp(\beta)=\{\alpha_i\in \Phi \mid b_i\neq 0 \}$.

\begin{proposition}[Proposition 2.9 and Lemma 2.6, \cite{adjoint}]
  \label{ordre-b-x}
Let $X$ be adjoint or quasi-minuscule.
\begin{enumerate}
\item If $X$ is adjoint we have $\dim X = 2 \langle \rho , \Theta^\vee \rangle - 1$ and if $X$ is quasi-minuscule its dimension is the same as the one of the adjoint variety of the dual Langlands group; moreover $c_1(X) = \langle \rho , \varpi^\vee \rangle$.
\item We have
$X_\alpha \subset X_\beta \Leftrightarrow \left\{
  \begin{array}{ll}
    \alpha \leq \beta & \textrm{for $\alpha$ and $\beta$ of the same sign} \\
    \supp(\alpha) \cup \supp(\beta) \textrm{ is connected} & \textrm{for $\alpha < 0$ and $\beta > 0$.} \\
  \end{array}
  \right.$
\end{enumerate}
\end{proposition}

From the Chevalley formula proved in \cite[Theorem 3]{adjoint}, it is easy to check the following result.

\begin{proposition}
\label{remTcurvesX}
Let $\alpha,\beta \in \aleph$ with $\dim X_\beta = \dim X_\alpha - 1$. Then $x_\beta \in X_\alpha$ if and only if there exists a root $\gamma$ with $\beta = s_\gamma(\alpha)$ and one of the following occurs:
\begin{enumerate}
\item the root $\gamma$ is simple or 
\item the roots $\alpha$ and $-\beta$ are simple and $\gamma = \alpha - \beta$.
\end{enumerate}
\end{proposition}

We now consider the case $Z = Y$. The same one-dimensional torus $\top \subset T$ induces a decomposition of $Y$. We will refer to $\Omega_\alpha = \Omega_{\alpha,Y}$ and $\Omega_\alpha^- = \Omega_{\alpha,Y}^-$ as the \emph{Schubert} cells of $Y$ and to $Y_\alpha$ and $Y_\alpha^-$ as the \emph{Schubert} varieties of $Y$, in analogy with the homogeneous case.

lls of $Y$ and to $Y_\alpha$ and $Y_\alpha^-$ as the \emph{Schubert} varieties of $Y$, in analogy with the homogeneous case.

\begin{remark}
  Note that we have the equalities $\Omega_\alpha = \Omega_{\alpha,X} \cap Y$ and $\Omega_\alpha^- = \Omega_{\alpha,X}^- \cap Y$. In particular we have $Y_\alpha = \overline{\Omega_\alpha} = \overline{\Omega_{\alpha,X} \cap Y} = \overline{B^- \cdot x_\alpha \cap Y} \subset \PP(V)$.

  However we have $Y_\alpha \neq X_\alpha \cap Y$ in general! The equality fails in general because, if $x_\alpha \in X_\beta$, then 
\[
X_\alpha \cap Y=\overline{B^- \cdot x_\alpha}\cap Y \subset \overline{B^- \cdot x_\beta}\cap Y = X_\beta\cap Y,
\]
while we will see that in certain cases it is not true that $Y_\alpha \subset Y_\beta$ for $x_\alpha \in Y_\beta$ (the decomposition does not satisfy the inclusion condition). Indeed, it may very well happen that $x_\alpha \in Y_\beta$ and $\dim(Y_\alpha) = \dim(Y_\beta)$. 
\end{remark}

We first compare the Schubert varieties in $X$ and in $Y$.

\begin{proposition}
  \label{prop-dim}
  Let $\alpha \in \aleph$, we have the following alternative:
  \begin{enumerate}
  \item If $\alpha > 0$ \emph{i.e.} $\dim(X_\alpha) > \dim(X)/2$, then $X_\alpha\not\subset Y$ and $\dim Y_\alpha = \dim X_\alpha - 1$,
  \item If $\alpha < 0$ \emph{i.e.} $\dim(X_\alpha) < \dim(X)/2$, then $X_\alpha \subset Y$ and $\dim Y_\alpha = \dim X_\alpha$.
  \end{enumerate}
\end{proposition}

\begin{proof}
  For $\alpha >0$, the root-conic $\SL_2(\alpha) \cdot x_\alpha$ is contained in $X_\alpha$ but not inside $Y$, which implies that $X_\alpha \not \subset Y$. For $\alpha < 0$, since $X_\alpha = \overline{B^- \cdot x_{\alpha}}$ is contained in $\PP(\oplus_{\beta < 0} V_\beta) \subset \PP(E)$, we deduce that $X_\alpha \subset Y$.
\end{proof}

We are now interested in inclusions of $T$-fixed points in the Schubert varieties $(Y_\alpha)_{\alpha \in \aleph}$. As $Y$ is a hyperplane section inside $X$, by Lefschetz hyperplane theorem most of the cohomology of $Y$ is determined by the cohomology of $X$ and we will be mainly interested in the middle cohomology of $Y$, therefore in Schubert varieties $Y_{\pm \alpha}$ for $\alpha$ simple. We set $\Phi_\aleph = \Phi \cap \aleph$.

Since $V$ is self-dual as a $G$-representation, there exists a non-degenerate $G$-invariant scalar product $(\ ,\ )$ identifying $V$ with $V^\vee$. Set $h^\perp = \{ v \in V \ | \ \langle h , v \rangle = 0 \} =  \{ v \in V \ | \ (h , v) = 0 \}$. Since $(\ ,\ )$ is $G$-invariant, we have $(V_\alpha,V_\beta) = 0$ for $\alpha + \beta \neq 0$ with $\alpha,\beta \in \aleph$. Furthermore for $\alpha \in \aleph$ the bilinear form $(\ ,\ )$ realises a duality between $V_\alpha$ and $V_{-\alpha}$. In particular we have $(v_\alpha,v_{-\alpha}) \neq 0$.

\begin{lemma}
  \label{lem-crochet-rs}
  Let $\alpha \in \Phi_\aleph$ and $\beta \in \Phi$ be simple roots. If $\beta \neq \alpha$, then $\rho(\fg_{-\beta})V_\alpha = 0$.
\end{lemma}

\begin{proof}
Follows from $\rho(\fg_{-\beta})V_\alpha \subset V_{\alpha - \beta}$. and the fact that non-zero weights of $V$ are roots.
\end{proof}

Endow the Lie algebra $\mathfrak{g}$ with the $\ZZ$-grading induced by the height: $\mathfrak{g}=\bigoplus_{i\in\ZZ}\mathfrak{g}_i$, where $\fg_\alpha \subset \fg_{\haut(\alpha)}$ and $\fh = \fg_0$.
Let $f \in \mathfrak{b}^-$ be an element in the Borel subalgebra of $\mathfrak{g}$ corresponding to $B^-$ and write $f = \sum_{i\leq 0} f_i$ with $f_i \in \fg_i$.

\begin{lemma}
\label{lem1exp}
Let $f = \sum_{i\leq 0} f_i \in \mathfrak{b}^-$ and let $x_\alpha = [v_\alpha] \in X^T$ with $\alpha$ simple. Then $\exp(f) \cdot x_\alpha$ belongs to $Y$ if and only if $f_{-1} \in \fg_\alpha^\perp$ or equivalently $\rho(f_{-1})v_\alpha = 0$.
\end{lemma}

\begin{proof}
  We have the inclusion $\exp(f) \cdot x_\alpha \in Y$ if and only if the vanishing $\langle h , \exp(f) \cdot v_\alpha \rangle = 0$ holds. We compute this evaluation (recall that $h$ vanishes on $E = \oplus_{\gamma \in \aleph} V_\gamma$ and that $\rho(f_0)(V_0) = 0$):
  $$
  \arraycolsep=1.4pt\def\arraystretch{2.2}
  \begin{array}{ll}
    \langle h , \exp(f) \cdot v_\alpha \rangle & \displaystyle{ = \left\langle h, \sum_{k\geq 0} \frac{1}{k!} \rho(f)^k(v_\alpha) \right\rangle} \\
& \displaystyle{= \left\langle h , \sum_{k\geq 1}\frac{1}{k!} \rho(f_{-1}) \rho(f_0)^{k-1} v_\alpha \right\rangle} \\
& = \displaystyle{ \left( \sum_{k\geq 1} \frac{\langle f_0 , \alpha \rangle^{k-1}}{k!} \right) \langle h , \rho(f_{-1}) v_\alpha \rangle } \\
& \displaystyle{ = \frac{e^{\langle f_0 , \alpha \rangle} - 1}{\langle f_0 , \alpha \rangle} \langle h , \rho(f_{-1}) v_\alpha \rangle.} \\
  \end{array}$$
  Since the scalar $\frac{e^\lambda - 1}{\lambda}$ never vanishes, the vanishing of this term is equivalent to the vanishing of $\langle h , \rho(f_{-1}) v_\alpha \rangle$. By Lemma \ref{lem-crochet-rs} we have that the space $\rho(\fg_{-1})(V_\alpha) = \rho(\fg_{-\alpha})(V_\alpha)$ is one-dimensional. Since $h$ is general, it only vanishes on the zero-vector in this space. Since $\rho(f_{-1})(v_\alpha) \in \rho(\fg_{-1})(V_\alpha)$ and $\langle h , \rho(f_{-1}) v_\alpha \rangle = 0$, we get $\rho(f_{-1})(v_\alpha) = 0$. Finally note that $\rho(\fg_{-1})$ acts on $V_\alpha$ via $\rho(\fg_{-\alpha})$ and that this last action is non trivial. This proves that the condition $\rho(f_{-1})v_\alpha = 0$ is equivalent to $f_{-1} \in \fg_\alpha^\perp$.
\end{proof}

\begin{proposition}
\label{lemincl-alphaalpha}
Let $\alpha \in \Phi_\aleph$. Then $x_{-\alpha} \notin Y_\alpha$.
\end{proposition}

\begin{proof}
  Recall that we have $Y_\alpha = \overline{B^- \cdot x_\alpha \cap Y} = \overline{B^- \cdot x_\alpha \cap \PP(h^\perp)}$. Recall also that $(v_\alpha,v_{-\alpha}) \neq 0$. We will prove that $B^- \cdot v_\alpha \cap h^\perp \subset v_\alpha^\perp$, this implies that $x_{-\alpha} = [v_{-\alpha}] \notin Y_\alpha$.
  
  Let $f = \sum_{i\leq 0}f_i \in \mathfrak{b}^-$. To compute $(\exp(f) \cdot v_\alpha , v_\alpha)$, recall that $v_{-\alpha}$ is dual to $v_\alpha$ for $(\ , \ )$ so we only need to compute the coefficient of $v_{-\alpha}$ in $\exp(f) \cdot v_\alpha$. Recall that $\fg_{-2\alpha} = 0$ so that $(\rho(f_{-2})v_\alpha,v_\alpha) = 0$. Recall also that if $\exp(f) \cdot v_\alpha \in Y$, then $\rho(f_{-1})v_\alpha = 0$. Finally note that $\rho(f_0)(V_0) = 0$. Altogether this gives, for $\exp(f) \cdot v_\alpha \in Y$, the vanishing $(\exp(f) \cdot v_\alpha , v_\alpha) = 0$ and the result.
\end{proof}

\begin{proposition}
\label{lemincl2}
Let $\alpha \in \Phi_\aleph$ and $\beta \in \Phi$ be simple roots with $\langle \beta^\vee , \alpha \rangle < 0$.
\begin{enumerate}
\item If $\beta \in \aleph$, then $x_{-s_\beta(\alpha)} \not\in Y_\alpha$.
\item If $\beta \not \in \aleph$, then $x_{-s_\beta(\alpha)} \in Y_\alpha$ and there is a positive root $\gamma$ with $s_\gamma(\alpha) = -s_\beta(\alpha)$.
\end{enumerate}
\end{proposition}

\begin{proof}
  Note that the condition $\beta \in \aleph$ is equivalent to saying that $\alpha$ and $\beta$ have the same length. 
  
  Assume first that $\beta \not\in \aleph$ and set $m = \langle \beta^\vee , \alpha \rangle \langle \beta , \alpha^\vee \rangle$. In that case, there is a root $\gamma$ with $s_\gamma(\alpha) = - s_\beta(\alpha)$: for $m = 2$, set $\gamma = s_\alpha(\beta)$ and for $m = 3$, set $\gamma = s_\alpha s_\beta(\alpha)$.

  Since $s_\gamma(\alpha) = - s_\beta(\alpha)$, we have $\gamma \neq \pm \alpha$. We may therefore consider the plain $T$-stable curve $\SL_2(\gamma) \cdot x_\alpha$ which connects $x_\alpha$ with $x_{s_\gamma(\alpha)} = x_{-s_\beta(\alpha)}$ inside $Y$. Moreover, since $\alpha > 0 > - s_\beta(\alpha)$ we have $\SL_2(\gamma) \cdot x_\alpha = \overline{ (B^- \cap \SL_2(\gamma)) \cdot x_\alpha } \subset Y_\alpha$. In particular $x_{-s_\beta(\alpha)} = x_{s_\gamma(\alpha)} \in Y_\alpha$. Note that the curve $\SL_2(\gamma) \cdot x_\alpha$ is a line in all cases except if $m = 2$ and $\alpha$ is a long root in which case it is a conic.

  Assume now that $\beta \in \aleph$. In that case $\alpha$ and $\beta$ have the  same length  thus we have $\langle \beta^\vee , \alpha \rangle = -1$ and $s_\beta(\alpha) = \alpha + \beta$. We claim that $Y_\alpha \subset \PP(v_{\alpha+\beta}^\perp)$ from which we get $x_{-s_\beta(\alpha)} = x_{-(\alpha + \beta)} \not \in Y_\alpha$. To prove the claim, let $f = \sum_{i\leq 0} f_i \in \mathfrak{b}^-$ such that $\exp(f) \cdot v_\alpha \in h^\perp$. We prove that $\exp(f) \cdot v_\alpha \in v_{\alpha+\beta}^\perp$. By Lemma \ref{lem1exp} we know that $\rho(f_{-1})v_\alpha = 0$. The only possibly non zero factors of $(\exp(f) \cdot v_\alpha , v_{\alpha+\beta})$ are
  $$(\rho(f_{-3})v_\alpha , v_{\alpha+\beta}),\,\,\,\,\, (\rho(f_{-1})\rho(f_{-2})v_\alpha , v_{\alpha+\beta}),$$
$$(\rho(f_{-2}) \rho(f_{-1}) v_\alpha , v_{\alpha+\beta}),\,\,\,\,\,(\rho(f_{-1})^3 v_\alpha , v_{\alpha+\beta}).$$
  The last two terms are equal to zero because $\rho(f_{-1})v_\alpha=0$. The first term is equal to zero because $\mathfrak{g}_{2\alpha+\beta}=0$ since $2\alpha + \beta$ is not a root. We are left with computing $(\rho(f_{-1})\rho(f_{-2})v_\alpha , v_{\alpha+\beta}) = - (\rho(f_{-2})v_\alpha , \rho(f_{-1})v_{\alpha+\beta})$. Note that $\rho(f_{-1})v_{\alpha + \beta} \in V_\alpha \oplus V_\beta$ and by Lemma \ref{lem1exp}, since $f_{-1}$ has no component in $\fg_{-\alpha}$, we have $\rho(f_{-1})v_{\alpha + \beta} \in V_\alpha$. Therefore $(\rho(f_{-2})v_\alpha , \rho(f_{-1})v_{\alpha+\beta})$ is non-zero only if $\rho(f_{-2})v_\alpha$ has a non trivial component in $V_{-\alpha}$, which means that $f_{-2}$ has a non-trivial component in $\fg_{-2\alpha} = 0$. This is not possible and proves the last vanishing.
\end{proof}

\begin{remark}
\label{remusefulcurves}
Note that we have proven something more in the previous lemma: if $\alpha$ and $\beta$ are not of the same length, then the positive root $\gamma$ induces a plain $T$-invariant curve $\SL_2(\gamma)\cdot x_\alpha$ inside $Y_\alpha$ joining $x_\alpha$ and $x_{-s_\beta(\alpha)} = x_{s_{\gamma(\alpha)}}$. This curve is a line in all cases except for $\langle \beta^\vee , \alpha \rangle = 2$ in which case it is a conic. 
\end{remark}

\begin{proposition}
  \label{prop-inc-a-b}
Let $\alpha,\beta \in \Phi_\aleph$ be two simple roots with $\langle \beta^\vee , \alpha \rangle < 0$. Then there exists a plain $T$-stable line joining $x_\alpha$ and $x_{-\beta}$ in $Y_\alpha$.
\end{proposition}

\begin{proof}
Note that $\alpha$ and $\beta$ have the same length. Set $\gamma = s_\beta(\alpha) = \alpha + \beta$. Then $s_\gamma(\alpha) = - \beta$ and $SL_2(\gamma) \cdot x_\alpha = \overline{(\SL_2(\gamma) \cap B^-) \cdot x_\alpha } \subset Y_\alpha$ is a plain line joining $x_\alpha$ and $x_{-\beta}$ in $Y_\alpha$.
\end{proof}

\begin{corollary}
If $|\Phi_\aleph| \geq 2$, the Bia\l ynicki-Birula decomposition in $Y$ does not satisfy the inclusion condition.
\end{corollary}

\begin{proof}
Let $\alpha,\beta \in \Phi_\aleph$ be two simple roots with $\langle \beta^\vee , \alpha \rangle < 0$. This is possible by assumption. Then $\alpha$ and $\beta$ have the same length and $s_\beta(\alpha) = \alpha + \beta$. By Proposition \ref{prop-inc-a-b}, we have $x_{-\beta} \in Y_\alpha$. However, since $Y_{-\beta} = X_{-\beta}$ we have $x_{-s_\beta(\alpha)} \in Y_{-\beta}$ but $x_{-s_\beta(\alpha)} \not\in Y_\alpha$ by Proposition \ref{lemincl2}.
\end{proof}

\begin{remark}
We will see, using the equivariant Chevalley formula (see Section \ref{equiv-chev-form}), that the Schubert basis $(\sigma_\alpha^+)_{\alpha \in \aleph}$ and $(\sigma_\alpha^-)_{\alpha \in \aleph}$ are not dual for the Poincar\'e pairing.
\end{remark}

\begin{definition}
  \label{def-prec}
  Define a relation $\vdash$ on $\aleph$  by $\alpha \vdash \beta$ if $x_\alpha \in Y_\beta$. Define the order $\preccurlyeq$ on $\aleph$ as the transitive closure of $\vdash$.
\end{definition}

\begin{remark}
  \label{rem-order}
  We summarise few basic facts on $\vdash$ and $\preccurlyeq$.
  \begin{enumerate}
  \item The relation $\alpha \vdash \beta$ is not transitive. Indeed, for $\alpha,\beta \in \Phi_\aleph$ two simple roots with $\langle \beta^\vee , \alpha \rangle < 0$, we have $(-\beta) \vdash \alpha$  by Proposition \ref{prop-inc-a-b}, since $Y_{-\beta} = X_{-\beta}$, we have $(-\alpha - \beta) = (-s_\alpha(\beta)) \vdash (-\beta)$ but  $(-\alpha - \beta) = (-s_\beta(\alpha)) \nvdash \alpha$ by Proposition \ref{lemincl2}.
      \item We have the following properties.
  \begin{enumerate}
  \item We have the implication ($\alpha \vdash \beta$ $\Rightarrow$ $X_\alpha \subset X_\beta$). Indeed, for $\alpha \vdash \beta$, we have $x_\alpha \in Y_\beta \subset X_\beta$ and this implies $X_\alpha \subset X_\beta$ since the decomposition in $X$ satisfies the inclusion condition.
  \item The converse implication is not true. Indeed, if $\alpha \in \Phi_\aleph$ is simple, then $X_{-\alpha} \subset X_\alpha$ while $x_{-\alpha} \not\in Y_\alpha$ therefore $-\alpha \nvdash \alpha$.
  \item We also have the implication ($\alpha \preccurlyeq \beta$ $\Rightarrow$ $X_\alpha \subset X_\beta$) since $\subset$ is a transitive relation.
  \item For $\alpha \in \Phi_\aleph$ simple, we have $(-\alpha) \not\preccurlyeq \alpha$. Indeed, the inclusion $X_{-\alpha} \subset X_\alpha$ is of codimension~$1$. If $(-\alpha) \preccurlyeq \alpha$ then, since $x_{-\alpha} \notin Y_\alpha$ we must have a chain $(-\alpha) \prec \beta\prec \alpha$. This gives $X_{-\alpha} \subsetneq X_\beta \subsetneq X_\alpha$ and the inclusion $X_{-\alpha} \subset X_\alpha$ would be of codimension at least $2$, a contradiction.
  \item The converse implication of (c) 
    is not true: for $\alpha \in \Phi_\aleph$ simple, $X_{-\alpha} \subset X_\alpha$ but $(-\alpha) \not\preccurlyeq \alpha$.
  \end{enumerate}
  \end{enumerate}
\end{remark}

The order $\prec$ is fully described by the following result.

\begin{proposition}
  \label{prop-ordre}
Let $\alpha,\beta \in\aleph$, we have: $\beta \prec \alpha \Leftrightarrow (X_\beta \subsetneq X_\alpha \textrm{ and } \beta \neq - \alpha \textrm{ if $\alpha$ is simple})$.
\end{proposition}

\begin{proof}
  Note that by Remark \ref{rem-order}, we have the implication $\beta \prec \alpha \Rightarrow X_\beta \subsetneq X_\alpha$ and $\beta \neq -\alpha$ if $\alpha$ is simple. We prove the converse implication. Assume that $X_\beta \subsetneq X_\alpha$ and $\beta \neq - \alpha$ if $\alpha$ is simple.

If $\beta > 0$, then $\alpha > 0$ and it is an easy consequence of Proposition \ref{remTcurvesX} that there exists a sequence of plain curves joining $x_\alpha$ to $x_\beta$ in $X_\alpha$. All these plain curves are contained in $Y_\alpha$ thus $x_\beta \in Y_\alpha$ and $\beta \prec \alpha$. If $\alpha <0$, then $\beta < 0$ and $x_\beta \in X_\beta \subset X_\alpha = Y_\alpha$ thus $\beta \prec \alpha$.
Finally assume $\alpha > 0 > \beta$. Then by the previous cases and Proposition \ref{lemincl2}, we may assume that $\alpha$ and $-\beta$ are simple and $\alpha - \beta$ is a root. The result follows from Proposition \ref{prop-inc-a-b}
\end{proof}

\begin{remark}
  \label{rem-ordre-opp}
  We have the equivalence $\beta \prec \alpha \Leftrightarrow -\alpha \prec -\beta$.
\end{remark}

  Using the equivariant Chevalley formula, we will also describe the intersections $X_\alpha \cap Y$ in terms of the Schubert varieties $(Y_\beta)_{\beta \in \aleph}$. We conclude this section with the following results.

\begin{proposition}
  \label{prop-tang}
  Let $\alpha \in \aleph$.
  \begin{enumerate}
  \item The $T$-weights of $T_{x_{\alpha}}X$ are $\{\gamma \in \Delta \ | \ \alpha + \gamma \in \Pi(V) \}$ and have multiplicity $1$.
  \item The $T$-weights of $T_{x_{\alpha}}Y$ are $\{ \gamma \in \Delta \setminus \{-\alpha\} \ | \ \alpha + \gamma \in \Pi(V) \}$ and have multiplicity $1$.
  \item The $T$-weights of $T_{x_{\alpha}}X_\alpha$ are $\{\gamma \in \Delta^-  \ | \ \alpha + \gamma \in \Pi(V) \}$ and have multiplicity $1$.
  \item The $T$-weights of $T_{x_{\alpha}}Y_\alpha$ are $\{\gamma \in \Delta^- \setminus \{-\alpha\} \ | \ \alpha + \gamma \in \Pi(V)\}$ and have multiplicity $1$.
  \item The Schubert variety $Y_\alpha$ is smooth at $x_\alpha$.
  \end{enumerate}
\end{proposition}

\begin{proof}
  Let $\fb$ and $\fb^-$ be the Lie algebras of $B$ and $B^-$ and set $\fn = [\fb,\fb]$ and $\fn^-  =[\fb^-,\fb^-]$.
  
  1. By differentiation and since $x_\alpha$ is $T$-fixed, we get that a first order neighbourhood of $x_\alpha$ in $X$ has the form $v_\alpha + \rho(\fn^- \oplus \fn)V_\alpha = v_\alpha + \rho(\oplus_{\gamma \in \Delta}\fg_\gamma)V_\alpha = v_\alpha + \oplus_{\gamma \in \Delta \setminus \{-\alpha\}} V_{\alpha + \gamma} \oplus \rho(\fg_{-\alpha})V_\alpha$, proving the weight description. Finally note that $\rho(\fg_{-\alpha})V_\alpha \subset V_0$ is one-dimensional and that for $\beta \neq 0$, the space $V_\beta$ is also one-dimensional.

  2. Note that $\rho(\fg_{-\alpha})V_\alpha \subset V_0$ is not contained in $h^\perp$ for $h$ general. In particular, we get that a first order neighbourhood of $x_\alpha$ in $Y$ has the form $v_\alpha + \oplus_{\gamma \in \Delta \setminus \{-\alpha\}} V_{\alpha + \gamma}$.

  3. Recall the equality $X_\alpha = \overline{B^- \cdot x_\alpha}$. In particular $B^- \cdot x_\alpha$ is an open neighbourhood of $x_\alpha$ in $X_\alpha$. Differentiating, we get that a first order neighbourhood of $x_\alpha$ in $X_\alpha$ has the form $v_\alpha + \rho(\fn^-) \cdot V_\alpha$. For $\alpha > 0$ we get that $v_\alpha + \oplus_{\gamma \in \Delta^- \setminus \{-\alpha\}} V_{\alpha + \gamma} \oplus\rho(\fg_{-\alpha})(V_\alpha)$, while for $\alpha < 0$ we have $v_\alpha + \oplus_{\gamma \in \Delta \setminus \{-\alpha\}} V_{\alpha + \gamma}$.
  
  4. Recall the equality $Y_\alpha = \overline{B^- \cdot x_\alpha \cap Y}$. In particular $B^- \cdot x_\alpha \cap Y$ is an open neighbourhood of $x_\alpha$ in $Y_\alpha$. Differentiating, this gives that a first order neighbourhood of $x_\alpha$ in $Y_\alpha$ has the form $v_\alpha + \rho(\fn^-) \cdot V_\alpha \cap h^\perp = v_\alpha + \oplus_{\gamma \in \Delta \setminus \{-\alpha\}} V_{\alpha + \gamma}$.

  5. Since $x_\alpha$ is smooth in $X_\alpha$, we have $\dim T_{x_\alpha}X_\alpha = \dim X_\alpha$ and in both cases $\alpha > 0$ and $\alpha <0$ we get $\dim T_{x_\alpha}Y_\alpha = \dim Y_\alpha$, thus proving the result.
\end{proof}

Recall the definition of the weight $\wt(C)$ of a $T$-stable curve $C$ (see Definition \ref{def-t-stable}).

\begin{corollary}
  \label{coro-poids-courbes}
  Let $\alpha \in \aleph$. We have a bijection $C \mapsto \wt(C)$ between $T$-stable curves in $X$ (resp. $Y$) passing through $x_\alpha$ and $T$-weights of $T_{x_\alpha}X$ (resp. $T_{x_\alpha}Y$).
  This bijection maps $T$-stable curves in $X_\alpha$ (resp. $Y_\alpha$) to $T$-weights of $T_{x_\alpha}X_\alpha$ (resp. $T_{x_\alpha}Y_\alpha$).

  A $T$-stable curve $C$ of weight $\gamma = \wt(C)$ in $X$ (resp. $Y$) passing through $x_\alpha$ is contained in $X_\alpha$ (resp. $Y_\alpha$) if and only if $\langle \gamma^\vee , \alpha \rangle \gamma > 0$.
\end{corollary}

\begin{proof}
  Let $\SL_2(\gamma) \cdot x_\alpha$ be a $T$-stable curve passing through $x_\alpha$. Replacing $\gamma$ by $-\gamma$, we may assume that $\fg_\gamma$ acts non-trivially on $x_\alpha$. Differentiating, we get that $v_\alpha + \rho(\fg_\gamma)V_\alpha$ is in a first order neighbourhood of $x_\alpha$ in $X$ and $\gamma$ is a weight of $T_{x_\alpha}X$. If the curve is contained in $Y$ then $\gamma$ is a weight of $T_{x_\alpha}Y$. Conversely, if $\gamma$ is a root such that $\alpha + \gamma \in \Pi(V)$, then $\langle \gamma^\vee , \alpha \rangle \neq 0$\footnote{If $\langle \gamma^\vee , \alpha \rangle = 0$ and $\alpha + \gamma$ is a root, then $\alpha$ and $\gamma$ are short thus $V$ is quasi-minuscule and $\alpha + \gamma$ is long and is not a weight of $V$.} and $\SL_2(\gamma) \cdot x_\alpha$ is a $T$-stable curve of weight $\gamma$ in $X$. If furthermore $\gamma \neq -\alpha$, we have $\gamma \neq \pm\alpha$ and the curve is a plain curve thus contained in $Y$.
Finally, the curve $\SL_2(\gamma) \cdot x_\alpha$ is contained in $X_\alpha$ if and only if $s_\gamma(\alpha) < \alpha$ \emph{i.e.} if and only if $\langle \gamma^\vee , \alpha \rangle \gamma > 0$.
\end{proof}

\section{Cohomology of hyperplane sections}
\label{equiv-chev-form}

In this section we study the $T$-equivariant cohomology $\HHH^*_T(Y)$ of $Y$. We prove a Chevalley formula that completely determines the ring structure of $\HHH_T^*(Y)$. We deduce some information on the classical cohomology and on the Bia\l ynicki-Birula cells. 

\subsection{Reminders on equivariant cohomology}

We start with a recollection of some basic facts about equivariant cohomology. Our exposition is based on the papers \cite{Brioneqintro}, \cite{GKM} and \cite{Brioneqtorus}. We refer to these texts for more details.

Let $Z$ be a smooth variety with an action of a torus $T$ such that $Z^T = \{z_\alpha \ | \ \alpha \in \aleph \}$ is finite. Let $\cX(T) \simeq \ZZ^n$ be the character group of $T$. The equivariant cohomology ring $\HHH^*_T(Z)$ is an algebra over the polynomial ring $\HHH^*_T(\pt)\cong \QQQ[\cX(T)]$ via the pull-back of the structural map $Z \to {\rm Spec}(\CC)$. The Bia\l ynicki-Birula decomposition induces an additive basis for this algebra $([Z_\alpha]_T)_{\alpha \in \aleph}$. Set $\HHH^*(Z):=\HHH^*(Z,\QQQ)$. The pullback map $i^* : \HHH^*_T(Z)\to \HHH^*_T(Z^T)$ of the natural inclusion $i : Z^T \to Z$ is injective, therefore $\HHH^*_T(Z)$ can be seen as a subring of $\HHH^*_T(Z^T) \simeq \HHH_T^*(\pt)^\aleph \simeq \QQQ[\cX(T)]^\aleph$.

Via this inclusion, we will denote by $f_\alpha \in \QQQ[\cX(T)]^{\aleph}$ the pullback of the class $[Z_\alpha]_T\in \HHH^*_T(Z)$, and by $f_{\alpha}(z_\beta) = (i\circ i_{z_\beta})^*[Z_\alpha]_T$, where $i_{z_\beta} : \{z_\beta\} \to Z^T$ is the natural inclusion. More generally for $f \in \HHH^*_T(Z^T)$, we set $f(\beta) = i_{z_\beta}^*f$. If $(\epsilon_i)_{i \in [1,n]}$ is a $\ZZ$-basis of $\cX(T)$, then $f_{\alpha}(z_\beta) \in \HHH^*_T(z_\beta)$ is a polynomial in $(\epsilon_i)_{i \in [1,n]}$.

For $\alpha,\beta \in \aleph$, if $C$ is a $T$-stable curve joining $z_\alpha$ and $z_\beta$ in $Z$, then $T$ acts on $T_{z_\alpha}C$ with character $\chi$. We set $\chi_C = \chi$, this character is determined up to sign. The following results are our basic tools to compute the polynomials $f_{\alpha}(z_\beta)$.

\begin{theorem}[Theorem 3.4 \cite{Brioneqtorus}]
\label{thmfinitecurvesgen}
If $Z$ contains finitely many $T$-stable curves, then $\HHH^*_T(Z)$ is the subalgebra of $\QQQ[\cX(T)]^{\aleph}$ consisting of elements $f=(f(z_\alpha))_{\alpha \in \aleph}$ satisfying the following condition: If there exists a $T$-stable curve $C$ joining $z_\alpha$ and $z_\beta$, then 
\begin{equation}
\label{reldeteqcohom}
\textrm{
  $f(z_\alpha) - f(z_\beta) = 0 \ ({\rm mod}\ \chi_C)$.}
\end{equation}
\end{theorem}

\begin{theorem}[Theorems 4.2 and 3.4 \cite{Brioneqtorus}]
\label{thmlocgen}
Let $\alpha,\beta,\gamma \in \aleph$, we have.
\begin{enumerate}
\item The polynomial $f_{\alpha}(z_\beta)$ is homogeneous of degree $\codim(Z_\alpha)$.
\item If $z_\beta \notin Z_\alpha$, then $f_\alpha(z_\beta) = 0$.
\item If $z_\beta \in Z_\alpha$ is a smooth point, then $f_\alpha(z_\beta)$ is the product of the $T$-characters of $N_{Z_\alpha/Z,z_\beta}$.
\item If there exists a $T$-stable curve $C$ joining $z_\beta$ and $z_\gamma$, then $\chi_C$ divides $f_\alpha(z_\beta) - f_\alpha(z_\gamma)$. 
\end{enumerate}
\end{theorem}

Finally, we recover the ordinary cohomology $\HHH^*(Z)$ from the equivariant cohomology $\HHH^*_T(Z)$.

\begin{theorem}[Corollary 2.3 \cite{Brioneqtorus}]
\label{thmeqclasscohomgen}
We have $\HHH^*(Z) \simeq \HHH^*_T(Z)/(\epsilon_1,\cdots,\epsilon_n)$.
\end{theorem}

\subsection{Chevalley formula}

The main result of this section will be a Chevalley formula, \emph{i.e.} a formula for multiplying the hyperplane class $f_H$ (see definition below) with Schubert classes in $Y$. This formula determines the full equivariant cohomology (see for example \cite{bcmp}).
Set $\aleph_1 = \{ \alpha \in \aleph \ | \ \codim Y_\alpha = 1 \}$. 

\begin{fact}
  The set $\aleph_1$ is described as follows.
  \begin{enumerate}
  \item In all cases except in type $A_n$, we have $\aleph_1 = \{\alpha_0 \}$ for a unique root $\alpha_0$.
  \item In type $A_n$, we have $\aleph_1 = \{\alpha_0 ,\beta_0 \}$ for two roots $\alpha_0,\beta_0$.
  \end{enumerate}
\end{fact}

\begin{definition}
  Define the hyperplane class $f_H$ as follows.
  \begin{enumerate}
  \item In all cases except in type $A_n$, set $f_H = f_{\alpha_0}$.
  \item In type $A_n$, set $f_H = f_{\alpha_0} + f_{\beta_0}$.
  \end{enumerate}
\end{definition}

Recall the order $\preccurlyeq$ on $\aleph$ from Definition \ref{def-prec}

\begin{lemma}
  There exist $a_{\alpha}^{\beta}\in \QQQ[\cX(T)]$ of degree $\codim(Y_\beta)-\codim(Y_\alpha)-1$ such that
\begin{equation}
\label{eqindmethY}
f_\alpha(\cdot)(f_H(\cdot)-f_H(x_\alpha)) = \sum_{\beta \prec \alpha}a_{\alpha}^{\beta} f_\beta(\cdot).
\end{equation}
\end{lemma}

\begin{proof}
  By Theorem \ref{thmlocgen}, the left hand side vanishes at all $x_\beta$ with $\beta \not \preccurlyeq \alpha$ and it obviously vanishes at $x_\alpha$. Since $(f_\gamma)_{\gamma \in \aleph}$ is a basis we get that the left hand side has the form $\sum_{\beta \in \aleph}a_{\alpha}^{\beta} f_\beta(\cdot)$ and since all classes are homogeneous, we have $\deg(a_\alpha^\beta) = \codim(Y_\beta)-\codim(Y_\alpha)-1$. Now by descending induction on $\aleph$ for the order $\preccurlyeq$ and evaluation at fixed points it is easy to check that $a_\alpha^\beta$ vanishes except maybe for $\beta \prec \alpha$.
\end{proof}

A formula as above is called equivariant Chevalley formula. In the rest of this subsection, we compute the coefficients $a_{\alpha}^{\beta}$ of the equivariant Chevalley formula. Among them the constant coefficients $a_{\alpha}^{\beta}$ are the coefficients of the Chevalley formula in classical cohomology (this follows from Theorem \ref{thmeqclasscohomgen}). Moreover, notice that knowing the equivariant Chevalley formula allows to compute all $f_\alpha$'s explicitly by induction, using the order relation $\preccurlyeq$ and starting from the maximal degree class.

Before computing the coefficients $a_\alpha^\beta$, we give an explicit expression for $f_H$ and prove useful results on $T$-stable curves. Recall that $\varpi$ is the maximal element in $\aleph$ for the order $\preccurlyeq$ (or the maximal root in $\aleph$).

\begin{proposition}
The element $f_H \in \HHH^*_T(Y^T)$ is given by $f_H(x_\alpha)=\varpi - \alpha$ for $\alpha \in \aleph$.
\end{proposition}

\begin{proof}
  Let $g_H$ be defined by $g_H(\alpha) = \varpi - \alpha$. Both elements $f_H$ and $g_H$ belong to the image of $\HHH^*_T(Y)$ inside $\HHH^*_T(Y^T)$: there are finitely many $T$-stable curves in $Y$ and the element $g_H$ satisfies \eqref{reldeteqcohom}. Consider the difference $c = f_H - g_H$ which also lies in the image of $\HHH^*_T(Y)$ inside $\HHH^*_T(Y^T)$.

  We have $f_H(x_\varpi) = 0$ since $x_\varpi \notin Y_{\alpha_0}$ (and $x_\varpi \not\in Y_{\beta_0}$ in type $A_n$) and $g_H(x_\varpi) = 0$ therefore $c(x_\varpi) = 0$. Furthermore, by Proposition \ref{prop-tang} the point $x_{\alpha_0}$ is smooth in $Y_{\alpha_0}$ and the $T$-weight of the normal bundle of $Y_{\alpha_0}$ in $Y$ at $x_{\alpha_0}$ is $\varpi - \alpha_0$. We thus have $f_H(x_{\alpha_0}) = \varpi - \alpha_0 = g_H(x_{\alpha_0})$ (in type $A_n$ the same results hold for $\beta_0$ in place of $\alpha_0$) thus $c(x_{\alpha_0}) = 0$ (and $c(x_{\beta_0}) = 0$ in type $A_n$). Assume now that $c \neq 0$ and let $x_\alpha$ with $c(x_\alpha) \neq 0$ be maximal for $\preccurlyeq$.

  Recall from Proposition \ref{prop-tang} that the set of weights of the normal bundle of $Y_{\alpha}$ at $x_\alpha$ is $\{\gamma \in \Delta^+ \setminus \{ - \alpha \} \ | \ \alpha + \gamma \in \Pi(V) \}$. For any $\gamma \in \Delta^+ \setminus \{ - \alpha \}$ with $\alpha + \gamma \in \Pi(V)$, we have $\langle \gamma^\vee , \alpha \rangle < 0$. The curve $\SL_2(\gamma) \cdot x_\alpha$ is a plain $T$-stable curve in $Y$ and if $\beta = s_\gamma(\alpha) \in \aleph$, we have $x_\alpha \prec x_{\beta}$. In particular $c(x_\beta) = 0$. Since $c$ satisfies \eqref{reldeteqcohom}, we have that $\gamma$ divides $c(\beta) - c(\alpha)$ and therefore $\gamma$ divides $c(\alpha)$. This is true for all $\gamma \in \Delta^+ \setminus \{ - \alpha \}$ such that $\alpha + \gamma \in \Pi(V)$. Since these weights are distinct we see that the product of all these weights divides $c$. The degree of this product is equal to $\codim Y_\alpha$ but since $c(x_\alpha) \neq 0$ and $c$ vanishes on codimension $1$ Schubert varieties, the degree of this product is at least $2$. Since $c$ has degree $1$, we obtain a contradiction.  
\end{proof}

We start with the identification of the potential non-zero terms $a_\alpha^\beta$.

\begin{proposition}
  \label{lem-les-cas}
  Let $\alpha,\beta \in \aleph$, then $a_\alpha^\beta = 0$ except maybe if $\beta \prec \alpha$ and one of the following holds:
  \begin{enumerate}
\item $\dim Y_\beta = \dim Y_\alpha - 1$, $\dim X_\beta = \dim X_\alpha - 1$, the roots $\alpha$ and $\beta$ have the same sign, in which case there is no root $\eta \in \aleph$ with $\beta \prec \eta \prec \alpha$ and there exists a simple root $\gamma$ such that $\beta = s_\gamma(\alpha)$ and $\langle \gamma^\vee , \alpha \rangle > 0$.
\item $\dim Y_\beta = \dim Y_\alpha$, $\dim X_\beta = \dim X_\alpha - 1$, the roots $\alpha$ and $-\beta$ are simple and $\gamma = \alpha - \beta$ is a root such that $s_\gamma(\alpha) = \beta$.
\item $\dim Y_\beta = \dim Y_\alpha - 1$, $\dim X_\beta = \dim X_\alpha - 2$, $\alpha > 0 > \beta$ and one of the following occurs:
  \begin{enumerate}
  \item There is no root $\eta \in \aleph$ with $\beta \prec \eta \prec \alpha$ and there is a simple root $\epsilon \not\in \aleph$ with $\beta = - s_\epsilon(\alpha)$.
  \item There is a unique root $\eta \in \aleph$ with $\beta \prec \eta \prec \alpha$.
  \end{enumerate}
  \end{enumerate}
\end{proposition}

\begin{proof}
  By equation \eqref{eqindmethY} we already know that if $a_\alpha^\beta \neq 0$ then $\beta \prec \alpha$. Furthermore, we have $\codim Y_\alpha - \codim Y_\beta + 1 = \deg a_{\alpha}^\beta \geq 0$. We get $\dim Y_\beta \geq \dim Y_\alpha - 1$. The condition $\beta \prec \alpha$ implies $X_\beta \subsetneq X_\alpha$ (see Remark \ref{rem-order}), thus $\dim X_\beta \leq \dim X_\alpha - 1$. We thus have the inequalities:
  $$\dim Y_\alpha - 1 \leq \dim Y_\beta \leq \dim X_\beta \leq \dim X_\alpha -1.$$
  The different cases depend on the signs of $\alpha$ and $\beta$. If $\alpha <0$ then $\dim Y_\alpha = \dim X_\alpha$ and we get equality in all the above inequalities. 
  We are in case 1: $\dim Y_\beta = \dim Y_\alpha - 1$, $\dim X_\beta = \dim X_\alpha - 1$ and $\dim Y_\beta = \dim X_\beta$. The last equality implies $\beta < 0$. Furthermore, Proposition \ref{remTcurvesX} implies that there exists a simple root $\gamma$ with $\beta = s_\gamma(\alpha)$. Since $\beta \prec \alpha$, we have $\langle \gamma^\vee , \alpha \rangle > 0$. Proposition \ref{prop-ordre} and the fact that $\dim X_\beta = \dim X_\alpha - 1$ imply that there is no root $\eta \in \aleph$ with $\beta \prec \eta \prec \alpha$.

  If $\alpha > 0$, then $\dim Y_\alpha = \dim X_\alpha - 1$ and we get $\dim Y_\alpha - 1 \leq \dim Y_\beta \leq \dim X_\beta \leq \dim X_\alpha - 1 = \dim Y_\alpha$. We either have $\dim Y_\beta = \dim Y_\alpha - 1$ or $\dim Y_\beta = \dim Y_\alpha$.

  If $\dim Y_\beta = \dim Y_\alpha - 1$ and if $\beta > 0$, we are in case 1 and the same arguments as above imply that there is no root $\eta \in \aleph$ with $\beta \prec \eta \prec \alpha$, and there exists a simple root $\gamma$ such that $\beta = s_\gamma(\alpha)$ and $\langle \gamma^\vee , \alpha \rangle > 0$. Otherwise $\beta < 0$ and we have $\dim X_\beta = \dim Y_\beta = \dim Y_\alpha - 1 = \dim X_\alpha - 2$ and we are in case 3.

  If $\dim Y_\beta = \dim Y_\alpha$, we must have $\dim Y_\beta = \dim X_\beta$ and therefore $\beta < 0$ and $\dim X_\beta = \dim X_\alpha - 1$. Now by Proposition \ref{remTcurvesX}, there exists a root $\gamma$ with $s_\gamma(\alpha) = \beta$. Since both $\alpha$ and $-\beta$ are simple, the root $\gamma$ is not simple and Proposition \ref{remTcurvesX} implies that $\gamma = \alpha - \beta$.

  We finally deal with the different possibilities in case 3. Note that, since $\dim X_\beta = \dim X_\alpha - 2$, there exists a root $\delta \in \aleph$ such that $X_\beta \subsetneq X_\delta \subsetneq X_\alpha$. Furthermore, since $\alpha > 0 > \beta$, we have two possibilities: $\alpha$ is simple and  $-\delta$ is simple or $\delta$ is simple and $-\beta$ is simple.

  In the first case, by Proposition \ref{remTcurvesX} there is a simple root $\epsilon$ such that $\beta = s_\epsilon(\delta)$. Furthermore, we have $\delta = -\alpha$ or $\alpha - \delta$ is a root. If $\delta = -\alpha$ and $\epsilon \not\in \aleph$, then $\delta$ is the only root with $X_\beta \subsetneq X_\delta \subsetneq X_\alpha$ and there is no root $\eta \in \aleph$ with $\beta \prec \eta \prec \alpha$. If $\delta = -\alpha$ and $\epsilon \in \aleph$ then $\beta = \delta - \epsilon = -(\alpha + \epsilon)$ and $\eta = -\epsilon$ is the unique root such that $\beta \prec \eta \prec \alpha$. If $\delta \neq - \alpha$ and $\epsilon = \alpha$ then $\beta = -(\alpha + \epsilon)$ and as above $\eta = -\epsilon$ is the unique root such that $\beta \prec \eta \prec \alpha$. If $\delta \neq - \alpha$ and $\epsilon \neq \alpha$ then $\langle \alpha^\vee , \epsilon \rangle = 0$ and $\eta = \delta$ is the unique root such that $\beta \prec \eta \prec \alpha$.
 
For the second case, apply the previous case to $-\alpha$ and $-\beta$ using Remark \ref{rem-ordre-opp}.
\end{proof}

We will compute $a_\alpha^\beta$ in the three cases of the previous lemma. The following will be useful.
Let $\alpha,\beta \in \aleph$ with a codimension $1$ inclusion $X_\beta \subset X_\alpha$. Recall from Proposition \ref{remTcurvesX} that, in that case, there exists a root $\gamma$ with $\beta = s_\gamma(\alpha)$ such that one of the following occurs:\\
\indent
1. the root $\gamma$ is simple or \\
\indent
2. the roots $\alpha$ and $-\beta$ are simple and $\gamma = \alpha - \beta$.

\begin{lemma}
  \label{lem-cong}
Let $\alpha,\beta \in \aleph$ and $\gamma \in \Delta$ as above.
\begin{enumerate}
\item If $\gamma$ is simple, then
  we have a bijection $\{ \textrm{weights of $N_{x_\alpha,Y_\alpha}$} \}  \to \{ \textrm{weights of $N_{x_\beta,Y_\beta}$} \}\setminus \{ \gamma \}$ given by $\delta \mapsto s_\gamma(\delta)$. In particular, we have
    $$\frac{f_\beta(x_\beta)}{\gamma}  = f_\alpha(x_\alpha)  \,\,\,\,\, (\modulo \ \gamma).$$
\item If $\alpha$ and $-\beta$ are simple roots and if $\gamma = \alpha - \beta$ is a root with $s_\gamma(\alpha) = \beta$, then we have a bijection $\{ \textrm{weights of $N_{x_\alpha,Y_\alpha}$} \} \setminus \{ - \beta \} \to \{ \textrm{weights of $N_{x_\beta,Y_\beta}$} \} \setminus \{ \gamma \}$ given by $\delta \mapsto s_\gamma(\delta)$. In particular, we have
$$\frac{f_\beta(x_\beta)}{\gamma}  = - \frac{f_\alpha(x_\alpha)}{\beta}  \,\,\,\,\, (\modulo \ \gamma).$$
\item If $\alpha$ is simple, then we have a bijection $\{ \textrm{weights of $N_{x_\alpha,Y_\alpha}$} \} \to \{ \textrm{weights of $N_{x_{-\alpha},Y_{-\alpha}}$} \}$ given by $\delta \mapsto s_\alpha(\delta)$. In particular, we have $f_\alpha(x_\alpha) = s_\alpha(f_{-\alpha}(x_{-\alpha}))$.
\end{enumerate}
\end{lemma}

\begin{proof}
 1. Note that $\pm\gamma$ is not a weight of $N_{x_\alpha,Y_\alpha}$ and that the two sets have the same size. We therefore only need to prove that $s_\gamma$ maps the weights of $N_{x_\alpha,Y_\alpha}$ to weights of $N_{x_\beta,Y_\beta}$.

 Let $\delta$ be a weight of $N_{x_\alpha,Y_\alpha}$.
By Corollary \ref{coro-poids-courbes}, there is a plain $T$-stable curve $\SL_2(\delta) \cdot x_\alpha$ in $Y$ but not in $Y_\alpha$. Furthermore, we have $\langle \delta^\vee , \alpha \rangle \delta < 0$. Consider the root $\eta = s_\gamma(\delta)$. We have $\delta \neq \pm \alpha$, therefore $\eta = s_\gamma(\delta) \neq \pm s_\gamma(\alpha) = \pm \beta$ and thus $\SL_2(\eta) \cdot x_\beta$ is a plain $T$-stable curve. Furthermore, we have $\langle \eta^\vee , \beta \rangle \eta = \langle s_\gamma(\delta)^\vee , \beta \rangle s_\gamma(\delta) = \langle \delta^\vee , s_\gamma(\beta) \rangle s_\gamma(\delta) = \langle \delta^\vee , \alpha \rangle (\delta - \langle \gamma^\vee , \delta \rangle \gamma)$. Since this is a multiple of a root, since $\delta \neq \gamma$, since $\gamma$ is simple and since $\langle \delta^\vee , \alpha \rangle \delta < 0$, we have $\langle \eta^\vee , \beta \rangle \eta < 0$. In particular the $T$-stable curve $\SL_2(\eta) \cdot x_\beta$ is not contained in $Y_\beta$ and $\eta$ is a weight of $N_{x_\beta,Y_\beta}$. Note that $s_\gamma(\delta) = \delta  \,\,\,\,\, (\modulo \ \gamma)$ and using Theorem \ref{thmlocgen}.3 we get the last equality.

  2. Note that $\SL_2(\beta) \cdot x_\alpha$ is a plain $T$-stable curve joining $x_\alpha$ and $x_{\alpha - \beta}$ thus $-\beta$ is a weight of $N_{x_\alpha,Y_\alpha}$. Similarly $\SL_2(\gamma) \cdot x_\beta$ is a plain $T$-stable curve joining $x_\beta$ and $x_\alpha$, thus $\gamma$ is a weight of $N_{x_\beta,Y_\beta}$. Note finally that $\dim Y_\alpha = \dim Y_\beta$ so that the two sets have the same size and it is enough to prove that the map is well defined.

 Let $\delta$ be a weight of $N_{x_\alpha,Y_\alpha}$ different from $-\beta$. Note that $\gamma \ne \alpha$. By Corollary \ref{coro-poids-courbes}, there is a plain $T$-stable curve $\SL_2(\delta) \cdot x_\alpha$ in $Y$ but not in $Y_\alpha$. Furthermore, we have $\langle \delta^\vee , \alpha \rangle \delta < 0$. Consider the root $\eta = s_\gamma(\delta)$. We have $\delta \neq \pm \alpha$, therefore $\eta = s_\gamma(\delta) \neq \pm s_\gamma(\alpha) = \pm \beta$ and thus $\SL_2(\eta) \cdot x_\beta$ is a plain $T$-stable curve. Furthermore, we have $\langle \eta^\vee , \beta \rangle \eta = \langle s_\gamma(\delta)^\vee , \beta \rangle s_\gamma(\delta) = \langle \delta^\vee , s_\gamma(\beta) \rangle s_\gamma(\delta) = \langle \delta^\vee , \alpha \rangle (\delta - \langle \gamma^\vee , \delta \rangle \gamma)$. Since this is a multiple of a root, since $\delta \neq \gamma$, since $\gamma = \alpha - \beta$ and $\delta \not\in \{\alpha,-\beta\}$ and since $\langle \delta^\vee , \alpha \rangle \delta < 0$, we have $\langle \eta^\vee , \beta \rangle \eta < 0$. In particular, the $T$-stable curve $\SL_2(\eta) \cdot x_\beta$ is not contained in $Y_\beta$ and $\eta$ is a weight of $N_{x_\beta,Y_\beta}$. Note that $s_\gamma(\delta) = \delta  \,\,\,\,\, (\modulo \ \gamma)$, and using Theorem \ref{thmlocgen}.3 we get the last equality.

 3. This is a direct consequence of Proposition \ref{prop-tang}.
\end{proof}

We compute $a_\alpha^\beta$ in case 1. of Proposition \ref{lem-les-cas}.

\begin{lemma}
\label{lemcoeffhom}
Let $\alpha,\beta \in \aleph$ such that $\alpha$ and $\beta$ are of the same sign with $\beta \prec \alpha$ and $\dim Y_\beta = \dim Y_\alpha - 1$. We are in case 1 of Proposition \ref{lem-les-cas} and there exists a simple root $\gamma$ such that $s_\gamma(\alpha) = \beta$.
Then $\SL_2(\gamma) \cdot x_\alpha$ is a plain $T$-stable curve inside $Y$ joining $x_\alpha$ and $x_\beta$ and we have $a_{\alpha}^{\beta} = \langle \gamma^\vee , \alpha \rangle$.
\end{lemma}

\begin{proof}
  The first statement follows directly from Proposition \ref{lem-les-cas} as well as the existence of the simple root $\gamma$ and the plain $T$-stable curve $\SL_2(\gamma) \cdot x_\alpha$ joining $x_\alpha$ and $x_\beta$. By Equation \eqref{eqindmethY} we have that $f_\alpha(x_\beta)(\alpha-\beta) = a_{\alpha}^{\beta}f_\beta(x_\beta)$. Equation \eqref{reldeteqcohom} and the equality $\alpha - \beta = \langle \gamma^\vee , \alpha \rangle \gamma$ give the following relation
$$0 = f_\alpha(x_\alpha) - f_\alpha(x_\beta) = f_\alpha(x_\alpha) - \frac{a_{\alpha}^{\beta} f_\beta(x_\beta)}{\langle \gamma^\vee , \alpha \rangle \gamma} \,\,\,\,\, (\modulo \ \gamma).$$
By Lemma \ref{lem-cong}, the root $\gamma$ does not divide $f_\alpha(x_\alpha)$ and $f_\beta(x_\beta)/\gamma = f_\alpha(x_\alpha) \,\,\,\,\, (\modulo \ \gamma)$. This implies
$$f_\alpha(x_\alpha) \left( 1 - \frac{a_{\alpha}^{\beta}}{\langle \gamma^\vee , \alpha \rangle}\right) = 0 \,\,\,\,\, (\modulo \ \gamma)$$
and since $f_\alpha(x_\alpha)$ is not divisible by $\gamma$, we get the equality $a_\alpha^\beta = \langle \gamma^\vee , \alpha \rangle$.
\end{proof}

We now compute $a_{\alpha}^{\beta}$ in case 2. of Proposition \ref{lem-les-cas}.

\begin{lemma}
\label{lemsamecodim}
Let $\alpha,\beta \in \aleph$ such that $\beta \prec \alpha$ , $\dim Y_\beta = \dim Y_\alpha$, $\dim X_\beta = \dim X_\alpha - 1$, the roots $\alpha$ and $-\beta$ are simple and $\gamma = \alpha - \beta$ is a root such that $s_\gamma(\alpha) = \beta$. Then $a_{\alpha}^{\beta}=-\alpha$.
\end{lemma}

\begin{proof}
The plain $T$-stable curve $\SL(\gamma) \cdot x_\alpha$ connects $x_\alpha$ and $x_{\beta}$. Since $\alpha,\beta\in \aleph$, the roots $\alpha$ and $\beta$ have the same length. Since $\gamma = \alpha - \beta$ is a root and $\alpha$ and $-\beta$ are simple, we must have $\langle \alpha^\vee , \beta \rangle = 1 = \langle \beta^\vee , \alpha \rangle$. Thus $s_\alpha(\beta) = \beta - \alpha = -\gamma$ and the plain $T$-stable curve $\SL_2(\alpha) \cdot x_{\beta}$ joins $x_{\beta}$ and $x_{-\gamma}$. Note that $\SL_2(\alpha) \cdot x_\beta \subset Y_\beta$ since $0 > \beta > -\gamma$.

  By Equation \eqref{eqindmethY} we have $\gamma f_\alpha(x_\beta) = f_\alpha(x_\beta)(\alpha - \beta) = a_{\alpha}^{\beta} f_{\beta}(x_\beta)$. By \eqref{reldeteqcohom} applied to the curve $\SL_2(\alpha) \cdot x_{\beta}$ we know that $f_\alpha(x_\beta) - f_\alpha(x_{-\gamma})$ is divisible by $\alpha$. Since $x_{-\gamma} \notin Y_\alpha$ by Proposition \ref{lemincl2}, we have $f_\alpha(x_{-\gamma}) = 0$, thus $f_\alpha(x_\beta)$ is divisible by $\alpha$. On the other hand $f_{\beta}(x_\beta)$ is not divisible by $\alpha$ (because $\SL_2(\alpha) \cdot x_\beta \subset Y_\beta$). This implies that $a_{\alpha}^{\beta} = m \alpha$ for some $m \in \ZZ$.

  Apply \eqref{reldeteqcohom} to the plain $T$-stable curve $\SL_2(\gamma) \cdot x_{\alpha}$ and Lemma \ref{lem-cong}.2 to get
  $$0 = f_\alpha(x_\alpha) - f_\alpha(x_\beta) = f_\alpha(x_\alpha) - \frac{m \alpha f_{\beta}(x_\beta)}{\gamma} = f_\alpha(x_\alpha) + \frac{m \alpha f_{\alpha}(x_\alpha)}{\beta} = \frac{f_\alpha(x_\alpha)}{\beta}(\beta + m \alpha) \,\,\,\,\, (\modulo \ \gamma).$$
  Lemme \ref{lem-cong}.2 also implies that $-\beta$ divides $f_\alpha(x_\alpha)$ while $\gamma$ does not. We get $\beta + m \alpha = 0 \,\,\,\,\, (\modulo \ \gamma)$ and since $\gamma = \alpha - \beta$ this implies $m = -1$.
\end{proof}

We now deal with case 3.(a) of Proposition \ref{lem-les-cas}.

\begin{lemma}
\label{remincld23}
Let $\alpha,\beta \in \aleph$ with $\alpha > 0 > \beta$, $\dim X_\alpha = \dim X_\beta + 2$ and such that there is no root $\eta \in \aleph$ with $\beta \prec \eta \prec \alpha$. We are in case 3.(a) of Proposition \ref{lem-les-cas} and there exists a simple root $\epsilon \not\in\aleph$ such that $\beta = - s_\epsilon(\alpha)$. 

We have $a_{\alpha}^{\beta} = \vert\langle \epsilon^\vee , \alpha \rangle\vert = \vert\langle \epsilon^\vee , \beta \rangle\vert = \left\{\begin{array}{ll}
\|\beta\| & \textrm{ if $\alpha$ is simple } \\
\|\alpha\| & \textrm{ if $-\beta$ is simple } \\
\end{array}
\right.$.
\end{lemma}

\begin{proof}
  The first statement follows directly from Proposition \ref{lem-les-cas} as well as the existence of the simple root $\epsilon$. By Equation \eqref{eqindmethY} we have that $f_\alpha(x_\beta)(\alpha-\beta) = a_{\alpha}^{\beta}f_\beta(x_\beta)$. We have two cases: either $\alpha$ is simple and $\gamma = -\alpha$ is the only root such that $X_\beta \subsetneq X_\gamma \subsetneq X_\alpha$ or $-\beta$ is simple and $\gamma = -\beta$ is the only root such that $X_\beta \subsetneq X_\gamma \subsetneq X_\alpha$. 

  Assume first that $\alpha$ is simple. Note that the plain curve $\SL_2(\epsilon) \cdot x_{-\alpha}$ connects $x_{-\alpha}$ and $x_\beta$ and since $\epsilon \not\in\aleph$, we have $m := \langle \alpha^\vee , \epsilon \rangle \langle \epsilon^\vee , \alpha \rangle \in \{2,3\}$.

  Assume that $m = 2$. We have $s_\alpha(\beta) = - s_\alpha s_\epsilon(\alpha) = \alpha + \langle \epsilon^\vee , \alpha \rangle \epsilon - m \alpha = - \alpha + \langle \epsilon^\vee , \alpha \rangle \epsilon = - s_\epsilon(\alpha) = \beta$. In particular $s_\alpha(f_\beta(x_\beta)) = f_\beta(x_\beta)$ and for $\gamma = s_\alpha(\epsilon)$ the plain curve $\SL_2(\gamma) \cdot x_\alpha$ connects $x_\alpha$ and $x_\beta$. By Equation \eqref{reldeteqcohom} we have $f_\alpha(x_\alpha) = f_\alpha(x_\beta) \,\,\,\,\, (\modulo \ s_\alpha(\epsilon))$.
  Since by Lemma \ref{lem-cong} we have $f_{-\alpha}(x_{-\alpha}) = s_\alpha(f_\alpha(x_\alpha))$, this gives
  $$s_\alpha(f_\alpha(x_\beta)) = s_\alpha(f_\alpha(x_\alpha)) = f_{-\alpha}(x_{-\alpha}) \,\,\,\,\, (\modulo \ \epsilon).$$
Recall the equality $- \alpha - \beta = -\langle \epsilon^\vee , \alpha \rangle \epsilon$. We get 
$$f_{-\alpha}(x_{-\alpha}) = \frac{a_\alpha^\beta}{s_\alpha(\alpha - \beta)} f_\beta(x_\beta) = \frac{a_\alpha^\beta}{-\alpha - \beta} f_\beta(x_\beta) = - \frac{a_\alpha^\beta}{\langle \epsilon^\vee , \alpha \rangle \epsilon} f_\beta(x_\beta) \,\,\,\,\, (\modulo \ \epsilon).$$
By Lemma \ref{lem-cong} again, the root $\epsilon$ does not divide $f_{-\alpha}(x_{-\alpha})$ and $f_{-\alpha}(x_{-\alpha}) = f_\beta(x_\beta)/\epsilon
\,\,\,\,\, (\modulo \ \epsilon)$. We get
$$\frac{f_\beta(x_\beta)}{\epsilon} = - \frac{a_\alpha^\beta}{\langle \epsilon^\vee , \alpha \rangle \epsilon} f_\beta(x_\beta) \,\,\,\,\, (\modulo \ \epsilon).$$
This implies $a_\alpha^\beta = -\langle\epsilon^\vee , \alpha \rangle$ as desired.

Assume now that $m = 3$. Then $\alpha$ and $\epsilon$ are the simple roots of the $G_2$-root system. Note that for $\gamma = s_\alpha s_\epsilon(\alpha)$, we have $s_\gamma(\alpha) = \beta$. We thus have
$$f_\alpha(x_\alpha) = f_\alpha(x_\beta) = a_\alpha^\beta \frac{f_\beta(x_\beta)}{\alpha - \beta}  \,\,\,\,\, (\modulo \ \gamma).$$
If $\alpha$ is short, we have $f_\alpha(x_\alpha) = \epsilon (\alpha + \epsilon)$ and $f_\beta(x_\beta) = \epsilon (2\alpha + \epsilon) (3\alpha + 2\epsilon)$. We get $\epsilon(\alpha + \epsilon) = a_\alpha^\beta \epsilon(3\alpha + 2 \epsilon) \,\,\,\,\, (\modulo \ 2\alpha + \epsilon)$. This gives $a_\alpha^\beta = 1$ as desired. If $\alpha$ is long, we have $f_\alpha(x_\alpha) = \epsilon (\alpha + 3 \epsilon)$ and $f_\beta(x_\beta) = \epsilon (\alpha + 2\epsilon) (2\alpha + 3\epsilon)$. We get $\epsilon(\alpha + 3\epsilon) = a_\alpha^\beta \epsilon(\alpha + 2 \epsilon) \,\,\,\,\, (\modulo \ 2\alpha + 3\epsilon)$. This gives $a_\alpha^\beta = 3$ as desired.

Assume that $-\beta$ is simple. The proof will be very similar. Note that the plain curve $\SL_2(\epsilon) \cdot x_{\alpha}$ connects $x_{\alpha}$ and $x_{-\beta}$ and since $\epsilon \not\in\aleph$, we have $m := \langle \alpha^\vee , \epsilon \rangle \langle \epsilon^\vee , \alpha \rangle \in \{2,3\}$.

  Assume that $m = 2$. As above, we have $s_\alpha(\beta) = \beta$ and thus $s_\beta(\alpha) = \alpha$. In particular $s_\beta(f_\alpha(x_\alpha)) = f_\alpha(x_\alpha)$ and for $\gamma = s_\beta(\epsilon)$ the plain curve $\SL_2(\gamma) \cdot x_\alpha$ connects $x_\alpha$ and $x_\beta$. By Equation \eqref{reldeteqcohom} we have $f_\alpha(x_\alpha) = f_\alpha(x_\beta) \,\,\,\,\, (\modulo \ s_\beta(\epsilon))$.
    Since by Lemma \ref{lem-cong} we have $f_{-\beta}(x_{-\beta}) = s_\beta(f_\beta(x_\beta))$, this gives
  $$s_\beta(f_\alpha(x_\beta)) = s_\beta(f_\alpha(x_\alpha)) = f_{\alpha}(x_{\alpha}) 
  \,\,\,\,\, (\modulo \ \epsilon).$$
Recall the equality $\alpha + \beta = \langle \epsilon^\vee , \beta \rangle \epsilon$. We get 
$$f_{\alpha}(x_{\alpha}) = \frac{a_\alpha^\beta}{s_\beta(\alpha - \beta)} f_{-\beta}(x_{-\beta}) = \frac{a_\alpha^\beta}{\alpha + \beta} f_{-\beta}(x_{-\beta}) = \frac{a_\alpha^\beta}{\langle \epsilon^\vee , \beta \rangle \epsilon} f_{-\beta}(x_{-\beta}) \,\,\,\,\, (\modulo \ \epsilon).$$
By Lemma \ref{lem-cong} again, the root $\epsilon$ does not divide $f_{\alpha}(x_{\alpha})$ and $f_{\alpha}(x_{\alpha}) = f_{-\beta}(x_{-\beta})/\epsilon
\,\,\,\,\, (\modulo \ \epsilon)$. We get
$$\frac{f_{-\beta}(x_{-\beta})}{\epsilon} = \frac{a_\alpha^\beta}{\langle \epsilon^\vee , \beta \rangle \epsilon} f_{-\beta}(x_{-\beta}) \,\,\,\,\, (\modulo \ \epsilon).$$
This implies $a_\alpha^\beta = \langle \epsilon^\vee , \beta \rangle$ as desired.

Assume now that $m = 3$. Then $-\beta$ and $\epsilon$ are the simple root of the $G_2$-root system. Note that for $\gamma = - s_\beta s_\epsilon(\beta)$, we have $s_\gamma(\alpha) = \beta$. We thus have
$$f_\alpha(x_\beta) = f_\alpha(x_\alpha) = a_\alpha^\beta \frac{f_\beta(x_\beta)}{\alpha - \beta}  \,\,\,\,\, (\modulo \ \gamma).$$
If $-\beta$ is short, we get $(-\beta) = a_\alpha^\beta (\epsilon - 3 \beta) \,\,\,\,\, (\modulo \ \epsilon - 2\beta)$. This gives $a_\alpha^\beta = 1$ as desired. If $-\beta$ is long, we get $(-\beta) = a_\alpha^\beta (\epsilon -\beta) \,\,\,\,\, (\modulo \ 3\epsilon - 2\beta)$. This gives $a_\alpha^\beta = 3$ as desired.
\end{proof}

We now deal with case 3.(b) of Proposition \ref{lem-les-cas}.

\begin{lemma}
\label{lem_3b}
  Let $\alpha,\beta \in \aleph$ such that $\dim Y_\beta = \dim Y_\alpha - 1$, $\dim X_\beta = \dim X_\alpha - 2$ and there exists a unique $\eta \in \aleph$ such that $\beta \prec \eta \prec \alpha$. Then either $\alpha$ or $-\beta$ is simple. We have
\begin{enumerate}
\item If $\alpha$ is simple, then $-\eta$ is simple and $a_\alpha^\beta = \left\{\begin{array}{ll}
  1 = \|\beta\| & \textrm{ if $\alpha$ appears in the support of $\beta$} \\
    0 & \textrm{ otherwise.} \\
\end{array}
  \right.$.
\item If $-\beta$ is simple, then $\eta$ is simple and there exists a simple root $\epsilon$ with $\alpha = s_\epsilon(\eta)$. We have
  $$a_\alpha^\beta = \| \alpha - \beta \| = \left\{\begin{array}{ll}
  2 & \textrm{ if $-\beta$ appears in the support of $\alpha$} \\
    \langle \epsilon^\vee , \alpha \rangle & \textrm{ otherwise.} \\
\end{array}
  \right.$$
\end{enumerate}
\end{lemma}

\begin{proof}
  We have $\alpha > 0 > \beta$ and $X_\beta \subsetneq X_\eta \subsetneq X_\alpha$. Both inclusions are divisorial. This in particular implies that either $\alpha$ or $-\beta$ is simple. In the first case $-\eta$ is a simple root such that $\alpha - \eta$ is a root and there is a simple root $\epsilon$ such that $s_\epsilon(\eta) = \beta$. We have the alternative $\epsilon = \alpha$ or $\langle \epsilon^\vee , \alpha \rangle = 0$ depending on whether $\alpha$ occurs in the support of $\beta$ or not.
  In the second case $\eta$ is a simple root such that $\eta - \beta$ is a root and there is a simple root $\epsilon$ such that $s_\epsilon(\eta) = \alpha$. We have the alternative $\epsilon =  - \beta$ or $\langle \epsilon^\vee , \beta \rangle = 0$ depending on whether $-\beta$ occurs in the support of $\alpha$ or not.
  
Applying repeatedly Equation \eqref{eqindmethY} we get
  $$f_\alpha(x_\beta) = \frac{a_\alpha^\beta}{\alpha - \beta} f_\beta(x_\beta) + \frac{a_\alpha^\eta}{\alpha - \beta} f_\eta(x_\beta) = \frac{a_\alpha^\beta}{\alpha - \beta} f_\beta(x_\beta) + \frac{a_\alpha^\eta a_\eta^\beta}{(\alpha - \beta)(\eta - \beta)} f_\beta(x_\beta).$$

Assume that $\alpha$ is simple and does not appear in the support of $\beta$. By Lemma \ref{lemsamecodim} and Lemma \ref{lemcoeffhom},
we have $a_\alpha^\eta = -\alpha$ and $a_\eta^\beta = \langle \epsilon^\vee , \eta \rangle$. Note that $s_{\alpha - \eta}(\alpha) = \eta$ and $s_{s_\epsilon(\alpha-\eta)}(\alpha) = s_\epsilon s_{\alpha - \eta} s_\epsilon(\alpha) = s_\epsilon s_{\alpha - \eta}(\alpha) = s_\epsilon(\eta) = \beta$. In particular, if $\gamma = s_\epsilon(\alpha - \eta) = \alpha - \beta$, the plain $T$-stable curve $\SL_2(\gamma) \cdot x_\alpha$ connects $x_\alpha$ and $x_\beta$. By \eqref{reldeteqcohom} we get $f_\alpha(x_\alpha) = f_\alpha(x_\beta) \,\,\,\,\, (\modulo \ \alpha - \beta)$. We get
  $$f_\alpha(x_\alpha) =  \frac{a_\alpha^\beta(\eta - \beta) - \langle \epsilon^\vee , \eta \rangle \alpha}{(\alpha - \beta)(\eta - \beta)} f_\beta(x_\beta)  \,\,\,\,\, (\modulo \ \alpha - \beta).$$
  Note that $s_\beta(\alpha) = s_{s_\epsilon(\eta)}(\alpha) = s_\epsilon s_\eta s_\epsilon(\alpha) = s_\epsilon s_\eta (\alpha) = s_\epsilon (\alpha - \eta) = \alpha - s_\epsilon(\eta) = \alpha - \beta > \alpha$. In particular $\beta$ divides $f_\alpha(x_\alpha)$. Unless $f_\beta(x_\beta)$ is divisible by a root of the form $\beta + n \gamma$, this implies $a_\alpha^\beta = 0$. The only roots of the form $\beta + n \gamma$ are $\beta$ and $\beta + \gamma = \alpha$ but none of these roots divide $f_\beta(x_\beta)$, thus $a_\alpha^\beta = 0$.

  Assume that $\alpha$ is simple and appears in the support of $\beta$. Then $\epsilon = \alpha$ and $\beta = \eta - \alpha$. By Lemma \ref{lemsamecodim} and Lemma \ref{lemcoeffhom}, we have $a_\alpha^\eta = -\alpha$ and $a_\eta^\beta = \langle \alpha^\vee , \eta \rangle = 1$. Note that $\beta = -s_\eta(\alpha)$, thus by Proposition \ref{lemincl2} we have $x_\beta \not\in Y_\alpha$ and $f_\alpha(x_\beta) = 0$. We get
  $$0 = f_\alpha(x_\beta) =  \frac{(a_\alpha^\beta - 1) \alpha}{(\alpha - \beta)\alpha} f_\beta(x_\beta)$$
and thus $a_\alpha^\beta = 1$.

Assume now that $-\beta$ is simple and does not occur in the support of $\alpha$. By Lemma \ref{lemsamecodim} and Lemma \ref{lemcoeffhom}, we have $a_\eta^\beta = -\eta$ and $a_\alpha^\eta = \langle \epsilon^\vee , \alpha \rangle$.
Note that $s_{\eta - \beta}(\beta) = \eta$ and $s_{s_\epsilon(\eta - \beta)}(\beta) = s_\epsilon s_{\eta - \beta} s_\epsilon(\beta) = s_\epsilon s_{\eta - \beta}(\beta) = s_\epsilon(\eta) = \alpha$. In particular, if $\gamma = s_\epsilon(\eta - \beta) = \alpha - \beta$, the plain $T$-stable curve $\SL_2(\gamma) \cdot x_\alpha$ connects $x_\alpha$ and $x_\beta$. By \eqref{reldeteqcohom} we get $f_\alpha(x_\alpha) = f_\alpha(x_\beta) \,\,\,\,\, (\modulo \ \alpha - \beta)$. We get
  $$f_\alpha(x_\alpha) =  \frac{a_\alpha^\beta(\eta - \beta) - \langle \epsilon^\vee , \alpha \rangle \eta}{(\alpha - \beta)(\eta - \beta)} f_\beta(x_\beta)  \,\,\,\,\, (\modulo \ \alpha - \beta).$$
Note that $s_\beta(\alpha) = s_\beta(s_\epsilon(\eta)) = s_\epsilon s_\beta (\eta) = s_\epsilon (\eta - \beta) = \alpha - \beta > \alpha$. In particular $\beta$ divides $f_\alpha(x_\alpha)$. Note furthermore that $s_\alpha(\beta) = \beta - \alpha < \beta$, thus $\alpha$ does not divide $f_\beta(x_\beta)$, therefore we must have the equality $a_\alpha^\beta(\eta - \beta) - \langle \epsilon^\vee , \alpha \rangle \eta = k\beta  \,\,\,\,\, (\modulo \ \alpha - \beta)$ for some $k \in \ZZ$. This gives $a_\alpha^\beta =  \langle \epsilon^\vee , \alpha \rangle$.

Assume finally that $-\beta$ is simple and appears in the support of $\alpha$. Then $\epsilon = -\beta$ and $\alpha = \eta - \beta$. By Lemma \ref{lemsamecodim} and Lemma \ref{lemcoeffhom}, we have $a_\eta^\beta = -\eta$ and $a_\alpha^\eta = \langle \epsilon^\vee , \alpha \rangle = 1$. We get
 $$f_\alpha(x_\beta) =  \frac{a_\alpha^\beta \alpha - \eta}{(\alpha - \beta)\alpha} f_\beta(x_\beta).$$
 However, $\alpha - \beta = \eta - 2\beta$ is not a root, thus $\alpha - \beta$ does not divide $f_\beta(x_\beta)$. Therefore $\alpha - \beta$ divides $a_\alpha^\beta \alpha - \eta = (a_\alpha^\beta - 1)\alpha - \beta$ and we must have $a_\alpha^\beta = 2$.
\end{proof}

\begin{theorem}[Equivariant Chevalley formula]
  \label{thm:equi-chev}
  Let $\alpha,\beta \in \aleph$. Then $a_\alpha^\beta = 0$ unless ($\alpha \geq \beta$, $\beta \neq - \alpha$ and $|\alpha - \beta| = 1$) or ($\alpha \geq \beta$, $|\alpha - \beta| \in \{2,3\}$, $\supp(\alpha - \beta)$ is connected and $\supp(\alpha - \beta) \cap \{\alpha , -\beta \} \neq \emptyset$).

  Assume that $\alpha,\beta \in \aleph$ satisfy the above condition, then
\begin{enumerate}
\item If $|\alpha - \beta| = 1$, then $a_\alpha^\beta = \| \alpha - \beta \|$.
\item If $\alpha,-\beta \in \Phi_\aleph$ are simple, then $a_\alpha^\beta = -\alpha$.
\item If $\alpha \in \Phi_\aleph$ and $-\beta \not\in \Phi_\aleph$, then $a_\alpha^\beta = \left\{\begin{array}{ll}
\| \beta \| & \textrm{ if $|\alpha - \beta] = 2$} \\
0 & \textrm{ if $|\alpha - \beta| =3$. } \\
\end{array}\right.$
\item If $\alpha \not\in \Phi_\aleph$ and $-\beta \in \Phi_\aleph$, then $a_\alpha^\beta = \left\{\begin{array}{ll}
\| \alpha \| & \textrm{ if $|\alpha - \beta] = 2$ and $\supp(\alpha - \beta) \not\subset \aleph$} \\
\|\alpha - \beta \| & \textrm{ otherwise. } \\
\end{array}\right.$
  \end{enumerate}
\end{theorem}

\begin{proof}
  Assertion 1. is a consequence of Lemma \ref{lemcoeffhom}. Assertion 2. is a consequence of Lemma \ref{lemsamecodim}. The first case of Assertion 3. follows from Lemmas \ref{remincld23} and \ref{lem_3b}.1. The second case of Assertion 3. follows from Lemma \ref{lem_3b}.1. The first case of Assertion 4. follows from Lemma \ref{remincld23}. The second case of Assertion 4. follows from Lemma \ref{lem_3b}.2.
\end{proof}

We will give more explicit formulas in Corollary \ref{chevalley}.

\subsubsection{Hasse diagrams}

We report in the following the Hasse diagrams of $Y$ obtained with the Chevalley formula. They are the diagrams for classical cohomology (and not equivariant cohomology). We drew in red all the arrows which do not already appear in the Hasse diagrams of the corresponding $X$. In the diagrams, the codimension of the corresponding Schubert classes grows from the left to the right.

%
%
%
%
%
%

\begin{figure}[h]
\centering
\caption{Hasse diagram of $Y\subset \OGr(2,7)$}
\begin{tikzpicture}[scale=0.5]
\draw (0,0) node (1) {} node{$\bullet$} -- ++(3,0)  node (2) {} node{$\bullet$} ;
\draw (3,0) -- ++(3,1)  node (2) {} node{$\bullet$} ;
\draw (3,0)[double] -- ++(3,-1)  node (2') {} node{$\bullet$} ;
\draw (6,1)[double] -- ++(3,0)  node (3) {} node{$\bullet$} ;
\draw (6,-1) -- ++(3,0)  node (3') {} node{$\bullet$} ;
\draw (6,-1) -- ++(3,2) ;

\draw (18,-4) node (1) {} node{$\bullet$} -- ++(-3,0)  node (5) {} node{$\bullet$} ;
\draw (15,-4) -- ++(-3,1)  node (4) {} node{$\bullet$} ;
\draw (15,-4)[double] -- ++(-3,-1)  node (4') {} node{$\bullet$} ;
\draw (12,-3)[double] -- ++(-3,0)  node (-3) {} node{$\bullet$} ;
\draw (12,-5) -- ++(-3,0)  node (-3') {} node{$\bullet$} ;
\draw (12,-5) -- ++(-3,2) ;

\draw (6,-1)[color=red][double] -- ++(3,-4) ;
\draw (6,-1)[color=red][double] -- ++(3,-2) ;

\draw (6,1)[color=red][double] -- ++(3,-4) ;
\draw (6,1)[color=red][double] -- ++(3,-6) ;

\draw (12,-5)[color=red] -- ++(-3,4) ;
\draw (12,-5)[color=red] -- ++(-3,6) ;

\draw (12,-3)[color=red][double] -- ++(-3,4) ;

\end{tikzpicture}
\end{figure}
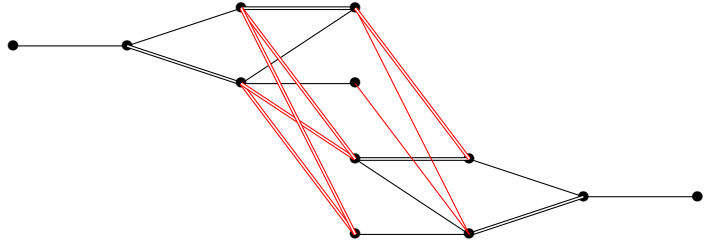

  \begin{figure}[h]
    \centering
\caption{Hasse diagram of $Y\subset \IGr(2,6)$ (already appearing in \cite{Bbisymp})}
\begin{tikzpicture}[scale=0.5]
\draw (0,0) node (1) {} node{$\bullet$} -- ++(3,0)  node (2) {} node{$\bullet$} ;
\draw (3,0) -- ++(3,1)  node (2) {} node{$\bullet$} ;
\draw (3,0) -- ++(3,-1)  node (2') {} node{$\bullet$} ;
\draw (6,1) -- ++(3,0)  node (3) {} node{$\bullet$} ;
\draw (6,-1) -- ++(3,0)  node (3') {} node{$\bullet$} ;
\draw (6,-1) -- ++(3,2) ;

\draw (18,-4) node (1) {} node{$\bullet$} -- ++(-3,0)  node (5) {} node{$\bullet$} ;
\draw (15,-4) -- ++(-3,1)  node (4) {} node{$\bullet$} ;
\draw (15,-4) -- ++(-3,-1)  node (4') {} node{$\bullet$} ;
\draw (12,-3) -- ++(-3,0)  node (-3) {} node{$\bullet$} ;
\draw (12,-5) -- ++(-3,0)  node (-3') {} node{$\bullet$} ;
\draw (12,-5) -- ++(-3,2) ;

\draw (6,-1)[color=red][double] -- ++(3,-4) ;
\draw (6,-1)[color=red][double] -- ++(3,-2) ;

\draw (6,1)[color=red] -- ++(3,-4) ;
\draw (6,1)[color=red] -- ++(3,-6) ;

\draw (12,-5)[color=red] -- ++(-3,4) ;
\draw (12,-5)[color=red] -- ++(-3,6) ;

\draw (12,-3)[color=red] -- ++(-3,4) ;

\end{tikzpicture}
\end{figure}
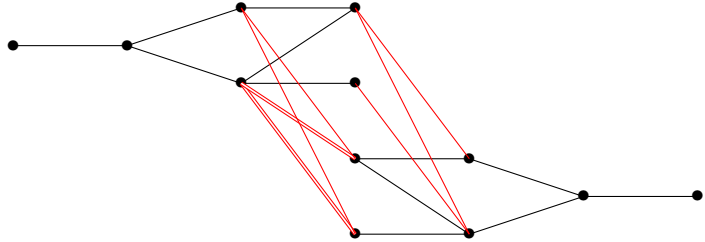

  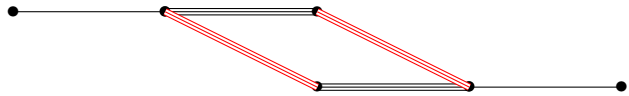
\begin{figure}[h]
    \centering
\caption{Hasse diagram of $Y\subset G_2/P_2$}
\begin{tikzpicture}[scale=0.5]
\draw (0,0) node (1) {} node{$\bullet$} -- ++(4,0)  node (2) {} node{$\bullet$} ;
\draw[double distance = 2pt] (4,0) node (1) {} node{$\bullet$} -- ++(4,0)  node (3) {} node{$\bullet$} ;
\draw (4,0) -- ++(4,0) ;

\draw (16,-2) node (-1) {} node{$\bullet$} -- ++(-4,0)  node (-2) {} node{$\bullet$} ;
\draw[double distance = 2pt] (12,-2) node (1) {} node{$\bullet$} -- ++(-4,0)  node (-3) {} node{$\bullet$} ;
\draw (12,-2) -- ++(-4,0) ;

\draw[double distance = 2pt][color=red] (12,-2) -- ++(-4,2) ;
\draw[color=red] (12,-2) -- ++(-4,2) ;
\draw[double distance = 2pt][color=red] (4,0) -- ++(4,-2) ;
\draw[color=red] (4,0) -- ++(4,-2) ;

\end{tikzpicture}
\end{figure}

  \begin{figure}[h]
\centering
\caption{Hasse diagram of $Y\subset F_4/P_1$}
\begin{tikzpicture}[scale=0.5]
\draw (0,0) node{$\bullet$} -- ++(1.5,0)  node{$\bullet$} ;
\draw (1.5,0) -- ++(1.5,0)  node{$\bullet$} ;
\draw (3,0)[double] -- ++(1.5,0)  node{$\bullet$} ;
\draw[double] (4.5,0) -- ++(1.5,1)  node{$\bullet$} ;
\draw (4.5,0) -- ++(1.5,-1)  node{$\bullet$} ;
\draw (6,1) -- ++(1.5,-2)  node{$\bullet$} ;
\draw (6,-1) -- ++(1.5,2)  node{$\bullet$} ;
\draw[double] (6,-1) -- ++(1.5,0) ;
\draw[double] (7.5,1) -- ++(1.5,0)  node{$\bullet$} ;
\draw (7.5,-1) -- ++(1.5,2)  node{$\bullet$} ;
\draw[double] (7.5,-1) -- ++(1.5,0)  node{$\bullet$} ;
\draw[double] (9,1) -- ++(3,0)  node{$\bullet$} ;
\draw (9,-1) -- ++(3,2)  node{$\bullet$} ;
\draw (9,-1) -- ++(3,0)  node{$\bullet$} ;

\draw (24,-4) node{$\bullet$} -- ++(-1.5,0)  node{$\bullet$} ;
\draw (22.5,-4) -- ++(-1.5,0)  node{$\bullet$} ;
\draw (21,-4)[double] -- ++(-1.5,0)  node{$\bullet$} ;
\draw[double] (19.5,-4) -- ++(-1.5,1)  node{$\bullet$} ;
\draw (19.5,-4) -- ++(-1.5,-1)  node{$\bullet$} ;
\draw (18,-3) -- ++(-1.5,-2)  node{$\bullet$} ;
\draw (18,-5) -- ++(-1.5,2)  node{$\bullet$} ;
\draw[double] (18,-5) -- ++(-1.5,0) ;
\draw[double] (16.5,-3) -- ++(-1.5,0)  node{$\bullet$} ;
\draw (16.5,-5) -- ++(-1.5,2)  node{$\bullet$} ;
\draw[double] (16.5,-5) -- ++(-1.5,0)  node{$\bullet$} ;
\draw[double] (15,-3) -- ++(-3,0)  node{$\bullet$} ;
\draw (15,-5) -- ++(-3,2)  node{$\bullet$} ;
\draw (15,-5) -- ++(-3,0)  node{$\bullet$} ;

\draw (9,-1)[color=red][double] -- ++(3,-4) ;
\draw (9,-1)[color=red][double] -- ++(3,-2) ;
\draw (9,1)[color=red][double] -- ++(3,-4) ;
\draw (9,1)[color=red][double] -- ++(3,-6) ;

\draw (15,-3)[color=red][double] -- ++(-3,4) ;
\draw (15,-5)[color=red] -- ++(-3,6) ;
\draw (15,-5)[color=red] -- ++(-3,4) ;

\end{tikzpicture}
\end{figure}
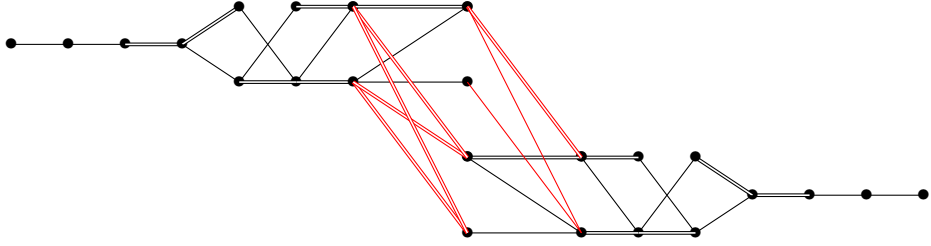

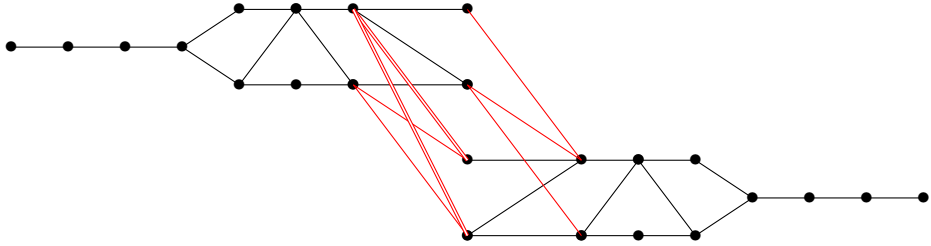
\begin{figure}[h]
  \centering
\caption{Hasse diagram of $Y\subset F_4/P_4$}
\begin{tikzpicture}[scale=0.5]
\draw (0,0) node{$\bullet$} -- ++(1.5,0)  node{$\bullet$} ;
\draw (1.5,0) -- ++(1.5,0)  node{$\bullet$} ;
\draw (3,0) -- ++(1.5,0)  node{$\bullet$} ;
\draw (4.5,0) -- ++(1.5,1)  node{$\bullet$} ;
\draw (4.5,0) -- ++(1.5,-1)  node{$\bullet$} ;
\draw (6,1) -- ++(1.5,0)  node{$\bullet$} ;
\draw (6,-1) -- ++(1.5,2)  node{$\bullet$} ;
\draw (6,-1) -- ++(1.5,0)  node{$\bullet$} ;
\draw (7.5,1) -- ++(1.5,0)  node{$\bullet$} ;
\draw (7.5,1) -- ++(1.5,-2)  node{$\bullet$} ;
\draw (7.5,-1) -- ++(1.5,0)  node{$\bullet$} ;
\draw (9,1) -- ++(3,0)  node{$\bullet$} ;
\draw (9,1) -- ++(3,-2)  node{$\bullet$} ;
\draw (9,-1) -- ++(3,0)  node{$\bullet$} ;

\draw (24,-4) node{$\bullet$} -- ++(-1.5,0)  node{$\bullet$} ;
\draw (22.5,-4) -- ++(-1.5,0)  node{$\bullet$} ;
\draw (21,-4) -- ++(-1.5,0)  node{$\bullet$} ;
\draw (19.5,-4) -- ++(-1.5,1)  node{$\bullet$} ;
\draw (19.5,-4) -- ++(-1.5,-1)  node{$\bullet$} ;
\draw (18,-3) -- ++(-1.5,0)  node{$\bullet$} ;
\draw (18,-5) -- ++(-1.5,2)  node{$\bullet$} ;
\draw (18,-5) -- ++(-1.5,0)  node{$\bullet$} ;
\draw (16.5,-3) -- ++(-1.5,0)  node{$\bullet$} ;
\draw (16.5,-3) -- ++(-1.5,-2)  node{$\bullet$} ;
\draw (16.5,-5) -- ++(-1.5,0)  node{$\bullet$} ;
\draw (15,-3) -- ++(-3,0)  node{$\bullet$} ;
\draw (15,-3) -- ++(-3,-2)  node{$\bullet$} ;
\draw (15,-5) -- ++(-3,0)  node{$\bullet$} ;

\draw (9,-1)[color=red] -- ++(3,-4) ;
\draw (9,-1)[color=red] -- ++(3,-2) ;
\draw (9,1)[color=red][double] -- ++(3,-4) ;
\draw (9,1)[color=red][double] -- ++(3,-6) ;

\draw (15,-3)[color=red] -- ++(-3,4) ;
\draw (15,-3)[color=red] -- ++(-3,2) ;
\draw (15,-5)[color=red] -- ++(-3,4) ;

\end{tikzpicture}
\end{figure}

\subsection{Classical cohomology formulae}

We have already said that from the equivariant Chevalley formula one can recover completely by induction the equivariant cohomology. This in turn gives the classical cohomology. However, this involves computing each class $f_\alpha$, which can become painful. In the following we want to give an explicit description of a set of generators and relations which define the cohomology of $Y$. In particular we describe the intersecton form on the middle cohomology.

\subsubsection{Pull-back and push-forward}

Let $j:Y\to X$ be the natural inclusion. In this section we describe the pull-back $j^*$ and the push-forward $j_*$. This will help understanding the cohomology of $Y$ in terms of the cohomology of $X$. Define $i:Y^T\to Y$ and set $f^X_\alpha := i^*j^* [X_\alpha]$. Since $i^*$ is an embedding of the equivariant cohomology of $Y$ inside $\HHH^*_T(Y^T)$, we will identify $f^X_\alpha$ and $j^* [X_\alpha]$.

Recall the definition of $\Phi_\aleph = \aleph \cap \Phi$. Note that $\Phi_\aleph$ is the base of the root system $\aleph \cap \Delta$ which is a root subsystem of $\Delta$ (the subsystem of long resp. short roots for $X$ adjoint resp. quasi-minuscule).

\begin{proposition}
\label{lempullback}
Let $\alpha \in \aleph$, we have the following formulas:
\begin{enumerate}
\item If $\alpha$ is a non-simple positive root, then $f^X_\alpha=f_\alpha$.
\item If $\alpha \in \Phi_\aleph$ is simple, then $\displaystyle{f^X_\alpha = f_\alpha + f_{-\alpha}+\sum_{\beta\in \Phi_\aleph,\ \alpha + \beta \in \Delta_\cdot} f_{-\beta}.}$
\item If $\alpha$ is negative, we have $\displaystyle{f^X_{\alpha} = -\alpha f_{\alpha} + \sum_{\beta \neq \alpha} a_{\alpha}^{\beta}f_{\beta}.}$
\end{enumerate}
\end{proposition}

\begin{proof}
  Since $(f_\alpha)_{\alpha \in \aleph}$ is a basis of equivariant cohomology, there are homogeneous elements $\lambda_\alpha^\beta \in \HHH^*_T(\pt)$ such that $f_\alpha^X = \sum_{\beta \in \aleph} \lambda_\alpha^\beta f_\beta$. By an easy descending induction on $\beta$, we have $\lambda_\alpha^\beta = 0$ for $x_\beta \not\in X_\alpha$. Comparing the degrees, we also have $\lambda_\alpha^\beta = 0$ for $\codim_Y(Y_\beta) > \codim_X(X_\alpha)$.

  For $\alpha$ positive non simple, the above vanishings imply $f_\alpha^X = \lambda_\alpha^\alpha f_\alpha$ and since $f_\alpha^X(x_\alpha) = f_\alpha(x_\alpha)$ we get the first formula.

  For $\alpha$ simple, the above vanishings imply $f_\alpha^X = \lambda_\alpha^\alpha f_\alpha + \lambda_\alpha^{-\alpha} f_{-\alpha} + \sum_{\beta \in \Phi_\aleph, \alpha + \beta \in \Delta} \lambda_\alpha^{-\beta} f_{-\beta}$. The equality $f_\alpha^X(x_\alpha) = f_\alpha(x_\alpha)$ gives $\lambda_\alpha^\alpha = 1$. We also have $f_\alpha^X(x_{-\alpha}) = f_{-\alpha}(x_{-\alpha})$ (note that $x_{-\alpha}$ is a smooth point of $X_\alpha$ and $Y_{-\alpha}$ and that the $T$-stable curves passing through $x_{-\alpha}$ and going out of $Y_{-\alpha}$ or of $X_\alpha$ are the same) giving $\lambda_\alpha^{-\alpha} = 1$. Note that $x_{-\beta}$ is a smooth point in $X_\alpha$. Comparing $T$-stable curves through $x_{-\beta}$, we get $(\alpha + \beta) f_\alpha^X(x_{-\beta}) = \beta f_{-\beta}(x_{-\beta})$. Using the equivariant Chevalley formula, we obtain $(\alpha + \beta) f_\alpha(x_{-\beta}) = -\alpha f_{-\beta}(x_{-\beta})$. This gives $\lambda_\alpha^{-\beta} = 1$.

  For $\alpha$ negative, note that $\codim_X(X_\alpha) = \codim_Y(Y_\alpha) + 1$ so that $\lambda_\alpha^\beta$ vanishes except for $\beta = \alpha$ or if $X_\beta = Y_\beta$ is a divisor in $X_\alpha = Y_\alpha$. Furthermore, we have $f_\alpha^X(x_\alpha) = - \alpha f_\alpha(x_\alpha)$ since the $T$-stable curve $\SL_2(\alpha) \cdot x_\alpha$ is in $X$ and not in $Y$. For $\beta$ such that $X_\beta = Y_\beta$ is a divisor in $X_\alpha = Y_\alpha$, we have $\alpha - \beta = a_\alpha^\beta \gamma$ for some simple root $\gamma$. Since $x_\beta$ is smooth in both $X_\alpha$ and $Y_\beta$ and the only $T$-stable curve in $X_\alpha$ not in $Y_\beta$ is $\SL_2(\gamma) \cdot x_\beta$ while $\SL_2(\beta) \cdot x_\beta$ exists in $X$ but not in $Y$, we get $\gamma f_\alpha^X(x_\beta) = -\beta f_\beta(x_\beta)$. This, together with the Chevalley formula, gives the last formula. Note that in this last formula the term $\sum a_{\alpha}^{\beta} f_{\beta}$ comes from the Chevalley formula for $f_{\alpha}$.
\end{proof}

\begin{corollary}
  \label{cor-j}
  Let $\alpha \in \aleph$. We have
  \begin{enumerate}
  \item For $\alpha$ positive, non simple, we have $j^*\sigma_{\alpha,X} = \sigma_\alpha$.
  \item For $\alpha$ negative, we have $j^*\sigma_\alpha = \sum_{\beta \neq \alpha} a_{\alpha}^{\beta}\sigma_{\beta}$ and $j_*\sigma_\alpha = \sigma_{\alpha,X}$.
    \item For $\alpha$ simple, we have $\displaystyle{j^*\sigma_{\alpha,X} = \sigma_\alpha + \sigma_{-\alpha} + \sum_{\beta\in \Phi_\aleph,\ \alpha + \beta \in \Delta_\cdot} \sigma_{-\beta}}$ and $j_*\sigma_\alpha = \sigma_{-\alpha,X}$.
  \end{enumerate}
  \end{corollary}

\begin{proof}
The pull-back formulas follow directly from the previous proposition. Since $Y_\alpha = X_\alpha$ for $\alpha$ negative, we get 2. To get $j_*\sigma_\alpha$ for $\alpha$ simple recall that $j_*j^*\sigma_{\alpha,X} = h_X \cup \sigma_{\alpha,X}$ and by the Chevalley formula in $X$ this cup product is equal to $2\sigma_{\alpha,X} + \sum_{\beta\in \Phi_\aleph,\ \alpha + \beta \in \Delta_\cdot} \sigma_{-\beta,X}$, giving the result.
\end{proof}

\begin{remark}
  \label{rem-j}
  Note that the same formulas hold for $\sigma_{\alpha,X}^-$ and $\sigma_\alpha^-$. In particular, for $\alpha$ simple, we have $j_*\sigma_{-\alpha}^- = \sigma_{-\alpha,X}^-$ and
$$j^*\sigma_{-\alpha,X}^- = \sigma_\alpha^- + \sigma_{-\alpha}^- + \sum_{\beta\in \Phi_\aleph,\ \alpha + \beta \in \Delta_\cdot} \sigma^-_{\beta}.$$
\end{remark}

Using the above formulas we give a more explicit non-equivariant Chevalley formula. First recall Chevalley formula for $X$ as explained in \cite{adjoint}.

\begin{proposition}
  Let $X$ be adjoint or quasi-minuscule and let $\alpha \in \aleph$.
  \begin{enumerate}
  \item If $\alpha$ is not simple then $h_X \cup \sigma_{\alpha,X} = \sum_{\beta \in \Phi, \ \langle \beta^\vee,\alpha \rangle > 0}\langle \beta^\vee,\alpha \rangle \sigma_{s_\beta(\alpha),X}$.
  \item If $\alpha$ is simple then $h_X \cup \sigma_{\alpha,X} = \sum_{\beta \in \Phi_\aleph}|\langle \beta^\vee,\alpha \rangle| \sigma_{-\beta,X}$.
  \end{enumerate}
\end{proposition}

\begin{corollary}[Chevalley formula]
  \label{chevalley}
 Let $Y \subset X$ be a general hyperplane section with $X$ adjoint or quasi-minuscule. Let $\alpha \in \aleph$.
  \begin{enumerate}
  \item If $\alpha > 0$ and $|\alpha| = 2$, then
    $$h \cup \sigma_\alpha = \sum_{\beta \in \Phi, \ \langle \beta^\vee,\alpha \rangle > 0}\langle \beta^\vee,\alpha \rangle \left( \sigma_{s_\beta(\alpha)} + \sigma_{-s_\beta(\alpha)} + 
    \sum_{\gamma \in \Phi \cap \aleph, \gamma + s_\beta(\alpha) \in \Delta} \sigma_{-\gamma}\right).$$
  \item If $\alpha$ is simple then $h \cup \sigma_{\alpha} = h \cup \sigma_{-\alpha} = \sum_{\beta \in \Phi, \ \langle \beta^\vee,-\alpha \rangle > 0}\langle \beta^\vee, - \alpha \rangle \sigma_{-s_\beta(\alpha)}$.
  \item Otherwise $h \cup \sigma_{\alpha} = \sum_{\beta \in \Phi, \ \langle \beta^\vee,\alpha \rangle > 0}\langle \beta^\vee,\alpha \rangle \sigma_{s_\beta(\alpha)}$.

   \end{enumerate}
\end{corollary}

\begin{proof}
  If $\alpha > 0$ with $|\alpha| > 2$, then $\sigma_\alpha = j^*\sigma_{\alpha,X}$ and $h \cup \sigma_{\alpha} = j^*(h_X \cup \sigma_{\alpha,X})$. Furthermore, since $|\alpha|> 2$, $j^*$ is an isomorphism in these degrees, therefore the result follows from Chevalley formula in $X$.

  If $\alpha < 0$ then $j_*(h \cup \sigma_\alpha) = h_X \cup \sigma_{\alpha,X}$ and since $j_*$ is an isomorphism, the result follows from Chevalley formula for $X$. Note that this also applies for $\alpha < 0$ with $-\alpha$ simple.

  If $\alpha$ is simple then $j_*(h \cup \sigma_\alpha) = h_X \cup j_*\sigma_{\alpha} = h \cup j_*\sigma_{-\alpha}$ and the result follows since $j_*$ is an isomorphism in this degree.

  Finally if $\alpha >0$ and $|\alpha| = 2$ then $h \cup \sigma_\alpha = j^*(h_X \cup \sigma_{\alpha,X}) = j^*\left(\sum_{\beta \in \Phi, \ \langle \beta^\vee,\alpha \rangle > 0}\langle \beta^\vee,\alpha \rangle \sigma_{s_\beta(\alpha),X}\right)$. Now all the roots $s_\beta(\alpha)$ occuring are simple and the result follows by Corollary \ref{cor-j}.(3).
\end{proof}

\begin{proposition}
  Let $\ell \in [0,\dim Y]$.
  \begin{enumerate}
  \item For $\ell < \dim Y$, the family $\{\sigma_\alpha \ | \ \alpha \in \aleph,\  \deg\sigma_\alpha = \ell\}$ is a basis of $\HHH^\ell(Y,\ZZ)$.
  \item For $\ell < \dim Y$, the family $\{\sigma_{-\alpha} \ | \ \alpha \in \aleph,\  \deg\sigma_\alpha = \ell\}$ is a basis of $\HHH^{2 \dim Y - \ell}(Y,\ZZ)$.
     \item For $\ell = \dim Y$, the family $\{\sigma_\alpha,\sigma_{-\alpha} \ | \ \alpha \in \Phi_\aleph \}$ is a basis of $\HHH^\ell(Y,\ZZ)$.
     \item Let $\alpha,\beta \in \aleph$ be positive non-simple roots such that $\sigma_\alpha,\sigma_\beta \in \HHH^\ell(Y,\ZZ)$. Then $\ell < \dim Y$ and
       $$\sigma_\alpha \cup \sigma_{w_0(\beta)} = \delta_{\alpha,\beta}.$$
       In particular the basis in 1. and 2. are Poincar\'e dual basis of $\HHH^\ell(Y,\ZZ)$ and $\HHH^{2 \dim Y -\ell}(Y,\ZZ)$.
  \end{enumerate}
\end{proposition}

\begin{proof}
  For $\ell > \dim Y$, by Lefschetz Theorem the two maps $j^* : \HHH^\ell(X,\ZZ) \to \HHH^\ell(Y,\ZZ)$ and  $j_* : \HHH^{2\dim Y - \ell}(Y,\ZZ) \to \HHH^{2\dim X - \ell}(X,\ZZ)$ are isomorphisms,  proving 1. and 2. Item 3. follows from the fact that $\sigma_\alpha$ is a middle comohology class if and only if $\alpha$ or $-\alpha$ is simple. Finally, for $\alpha,\beta$ as in 4., we have $\sigma_\alpha \cup \sigma_{w_0(\beta)} = j_*(\sigma_\alpha \cup \sigma_{w_0(\beta)}) = j_*(j^*\sigma_{\alpha,X} \cup \sigma_{w_0(\beta)}) =\sigma_{\alpha,X} \cup \sigma_{w_0(\beta),X} = \delta_{\alpha,\beta}$ by Poincar\'e duality in $X$ (cf. Lemma \ref{lem-dual-Z}).
\end{proof}

\begin{remark}
Note that the same results hold for opposite classes $\sigma_\alpha^-$ with degree function given by $\deg(\sigma_\alpha^-) = \dim Y - \deg(\sigma_\alpha)$.
\end{remark}

\subsubsection{Non ambient classes and middle cohomology}

The above proposition completely settles the computation of the cup product outside of the middle cohomology $\HHH^{\dim Y}(Y,\ZZ)$. We now focus on the cup product on $\HHH^{\dim Y}(Y,\ZZ)$.

\begin{definition}
  Recall the notation $\Phi_\aleph = \aleph \cap \Phi$, the set of simple roots of $\aleph \subset \Delta$.
\begin{enumerate}
\item Define the matrix $I_\aleph = (\delta_{\alpha,\beta})_{\alpha,\beta \in \Phi_\aleph}$.
\item Let $C_\aleph$ be the Cartan matrix associated to $\aleph$ and defined by $C_\aleph = (\langle \alpha^\vee,\beta \rangle)_{\alpha,\beta \in \Phi_\aleph}$. 
\end{enumerate}
Note that $C_\aleph$ is symmetric since all roots have the same length in $\aleph$.
\end{definition}

\begin{proposition}
  \label{prop-int-mat}
  We have the following intersection matrices.
\begin{enumerate}
\item $(\sigma_\alpha \cup \sigma_\beta^-)_{\alpha,\beta \in \Phi_\aleph} = I_\aleph$.
\item $(\sigma_{-\alpha} \cup \sigma_{-\beta}^-)_{\alpha,\beta \in \Phi_\aleph} = I_\aleph$.
\item $(\sigma_\alpha \cup \sigma_{-\beta}^-)_{\alpha,\beta \in \Phi_\aleph} = C_\aleph - 2 I_\aleph$.
\item $(\sigma_{-\alpha} \cup \sigma_\beta^-)_{\alpha,\beta \in \Phi_\aleph} = 0$.
\end{enumerate}
\end{proposition}

\begin{proof}
  Let $\alpha,\beta \in \Phi_\aleph$.

  Let $x_\gamma \in Y_\alpha \cap Y_{\beta}^-$. Then $\beta \leq \gamma \leq \alpha$ which implies $\alpha = \gamma = \beta$. In particular $Y_\alpha$ and $Y_\beta^-$ do not meet unless $\alpha = \beta$. In this case we have $Y_\alpha \cap Y_{\alpha}^- = \{x_\alpha \} = X_\alpha \cap X_\alpha^-$ and since the latter intersection is transverse, so is the first proving 1. Item 2. follows along the same lines.
  
Let $x_\gamma \in Y_{-\alpha} \cap Y_{\beta}^-$. Then $\beta \leq \gamma \leq -\alpha$. This is impossible so that $Y_\alpha$ and $Y_\beta^-$ never intersect, this proves 4.

For 3., note that for $x_\gamma \in Y_{\alpha} \cap Y_{-\beta}^-$, we have $-\beta \leq \gamma \leq \alpha$. This is not possible for $\beta = \alpha$ since $x_{-\alpha} \not\in Y_\alpha$ and $x_{\alpha} \not\in Y_{-\alpha}^-$. So this is possible if and only if $\alpha + \beta$ is a root. In this case, we see that $Y_\alpha \cap Y_{-\beta}$ contains the line joining $x_{-\beta}$ and $x_\alpha$ and the intersection is not transverse. We use the equality $j^*\sigma_{\alpha,X} \cup \sigma_{-\beta}^- = \sigma_{\alpha,X} \cup j_*\sigma_{-\beta}^- = \sigma_{\alpha,X} \cup \sigma_{\beta,X}^- = \delta_{\alpha,\beta}$, the formula $j^*\sigma_{\alpha,X} = \sigma_\alpha + \sigma_{-\alpha} + \sum_{\alpha + \gamma \in \Delta} \sigma_{-\gamma}$ and the cases 1., 2. and 4. to get $\sigma_\alpha \cup \sigma_{-\beta}^- = - \delta_{\alpha + \beta \in \Delta}$ where $\delta_{\alpha + \beta \in \Delta} = 1$ for $\alpha + \beta \in \Delta$ and $0$ otherwise. We conclude by the fact that $(-\delta_{\alpha + \beta \in \Delta})_{\alpha,\beta \in \Phi_\aleph} = C_\aleph - 2I_\aleph$. 
\end{proof}

\begin{definition}
  \label{def-hna}
  Let $\Hna = \Ker(j_*\vert_{\HHH^{\dim Y}(Y)})$  be the non ambient part of the middle cohomology and let $\Ham = \im(j^*\vert_{\HHH^{\dim Y}(X)})$  be the ambient part of the middle cohomology.

  For $\alpha \in \Phi_\aleph$ define $\Gamma_\alpha = \sigma_\alpha - \sigma_{-\alpha}$ and $\Gamma_\alpha^- = \sigma_{-\alpha}^- - \sigma_{\alpha}^-$.
\end{definition}

\begin{lemma}
The classes $(\Gamma_\alpha)_{\alpha \in \Phi_\aleph}$ form a basis of $\Hna$. The classes $(\Gamma_\alpha^-)_{\alpha \in \Phi_\aleph}$ form a basis of $\Hna$.
\end{lemma}

\begin{proof}
By Corollary \ref{cor-j}, we have $\Gamma_\alpha \in \Hna$ and $(\Gamma_\alpha)_{\alpha \Phi_\aleph}$ is a linearly independent family in $\Hna$. Furthermore, since $j^*j_*$ is bijective we have $\dim \Ham = |\Phi_\aleph|$. Since $\dim \HHH^{\dim Y}(Y) = 2|\Phi_\aleph|$ this gives $\dim \Hna = |\Phi_\aleph|$ and finishes the proof. The proof for the classes $(\Gamma_\alpha^-)_{\alpha \in \Phi_\aleph}$ is similar.
\end{proof}

\begin{lemma}
  \label{lemm-int-gam}
We have the following intersection matrix: $(\Gamma_\alpha \cup \Gamma_\beta^-)_{\alpha,\beta \in \Phi_\aleph} = C_\aleph - 4 I_\aleph$.
\end{lemma}

\begin{proof}
Follows directly from Proposition \ref{prop-int-mat}.
\end{proof}

For $\alpha \in \Phi_\aleph$, recall that $j_*\sigma^-_\alpha = \sigma_{\alpha,X}^- = \sigma_{w_0(\alpha),X} = j_*\sigma_{w_0(\alpha)}$ with $w_0 \in W$ the longest element of $W$ (see Lemma \ref{lem-dual-Z}), and similarly that $j_*\sigma^-_{-\alpha} = \sigma_{-\alpha,X}^- = \sigma_{-w_0(\alpha),X} = j_*\sigma_{-w_0(\alpha)}$.

\begin{definition}
  \label{def-AB}
  Define the matrices $A = (a_{\alpha,\beta})_{\alpha,\beta\in \Phi_\aleph}$ and $B = (b_{\alpha,\beta})_{\alpha,\beta\in \Phi_\aleph}$ by the following equalities:
  $$\sigma_\alpha^- = \sigma_{-i(\alpha)} + \sum_{\beta \in \Phi_\aleph} a_{\alpha,\beta} \Gamma_\beta \textrm{ and } \sigma_{-\alpha}^- = \sigma_{i(\alpha)} + \sum_{\beta \in \Phi_\aleph} b_{\alpha,\beta} \Gamma_\beta$$
where $i(\alpha) := -w_0(\alpha)$ is the Cartan involution on $\Phi$ (note that this involution stabilises $\Phi_\aleph$).
\end{definition}

\begin{proposition}
  \label{prop-AB}
  We have $B = (C_\aleph - 3I_\aleph)A$.
\end{proposition}

\begin{proof}
Let $\alpha \in \Phi_\aleph$. Since $\sigma_{-\alpha,X}^- = \sigma_{i(\alpha),X}$, by pull-back using Corollary \ref{cor-j} and Remark \ref{rem-j}, we get $\sigma_\alpha^- + \sigma_{-\alpha}^- + \sum_{\beta \in \Phi_\aleph,\ \alpha + \beta \in \Delta} \sigma^-_{\beta} = \sigma_{i(\alpha)} + \sigma_{-i(\alpha)} + \sum_{\beta \in \Phi_\aleph,\ \alpha + \beta \in \Delta} \sigma_{-i(\beta)}$. By Definition \ref{def-AB}, we get
  $$\sigma_{-i(\alpha)} + \sum_{\gamma \in \Phi_\aleph} a_{\alpha,\gamma} \Gamma_\gamma + \sigma_{i(\alpha)} + \sum_{\gamma \in \Phi_\aleph} b_{\alpha,\gamma} \Gamma_\gamma + \sum_{\beta \in \Phi_\aleph,\ \alpha + \beta \in \Delta} \left(\sigma_{-i(\beta)} + \sum_{\gamma \in {\Phi_\aleph}} a_{\beta,\gamma} \Gamma_\gamma\right) = $$
$$  = \sigma_{i(\alpha)} + \sigma_{-i(\alpha)} + \sum_{\beta \in \Phi_\aleph,\ \alpha + \beta \in \Delta} \sigma_{-i(\beta)}.$$
  Simplifying gives
    $$\sum_{\gamma \in \Phi_\aleph} a_{\alpha,\gamma} \Gamma_\gamma + \sum_{\gamma \in \Phi_\aleph} b_{\alpha,\gamma} \Gamma_\gamma + \sum_{\beta \in \Phi_\aleph,\ \alpha + \beta \in \Delta} \sum_{\gamma \in {\Phi_\aleph}} a_{\beta,\gamma} \Gamma_\gamma = 0$$
  Thus $B$ is obtained from $A$ by left multiplication with the matrix $(- \delta_{\alpha,\beta} - \delta_{\alpha+\beta \in \Delta})_{\alpha,\beta \in \Phi_\aleph} = C_\aleph - 3I_\aleph$.
\end{proof}

\begin{definition}
Define the matrix $J = (j_{\alpha,\beta})_{\alpha,\beta \in \Phi_\aleph}$ by $j_{\alpha,\beta} = \delta_{\alpha,i(\beta)}$.
\end{definition}

\begin{corollary}
  \label{coro-int-gam}
  Let $D = (d_{\alpha,\beta})_{\alpha,\beta \in \Phi_\aleph}$ be the base change matrix from $(\Gamma_\alpha)_{\alpha \in \Phi_\aleph}$ to $(\Gamma^-_\alpha)_{\alpha \in \Phi_\aleph}$ \emph{i.e.} $\Gamma_\alpha^- = \sum_{\beta\in \Phi_\aleph} d_{\alpha,\beta} \Gamma_\beta$. Then we have $D = J + (C_\aleph - 4I_\aleph)A$.
\end{corollary}

\begin{proof}
We have $\Gamma_\alpha^- = \sigma_{-\alpha}^- - \sigma_\alpha^- = \sigma_{i(\alpha)} + \sum_{\beta \in \Phi_\aleph} b_{\alpha,\beta} \Gamma_\beta - \sigma_{-i(\alpha)} - \sum_{\beta \in \Phi_\aleph} a_{\alpha,\beta} \Gamma_\beta = \Gamma_{i(\alpha)} + \sum_{\beta \in \Phi_\aleph} (b_{\alpha,\beta} - a_{\alpha,\beta}) \Gamma_\beta$, so $D = J + (B - A)$ and the result follows from the previous proposition. 
\end{proof}

\begin{corollary}
  \label{cor-int-G}
  We have the intersection matrix $(\Gamma_\alpha^- \cup \Gamma_\beta^-)_{\alpha,\beta \in \Phi_\aleph} = (J + (C_\aleph - 4I_\aleph)A)(C_\aleph - 4I_\aleph)$ and $(\sqrt{-1})^{\dim Y}(J + (C_\aleph - 4I_\aleph)A)(C_\aleph - 4I_\aleph)$ is positive definite.
\end{corollary}

\begin{proof}
  The formula for the intersection matrix follows from Lemma \ref{lemm-int-gam} and Corollary \ref{coro-int-gam}. The fact that this matrix multiplied by $(\sqrt{-1})^{\dim Y}$ is positive definite is a classical fact from Hodge Theory, see \cite[Theorem 6.32]{voisinI}.
\end{proof}

We recover the following well know fact on eigenvalues of Cartan matrices.

\begin{corollary}
The matrix $C_\aleph - 4 I_\aleph$ is invertible.
\end{corollary}

\begin{lemma}
The matrices $J$ and $C_\aleph$ are symmetric and commute.
\end{lemma}

\begin{proof}
We already remarked that $C_\aleph$ is symmetric since the roots in $\Phi_\aleph$ are of the same length. Since $i$ is an involution, we get $J = J^t = J^{-1}$. Finally, we have $\langle i(\alpha)^\vee,i(\beta) \rangle = \langle \alpha^\vee,\beta \rangle$ which implies $J^t C_\aleph J = C_\aleph$ proving the commuting relation.
\end{proof}

\begin{corollary}
  The matrix $A$ is symmetric.
\end{corollary}

\begin{proof}
The matrix $(J + (C_\aleph - 4I_\aleph)A)(C_\aleph - 4I_\aleph)$ is symmetric. Since $J$ and $C_\aleph$ are symmetric and commute, we get $(C_\aleph - 4I_\aleph)A(C_\aleph - 4I_\aleph) = (C_\aleph - 4I_\aleph)A^t(C_\aleph - 4I_\aleph)$ and since $C_\aleph - 4I_\aleph$ is invertible we get $A^t = A$.
\end{proof}

\begin{proposition}
  \label{prop-MN}
  Consider the matrices
  $$M = \left(\begin{array}{cc}
    J & 0 \\
    A & J + (C_\aleph - 4I_\aleph)A \\
  \end{array}\right)
  \textrm{ and }
  N = \left(\begin{array}{cc}
    J & 0 \\
    B & J + (C_\aleph - 4I_\aleph)A \\
  \end{array}\right).$$
  Then $M$ and $N$ are involutive.
\end{proposition}

\begin{proof}
  Let $(f_\alpha^-)_{\alpha \in \aleph}$ be the restriction to $\HHH^*_T(Y^T)$ of the equivariant classes of $(Y_\alpha^-)_{\alpha \in \aleph}$. Reversing the one parameter subgroup defining the Bia\l ynicki-Birula decomposition does not modify the graph of $T$-stable curves connecting $T$-stable points but changes the signs of weights. This implies that $f_{\alpha}^-(x_{\beta}) = (-1)^{\dim Y_\alpha} f_{-\alpha}(x_{-\beta})$. This in particular implies that expanding the classes $(f_\alpha^-)_{\alpha \in \aleph}$ in terms of the classes $(f_\alpha)_{\alpha \in \aleph}$ yields the same formulas as expanding the classes $(f_\alpha)_{\alpha \in \aleph}$ in terms of the classes $(f_\alpha^-)_{\alpha \in \aleph}$. More precisely, if $f_\alpha^- = \sum_\beta \lambda_{\alpha,\beta} f_\beta$, we have $f_{-\alpha} = \sum_{\beta} (-1)^{\dim Y_\alpha + \codim Y_\beta}\lambda_{\alpha,\beta}f_{-\beta}^-$. In particular, restricting to non equivariant cohomology and to middle cohomology, we get the equalities
    $$\sigma_{-\alpha} = \sigma_{i(\alpha)}^- + \sum_{\beta \in \Phi_\aleph} a_{\alpha,\beta} \Gamma_\beta^- \textrm{ and } \sigma_{\alpha} = \sigma_{-i(\alpha)}^- + \sum_{\beta \in \Phi_\aleph} b_{\alpha,\beta} \Gamma_\beta^-.$$
  This proves that the base change from the basis $(\Gamma_\alpha)_{\alpha \in \Phi_\aleph}$ to $(\Gamma_\alpha^-)_{\alpha \in \Phi_\aleph}$ is involutive and conversely that the base changes from the basis $(\sigma_\alpha,\Gamma_\alpha)_{\alpha \in \Phi_\aleph}$ to $(\sigma_\alpha^-,\Gamma_\alpha^-)_{\alpha \in \Phi_\aleph}$, which are given by the matrices $M$ and $N$, are involutive. This proves the result.
\end{proof}

\begin{corollary}
  \label{cor-form}
The matrices $A$ and $C_\aleph$ commute and we have $AJ + JA + (C_\aleph - 4I_\aleph)A^2 = 0$.
\end{corollary}

\begin{proof}
  The fact that $M$ and $N$ are involutive implies the relations $AJ + JA + (C_\aleph - 4I_\aleph)A^2 = 0$ and $(J + (C_\aleph - 4I_\aleph)A)^2 = I_\aleph$.
  Taking the transpose of the first equality gives $AJ + JA + A^2(C_\aleph - 4I_\aleph) = 0$. Furthermore, since $J$ and $C_\aleph$ commute, the second equality is equivalent to $(C_\aleph - 4I_\aleph)(JA + AJ + A(C_\aleph - 4I_\aleph)A) = 0$ and since $C_\aleph - 4I_\aleph$ is invertible, this gives $(JA + AJ + A(C_\aleph - 4I_\aleph)A) = 0$. We thus have the equalities
  $$(C_\aleph - 4I_\aleph)A^2 = A^2(C_\aleph - 4I_\aleph) = A(C_\aleph - 4I_\aleph) A.$$
  Since $A$ is a symmetric matrix with real entries, it is semi-simple and for any eigenvector $v$ of $A$ with eigenvalue $\lambda \neq 0$ the above relations imply that $(C_\aleph - 4I_\aleph)v$ is again an eigenvector for $A$ with eigenvalue $\lambda$. Furthermore, $A$ being semi-simple, we have $\Ker(A^2) = \Ker(A)$ and for $v \in \Ker(A)$ the above imply $(C_\aleph - 4I_\aleph)v \in \Ker(A^2) = \Ker(A)$. This proves that $A$ commutes with $C_\aleph - 4I_\aleph$ and thus with $C_\aleph$.
\end{proof}

\begin{theorem}
\label{thm_inter_product_middle_cohom}
  We have the formula $A = (J + (\sqrt{-1})^{\dim Y}I_\aleph)(4I_\aleph - C_\aleph)^{-1}$. Furthermore, we have the following intersection matrices
  $$(\Gamma_\alpha \cup \Gamma_\beta)_{\alpha,\beta \in \Phi_\aleph} = (\Gamma^-_\alpha \cup \Gamma^-_\beta)_{\alpha,\beta \in \Phi_\aleph} =
  (\sqrt{-1})^{\dim Y}(4I_\aleph - C_\aleph) \textrm{ and }$$
  $$(\sigma_\alpha \cup \sigma_\beta)_{\alpha,\beta \in \Phi_\aleph \cup - \Phi_\aleph} = \frac{1}{4I_\aleph - C_\aleph} \left(\begin{array}{cc}
    J + (\sqrt{-1})^{\dim Y}(C_\aleph - 3I_\aleph)^2 & J + (\sqrt{-1})^{\dim Y}(C_\aleph - 3I_\aleph) \\
    J + (\sqrt{-1})^{\dim Y}(C_\aleph - 3I_\aleph)   &  J + (\sqrt{-1})^{\dim Y}I_\aleph \\
  \end{array}\right).$$
\end{theorem}

\begin{proof}
  The fact that the intersection matrices $(\Gamma_\alpha \cup \Gamma_\beta)_{\alpha,\beta \in \Phi_\aleph}$ and $(\Gamma^-_\alpha \cup \Gamma^-_\beta)_{\alpha,\beta \in \Phi_\aleph}$ are the same follows from Proposition \ref{prop-MN}. The formula for this matrix is equivalent to the first formula thanks to Corollary \ref{cor-int-G}. 

  We now prove the equality $(\Gamma^-_\alpha \cup \Gamma^-_\beta)_{\alpha,\beta \in \Phi_\aleph} = (\sqrt{-1})^{\dim Y}(4I_\aleph - C_\aleph)$. By Corollary \ref{cor-int-G} the intersection matrix is equal to $(J + (C_\aleph - 4I_\aleph)A)(C_\aleph - 4I_\aleph)$. Since $A$, $J$ and $C_\aleph$ are real symmetric and pairwise commute, there exists a basis $(e_\alpha)_{\alpha \in \Phi_\aleph}$ of common eigenvectors. Let $a_\alpha$, $\mu_\alpha$ and $\lambda_\alpha$ be the corresponding eigenvalues. Note that $\mu_\alpha \in \{-1,1\}$ since $J$ is an involution and that $4 - \lambda_\alpha \neq 0$ since $C_\aleph - 4I_\aleph$ is invertible. By Corollary \ref{cor-int-G} and Corollary \ref{cor-form}, we have the following two conditions on these eigenvalues:
  $$2 \mu_\alpha a_\alpha +(\lambda_\alpha - 4)a^2_\alpha = 0 \textrm{ and } (\sqrt{-1})^{\dim Y}(\mu_\alpha + (\lambda_\alpha-4)a_\alpha)(\lambda_\alpha-4)>0.$$
In particular $\mu_\alpha + (\lambda_\alpha - 4)a_\alpha = \pm \mu_\alpha$ and the eigenvalues of the matrix $(\sqrt{-1})^{\dim Y}(J + (C_\aleph - 4I_\aleph)A)(C_\aleph - 4I_\aleph)$ are given by
  $$(\sqrt{-1})^{\dim Y}(\mu_\alpha + (\lambda_\alpha - 4)a_\alpha)(\lambda_\alpha - 4) = \pm(\sqrt{-1})^{\dim Y}\mu_\alpha(\lambda_\alpha - 4).$$
  Since this last matrix is positive definite its eigenvalues are equal to $|\pm(\sqrt{-1})^{\dim Y}\mu_\alpha(\lambda_\alpha - 4)| = |\lambda_\alpha - 4| = 4 - \lambda_\alpha$. The last equality follows from the fact that the eigenvalues $\lambda_\alpha$ of the Cartan matrix are strictly smaller than $4$. In fact if $\lambda_\alpha$ is an eigenvalue of $C_\aleph$ then so is $4 - \lambda_\alpha$ (see \cite{BLM}) proving that $4 - \lambda_\alpha >0$ since $C_\aleph$ is positive definite. We thus get $(\sqrt{-1})^{\dim Y}(\mu_\alpha + (\lambda_\alpha - 4)a_\alpha)(\lambda_\alpha - 4) = 4 - \lambda_\alpha$ and $a_\alpha = \frac{\mu_\alpha + (\sqrt{-1})^{\dim Y}}{4 - \lambda_\alpha}$ proving the first formula.

  Proposition \ref{prop-int-mat}, Proposition \ref{prop-AB} and the definition of the matrices $A$ and $B$ (see Definition \ref{def-AB}) imply the following formula
  $$(\sigma_\alpha \cup \sigma_\beta)_{\alpha,\beta \in \Phi_\aleph \cup - \Phi_\aleph} =
  \left(\begin{array}{cc}
    (C_\aleph - 2I_\aleph)J + (C_\aleph - 3I_\aleph)^2A & J + (C_\aleph - 3I_\aleph)A \\
    J + (C_\aleph - 3I_\aleph)A & A \\
  \end{array}\right).$$
  Expressing these matrices in the eigenbasis we get
$$ \left(\begin{array}{cc}
   (\lambda_\alpha - 2)\mu_\alpha + (\lambda_\alpha -3)^2a_\alpha & \mu_\alpha + (\lambda_\alpha - 3)a_\alpha \\
 \mu_\alpha + (\lambda_\alpha - 3)a_\alpha & a_\alpha \\
  \end{array}\right).$$
  Replacing $a_\alpha$ by its value $a_\alpha = \frac{\mu_\alpha + (\sqrt{-1})^{\dim Y}}{4 - \lambda_\alpha}$, we get
$$\left(\begin{array}{cc}
    \frac{\mu_\alpha + (\sqrt{-1})^{\dim Y}(\lambda_\alpha - 3)^2}{4 - \lambda_\alpha} & \frac{\mu_\alpha + (\sqrt{-1})^{\dim Y}(\lambda_\alpha - 3)}{4 - \lambda_\alpha} \\
      \frac{\mu_\alpha + (\sqrt{-1})^{\dim Y}(\lambda_\alpha - 3)}{4 - \lambda_\alpha} &      \frac{\mu_\alpha + (\sqrt{-1})^{\dim Y}}{4 - \lambda_\alpha} \\
  \end{array}\right)$$
  proving the last formula.
\end{proof}

We recover results from \cite{PhDBen} for hyperplane sections in $X = \IGr(2,2n)$.

\begin{corollary} In particular if $\dim Y \equiv 2 \ ({\rm mod}\ 4)$ and $J = I_\aleph$ we have $A = 0$, the intersection matrices $(\Gamma_\alpha \cup \Gamma_\beta)_{\alpha,\beta \in \Phi_\aleph} = 4I_\aleph - C_\aleph$ and 
  $$(\sigma_\alpha \cup \sigma_\beta)_{\alpha,\beta \in \Phi_\aleph \cup - \phi_\aleph} = \left(\begin{array}{cc}
    C_\aleph - 2I_\aleph & I_\aleph \\
    I_\aleph & 0 \\
  \end{array}\right).$$
  This occurs for $X = \OGr(2,2n+1)$, $X = \IGr(2,2n)$ or $X$ of (co)adjoint type for the group $F_4$.
\end{corollary}

For $\alpha \in \Phi_\aleph$, we write
$$\sigma_\alpha^\vee = \sum_{\beta \in \Phi_\aleph} (p_{\alpha,\beta}\sigma_\beta + q_{\alpha,\beta}\sigma_{-\beta}) \textrm{ and } \sigma_{-\alpha}^\vee = \sum_{\beta \in \Phi_\aleph} (r_{\alpha,\beta}\sigma_\beta + s_{\alpha,\beta}\sigma_{-\beta}).$$
Set $P = (p_{\alpha,\beta})_{\alpha,\beta \in \Phi_\aleph}$, $Q = (q_{\alpha,\beta})_{\alpha,\beta \in \Phi_\aleph}$, $R = (r_{\alpha,\beta})_{\alpha,\beta \in \Phi_\aleph}$ and $S = (s_{\alpha,\beta})_{\alpha,\beta \in \Phi_\aleph}$.

\begin{proposition}
  \label{prop_dual_middle_cohom}
  Set $\varepsilon = (\sqrt{-1})^{\dim Y}$. We have the equality
  $$\left(\begin{array}{cc}
    P & Q \\
    R & S \\
  \end{array}\right) = \varepsilon J(4I_\aleph - C_\aleph)^{-1}\left(\begin{array}{cc}
    J + \varepsilon I_\aleph
    &  - J - \varepsilon(C_\aleph - 3I_\aleph) \\
    - J - \varepsilon(C_\aleph - 3I_\aleph) & J + \varepsilon(C_\aleph - 3I_\aleph)^2 \\
  \end{array}\right).$$
\end{proposition}

\begin{proof}
  The above matrix is the inverse of the matrix $(\sigma_{\alpha} \cup \sigma_\beta)_{\alpha,\beta \in \Phi_\aleph \cup -\Phi_\aleph}$. Since $J$ and $C_\aleph$ commute and from Theorem \ref{thm_inter_product_middle_cohom}, we have
 $$\left(\begin{array}{cc}
    P & Q \\
    R & S \\
  \end{array}\right) = (4I_\aleph - C_\aleph) \left(\begin{array}{cc}
    J + (\sqrt{-1})^{\dim Y}(C_\aleph - 3I_\aleph)^2 & J + (\sqrt{-1})^{\dim Y}(C_\aleph - 3I_\aleph) \\
    J + (\sqrt{-1})^{\dim Y}(C_\aleph - 3I_\aleph)   &  J + (\sqrt{-1})^{\dim Y}I_\aleph \\
  \end{array}\right)^{-1}.$$
  The above inverse is easily computed since all matrices involved in the blocks commute. We get
   $$(\sqrt{-1})^{\dim Y}J^{-1}(4I_\aleph - C_\aleph)^{-2+1}\left(\begin{array}{cc}
    J + (\sqrt{-1})^{\dim Y}I_\aleph
    &  - J - (\sqrt{-1})^{\dim Y}(C_\aleph - 3I_\aleph) \\
    - J - (\sqrt{-1})^{\dim Y} (C_\aleph - 3I_\aleph) & J + (\sqrt{-1})^{\dim Y} (C_\aleph - 3I_\aleph)^2 \\
  \end{array}\right).$$
  The result follows from this and the fact that $J^{-1} = J$ since $J$ is involutive.
\end{proof}

\begin{example}
  \label{ex:classe-duale}
  \begin{enumerate}
    \item For $G$ not of type $A$, we compute the class $\sigma_{\alpha_0}^\vee + \sigma_{-\alpha_0}^\vee$ as a linear combination of classes $(\sigma_\alpha,\sigma_{-\alpha})_{\alpha \in \Phi_\aleph}$. Since we are computing this sum, we consider the sums $P + R = J$ and $Q + S = J(3I_\aleph - C_\aleph)$ and we are looking for the line corresponding to the root $\alpha_0$. Since $i(\alpha_0) = \alpha_0$, we may ignore the matrix $J$ and we get
      $$\sigma_{\alpha_0}^\vee + \sigma_{-\alpha_0}^\vee = \sigma_{\alpha_0} + \sigma_{-\alpha_0} + \sum_{\beta \in \Phi_\aleph , \langle \beta^\vee , \alpha_0 \rangle < 0} \sigma_{-\beta}=j^*\sigma_{\alpha_0,X}.$$
    \item For $G$ of type $A$, we can use the same technique and we get
      $$\sigma_{\alpha_1}^\vee + \sigma_{-\alpha_1}^\vee = \sigma_{\alpha_n} + \sigma_{-\alpha_n} + \sigma_{-\alpha_{n-1}}=j^*\sigma_{\alpha_n,X} \textrm{ and } \sigma_{\alpha_n}^\vee + \sigma_{-\alpha_n}^\vee = \sigma_{\alpha_1} + \sigma_{-\alpha_1} + \sigma_{-\alpha_{2}}=j^*\sigma_{\alpha_1,X}.$$
  \end{enumerate}
\end{example}

\subsubsection{Generators and relations}

At this point, one may notice that a multiplicative set of generators of $\HHH^*(Y)$ is given by a set of generators of $j^*\HHH^*(X)$ and the classes $\sigma_{-\alpha}$ for $\alpha>0$ simple. This comes from the fact that the multiplication by the hyperplane class generates all classes $\sigma_{-\beta}$ for $\beta>0$ non simple (by Lefschetz hyperplane theorem). Moreover notice that all intersection products can be derived from the Chevalley formula, the intersection product inside $j^*\HHH(X)$ and the results in the previous section. Since we have described a set of generators of $\HHH^*(Y)$, we want to recover a complete set of relations. Let us suppose that 
\[
\HHH^*(X)=\QQQ[\eta_{1}^X,\cdots,\eta_g^X]/ ( R_1^X,\cdots,R_r^X),
\]
where $\eta_i^X$ is a generator of degree $d_i$ and $R_j$ are homogeneous relations. We will suppose that $\eta_1^X$ is the class of a hyperplane section of $X\subset \PP(V)$. Notice that the generators can be chosen so that their degree is $< \dim(X)/2$ because multiplication by the hyperplane class generates all classes of degree $> \dim(X)/2$ (again by Lefschetz hyperplane theorem). We will denote by $\eta_i:=j^*\eta_i^X$. The following lemma is a general result for hyperplane sections.

\begin{lemma}
\label{lemrelationsfromX}
If $R^X$ is a relation in $\HHH^*(X)$ and $R^X=\eta_1^X Q^X$ for a certain element $Q^X\in \HHH^*(X)$ then $j^*Q^X=0$. All the relations in $j^* \HHH^*(X)$ are either of this kind or the pull-back of a relation in $ \HHH^*(X)$.
\end{lemma}

\begin{proof}

Since $j_*j^*(\bullet)=\eta_1^X \cup \bullet$, it is clear that if $R$ is a polynomial in the variables $\eta_1,\cdots,\eta_g$ which vanishes inside $j^*\HHH^*(X)$, then $j_* R=\eta_1^X R^X$, where $R^X$ is the same polynomial $R$ but in the variables $\eta_1^X,\cdots,\eta_g^X$; thus either $R^X=0$ or $\eta_1^X R^X=0$. The first assertion is a classical fact about hyperplane sections.
\end{proof}

Let $R^X$ be a relation in $\HHH^*(X)$. We define $R$ as follows:
\[
R:= \left\{ \begin{array}{ll}
  j^*Q^X & \mbox{ if }R^X=\eta_1^X Q^X, \\
  R:=j^*R^X & \mbox{ otherwise}.\\
\end{array}\right.
\]
Let us denote by $\cJ$ the ideal generated by $R_1^X,\cdots,R_r^X$, and by $j^*\cJ$ the ideal generated by all $R$'s obtained from all relations $R^X\in \cJ$. The previous lemma proves the equality $j^*  \HHH^*(X) = \QQQ[\eta_{1},\cdots,\eta_g]/ j^*\cJ$. Let us denote by $\alpha_1,\cdots,\alpha_n$ the simple roots in $\Phi_{\aleph}$ and let $[\pt] = L(\eta_1,\cdots,\eta_g) \in \HHH^*(Y)$.

\begin{theorem}
  \label{thm:presH}
The cohomology of $Y$ admits the following presentation
\[
\HHH^*(Y)=\QQQ[\eta_1,\cdots,\eta_g,\sigma_{-\alpha_1},\cdots,\sigma_{-\alpha_n}]/\cI,
\]
where $\cI$ is the ideal generated by $j^*\cJ$ and the following relations:
\[
(\sigma_{-\alpha_i} \eta_j)_{1\leq i\leq n, 1\leq k\leq g} = (j^*(((4I_\aleph-C_\aleph)^{-1}\sigma_{-\alpha_i,X}) \eta_j^X))_{1\leq i\leq n, 1\leq k\leq g},
\]
\[
(\sigma_{-\alpha_i} \sigma_{-\alpha_j})_{1\leq i, j \leq n}=  (J + (\sqrt{-1})^{\dim Y}I_\aleph)(4I_\aleph - C_\aleph)^{-1}L(\eta_1,\cdots,\eta_g).
\]
\end{theorem}

\begin{remark}
Notice that: 
\begin{enumerate}
\item in the second set of relations, $(4I_\aleph-C_\aleph)^{-1}:\HHH^{\dim(X)+1}(X) \to \HHH^{\dim(X)-1}(X)$ is the inverse of the multiplication by the hyperplane section in the two basis $\sigma_{-\alpha_1,X},\cdots,\sigma_{-\alpha_n,X}$ and $\sigma_{\alpha_1,X},\cdots,\sigma_{\alpha_n,X}$;
\item $\sigma_{\alpha_i,X}$ can be expressed in terms of the generators $\eta_1^X,\cdots,\eta_g^X$ using the Giambelli formulas for $X$, thus the RHS of the second set of relations (and of the third set of relations) is a polynomial in $\eta_1^X,\cdots,\eta_g^X$;
\end{enumerate}\end{remark}

\begin{proof}
The first set of relations are a consequence of Lemma \ref{lemrelationsfromX}, while the third set of relations are a consequence of Theorem \ref{thm_inter_product_middle_cohom}. Let us deduce the second set of relations. In order to do so, we want to compute coefficients $c_{i,j}^\beta$ such that $\sigma_{-\alpha_i}\cup \eta_j=\sum_{\beta\in \aleph} c_{i,j}^\beta \sigma_{\beta} $. Note that, in such a situation, if the coefficient $c_{i,j}^\beta$ is non-zero then $\beta$ is negative and not simple. Then, by Corollary \ref{cor-j}, $j_*(\sigma_{-\alpha_i}\cup \eta_j)=\sum_{\beta\in \aleph} c_{i,j}^\beta \sigma_{\beta,X}$. Moreover, by the projection formula, $j_*(\sigma_{-\alpha_i}\cup \eta_j)=\sigma_{-\alpha_i,X}\cup \eta_j^X$. Now, from the Chevalley formula for $X$, we know that the multiplication by the hyperplane class $\eta_1^X$ from the $\dim(X)-1$-cohomology of $X$ to the $\dim(X)+1$-cohomology is given by the matrix $4I - C_\aleph$. Thus, by inverting this matrix, one gets that $j_*j^*((4I-C_\aleph)^{-1}\sigma_{-\alpha_i,X}\cup \eta_j^X)=\sigma_{-\alpha_i,X}\cup \eta_j^X=j_*(\sigma_{-\alpha_i}\cup \eta_j)$. Therefore we easily obtain that $\sigma_{-\alpha_i}\cup \eta_j=j^*((4I-C_\aleph)^{-1}\sigma_{-\alpha_i,X}\cup \eta_j^X)$, and the RHS is computable in the cohomology of $X$.

In order to prove that these relations generate the ideal $I$, consider an element $R$ in $I$. Modding out by the second and third set of relations, $R$ can be rewritten as $P+\sum_i a_i \sigma_{-\alpha_i}$ with $P\in \QQQ[\eta_1,\cdots,\eta_n]$ and $a_i\in \QQQ$ for $1\leq i \leq n$. Since this element $R$ vanishes in the cohomology of $Y$, we have that $a_i=0$ for all $1\leq i \leq n$ and, by Lemma \ref{lemrelationsfromX}, $P\in j^*\cJ$.
\end{proof}

As an application let us explicitly give a presentation of the classical cohomology for hyperplane sections of the adjoint and coadjoint variety of $F_4$. We will use the presentation of cohomology of $X$ as it appears in \cite{adjoint} (for Giambelli formulas, we will refer to its companion file \cite{adjoint2}). Notice that in both cases we have
$$A= 0, \ 4I-C_\aleph=\begin{pmatrix} 2 & 1 \\ 1 & 2 \end{pmatrix} \textrm{ and } (4I-C_\aleph)^{-1}=\begin{pmatrix} 2/3 & -1/3 \\ -1/3 & 2/3 \end{pmatrix}.$$
We will denote by $\alpha_1,\alpha_2,\alpha_3,\alpha_4$ the simple roots of $F_4$, following Bourbaki's convention \cite{bourbaki}.

\medskip
{\noindent}$\bf{F_4}${\bf-adjoint:} Let $X=F_4/P_1$ be the adjoint variety of $F_4$, and let $Y\subset X\subset \PP(\ff_4)$ be a general hyperplane section. A presentation of the cohomology ring of $X$ is given by
\[
\HHH^*(X) = \QQQ[h_X,s_X]/\cJ \mbox{, with }\cJ=(h_X^8 -12s_X^2,3h_X^{12} -18h_X^8s_X + 24h_X^4s_X^2 + 8s_X^3),
\]
where $h_X$ is the hyperplane class of degree one and $s_X$ is a class of degree four. By the Giambelli formulas we have $\sigma_{\alpha_1,X}=3/8 h_X^7 - 5/4 h_X^3s_X$ and $\sigma_{\alpha_2,X}=-1/8 h_X^7 + 1/2 h_X^3s_X$. After some computations we obtain that \[
\HHH^*(Y) = \QQQ[h,s,\sigma_{-\alpha_1},\sigma_{-\alpha_2}]/\begin{pmatrix}j^*\cJ, \\
24\sigma_{-\alpha_1}h-5h^8+24h^4s,24\sigma_{-\alpha_2}h+4h^8-18h^4s,\\
24\sigma_{-\alpha_1}s-5h^7s+24h^3s^2,24\sigma_{-\alpha_2}s+4h^7s-18h^3s^2, \\
\sigma_{-\alpha_1}^2,\sigma_{-\alpha_2}^2,\sigma_{-\alpha_1}\sigma_{-\alpha_2}\end{pmatrix}.
\]

\medskip
{\noindent}$\bf{F_4}${\bf-coadjoint:} Let $X=F_4/P_4$ be the coadjoint variety of $F_4$, and let $Y\subset X$ be a general hyperplane section. A presentation of the cohomology ring of $X$ is given by
\[
\HHH^*(X) = \QQQ[h_X,s_X]/\cJ \mbox{, with }\cJ=(2h_X^8 - h_X^4s_X +3s_X^2,-11h_X^{12} +26h_X^8s_X + 24h_X^4s_X^2),
\]
where $h_X$ is the hyperplane class of degree one and $s_X$ is a class of degree four. By the Giambelli formulas we have $\sigma_{\alpha_3,X}=h_X^7 - 2 h_X^3s_X$ and $\sigma_{\alpha_4,X}=-2 h_X^7 + 5 h_X^3s_X$. After some computations we obtain that \[
\HHH^*(Y) = \QQQ[h,s,\sigma_{-\alpha_3},\sigma_{-\alpha_4}]/\begin{pmatrix}j^*\cJ, \\
3\sigma_{-\alpha_3}h+9h^8,3\sigma_{-\alpha_4}h-5h^8+12h^4s,\\
3\sigma_{-\alpha_3}s+9h^7s,3\sigma_{-\alpha_4}s-5h^7s+12h^3s^2,\\
\sigma_{-\alpha_3}^2,\sigma_{-\alpha_4}^2,\sigma_{-\alpha_3}\sigma_{-\alpha_4}\end{pmatrix}.
\]

\section{Towards quantum cohomology}

In this section we describe moduli spaces of stables maps to $Y$ and give techniques to compute Gromov-Witten invariants on $Y$. In particular, we compute a quantum Chevalley formula and prove results on the semi-simplicity of the quantum cohomology of $Y$.

\subsection{Moduli spaces of curves}

In this subsection we prove basic results on moduli spaces of stable maps to $Y$ a general hyperplane section of a (co)adjoint variety $X$. We start with general results on rational curves on linear sections of homogeneous spaces.

\subsubsection{Linear sections of homogeneous spaces}

Let $Z$ be a smooth projective variety, fix a non degenerate embedding $Z \subset \PP(V)$ and set $\cO_Z(1) = \cO_{\PP(V)}(1)\vert_Z$ so that $V^\vee = H^0(Z,\cO_Z(1))$. For $\eta \in H_2(Z,\ZZ)$, set $d = \langle \cO_Z(1),\eta \rangle$ where $\langle - , - \rangle$ is the pairing between divisors and curves on $Z$. Let $\overline{M}_{0,n}(Z,\eta)$ be the moduli space of genus zero $n$-pointed stable maps to $Z$ of class $\eta$. Recall that the expected dimension of $\overline{M}_{0,n}(Z,\eta)$ is
$${\rm expdim}(\overline{M}_{0,n}(Z,\eta)) = \langle - K_X,\eta \rangle + \dim Z + n - 3$$
and that any irreducible component of $\overline{M}_{0,n}(Z,\eta)$ has dimension at least ${\rm expdim}(\overline{M}_{0,n}(Z,\eta))$.

For all $i \in [1,n]$, we have evaluation maps ${\rm ev}_i : \overline{M}_{0,n}(Z,\eta) \to Z$ and forgetful maps $\pi : \overline{M}_{0,n+1}(Z,\eta) \to \overline{M}_{0,n}(Z,\eta)$. These maps induce the following diagram with $e := {\rm ev}_{n+1}$. We denote by $\cF := \pi_*e^*\cO_Z(1)$.
$$\xymatrix{\overline{M}_{0,n+1}(Z,\eta) \ar[r]^-e \ar[d]_-\pi & Z \\
  \overline{M}_{0,n}(Z,\eta). \\}$$

\begin{lemma}
  The sheaf $\cF$ is locally free of rank $d + 1$ and $R^i\pi_*e^*\cO_Z(1) = 0 $ for $i > 0$.
\end{lemma}

\begin{proof}
  Let $\overline{f} := (f : C \to Z,p_1,\cdots,p_n) \in \overline{M}_{0,n}(Z,\eta)$. We have an isomorphism $C \simeq \pi^{-1}(\overline{f})$ and $e^*\cO_Z(1)\vert_{\pi^{-1}(\overline{f})} = f^*\cO_Z(1)$. Furthermore, since $\cO_Z(1)$ is globally generated, so is $f^*\cO_Z(1)$. Since $C$ is rational, we have $R^i\pi_*e^*\cO_Z(1) = 0$ for $i > 0$ and $\cF$ is locally free of rank $\dim H^0(C,f^*\cO_Z(1))$. The result follows.
\end{proof}

Note that we have a map $H^0(Z,\cO_Z(1)) \otimes \cO_{\overline{M}_{0,n}(Z,\eta)} \to \cF$ whose restriction to the stalk at $\overline{f} = (f : C \to Z,p_1,\cdots,p_n) \in \overline{M}_{0,n}(Z,\eta)$ is given by $H^0(Z,\cO_Z(1)) \to H^0(C,f^*\cO_Z(1))$.

\begin{definition}
A map $\overline{f} = (f:C\to Z,p_1,\cdots,p_n) \in \overline{M}_{0,n}(Z,\eta)$ is called non-degenerate if the restriction map $H^0(Z,\cO_Z(1)) \to H^0(C,f^*\cO_Z(1))$ is surjective.
\end{definition}

Note that $\overline{f}$ is non-degenerate if and only if the map $H^0(Z,\cO_Z(1)) \otimes \cO_{\overline{M}_{0,n}(Z,\eta)} \to \cF$ is surjective at $\overline{f}$. In particular, the locus of non-degenerate stable maps is open. Let $\overline{M}_{\rm deg}(Z) \subset \overline{M}_{0,n}(Z,\eta)$ be the locus of maps which are degenerate.

\begin{remark}
  Let $\overline{f} = (f : C \to Z,p_1,\cdots,p_n) \in \overline{M}_{0,n}(Z,\eta)$.
  \begin{enumerate}
  \item Assume that $d = 1$, then $f(C)$ is a line in $\PP(V)$ and $\overline{f}$ is always non-degenerate.
  \item Assume that $d = 2$, then $\overline{f}$ is non-degenerate except if $f(C)$ is a line and $f : C \to f(C)$ is a double cover. Indeed, otherwise $f(C)$ is a conic in $\PP(V)$ and spans a plane.
  \end{enumerate}
\end{remark}

Let $j : Z' \subset Z$ be a general linear section of codimension $r$ such that $j_* : H_2(Z',\ZZ) \to H_2(Z,\ZZ)$ is an isomorphism (for example $\dim(Z') \geq 3$). We may thus consider $\eta$ as an element of $H_2(Z',\ZZ)$.

\begin{proposition}
  \label{prop-mor-gen}
  Assume that $\overline{M}_{0,n}(Z,\eta)$ is of expected dimension and let $Z' \subset Z$ be a general linear section of codimension $r$.
\begin{enumerate}
\item Any irreducible component of $\overline{M}_{0,n}(Z',\eta) \setminus \overline{M}_{\rm deg}(Z')$ is of expected dimension.
\item If $\overline{M}_{0,n}(Z,\eta)$ is smooth or has rational singularities, then so is $\overline{M}_{0,n}(Z',\eta) \setminus \overline{M}_{\rm deg}(Z')$.
  \end{enumerate}
\end{proposition}

\begin{proof}
We may work over $\overline{M}_{0,n}(Z,\eta) \setminus \overline{M}_{\rm deg}(Z)$. Let $\varphi \in \Hom(\cO_Z^r,\cO_Z(1))$ be the map defining $Z'$. Then $\varphi$ induces a map $\cO^r \to \cF$ and $\overline{M}_{0,n}(Z',\eta)$ is the vanishing locus of this map in $\overline{M}_{0,n}(Z,\eta)$. Since $\cF$ is globally generated on $\overline{M}_{0,n}(Z,\eta) \setminus \overline{M}_{\rm deg}(Z)$, the result follows from classical results on sections of globally generated vector bundles that we recall in Lemma \ref{lem-gg}.
\end{proof}

\begin{lemma}
  \label{lem-gg}
  Let $M$ be an irreducible variety and $\cF$ be a globally generated vector bundle of rank $d+1$. Assume that the zero locus $V(\varphi)$ of a general map $\varphi : \cO_M^r \to \cF$ is non empty, then  $V(\varphi)$ is equidimensional of dimension $\dim V(\varphi) = \dim M - r(d+1)$. 
  \begin{enumerate}
  \item If $M$ is smooth, then $V(\varphi)$ is smooth.
        \item If $M$ has rational singularities, then $V(\varphi)$ has rational singularities.
  \end{enumerate}
\end{lemma}

\begin{proof}
  We include a short proof for the convenience of the reader. We identify the sheaf $\cF$  with the associated vector bundle $\cF \to M$ over $M$. By assumption, the vector bundle map
  $$p : H^0(M,\cF^r) \times M \to \cF^r$$
  is surjective and linear of maximal rank on the fibers over $M$. Let $M_0 \subset \cF$ be the zero section which is isomorphic to $M$ and therefore smooth. The inverse image $p^{-1}(M_0)$ is therefore also smooth and of dimension $\dim M + \dim H^0(M,\cF^r) - r(d+1)$. We have a map
  $$\pi : p^{-1}(M_0) \to H^0(M,\cF^r)$$
  given by the first projection and for $\varphi \in \Hom(\cO_M^r,\cF) = H^0(M,\cF^r)$, the vanishing locus $V(\varphi)$ is the fiber $\pi^{-1}(\varphi)$. By assumption, this map is dominant. The first assertion follows from the generic dimension of the fibers of $\pi$.

  (1) Follows from generic smoothness of the fibers of $\pi$.

  (2) Follows from a similar general rational smoothness result of Brion, see \cite[Lemma 3]{brion-pos}.
\end{proof}

\begin{remark}
  In items 1. and 2. of Proposition \ref{prop-mor-gen}, we may have $\overline{M}_{0,n}(Z',\eta) = \overline{M}_{\rm deg}(Z')$. In particular, if $r \geq \dim H^0(Z,\cO_Z(1)) - d$ this is always the case. Indeed, if $\overline{f} = (f : C \to Z,p_1,\cdots,p_n)$ is a non degenerate map and $r \geq H^0(Z,\cO_Z(1)) - d$, there exists no map $\varphi : \cO_Z^r \to \cO_Z(1)$ with $\dim(\varphi(H^0(Z,\cO_Z^r)))=r$ such that the composition $H^0(Z,\cO_Z^r) \stackrel{\varphi}{\to } H^0(Z,\cO_Z(1)) \to H^0(C,f^*\cO_Z(1))$ vanishes.
\end{remark}

We now apply the above results to linear sections of rational projective homogeneous spaces. We first recall few facts on $\overline{M}_{0,n}(Z,\eta)$ for $Z$ a projective rational homogeneous space.

\begin{theorem}[See \cite{thomsen},\cite{KP},\cite{rat}]
  For $Z$ a projective rational homogeneous space and $\eta$ effective, the moduli space $\overline{M}_{0,n}(Z,\eta)$ is irreducible, rational and has rational singularities.
\end{theorem}

\begin{corollary}
  \label{cor-dim-section}
  Let $Z \subset \PP(V)$ be a projective rational homogeneous space with $\Pic(Z) = \ZZ \cO_Z(1)$. Write $\omega_Z^{-1} = \cO_Z(c_1(Z))$ with $c_1(Z) > 0$. Let $\eta \in H_2(Z,\ZZ)$ an effective class and set $d = \langle \cO_Z(1),\eta \rangle$. Let $Z' \subset Z$ be a general linear section of codimension $r < \dim Z$ in $Z$.
\begin{enumerate}
\item If $d = 1$ then any irreducible component of $\overline{M}_{0,n}(Z',\eta)$ is smooth of expected dimension.
  \item If $d = 1$ then  $\overline{M}_{0,n}(Z',\eta)$ is non-empty for $r \leq c_1(Z) - 2$ and irreducible for $r < c_1(Z) - 2$.
\item If $d = 2$ and $r \leq c_1(Z) - 2$, then $\overline{M}_{0,n}(Z',\eta)$ is non-empty of expected dimension.
\item If $d = 2$ and $r < c_1(Z) - 2$, then $\overline{M}_{0,n}(Z',\eta) \setminus \overline{M}_{\rm deg}(Z')$ is non empty, has rational singularities and is dense in $\overline{M}_{0,n}(Z',\eta)$.
\end{enumerate}
\end{corollary}

\begin{proof}
  (1) For $d = 1$, there is no degenerate map and the moduli space of maps to homogeneous spaces is smooth by a result of Landsberg and Manivel, see \cite{LM}. This proves the result.

  (2) We may assume $n = 1$. Consider the map ${\rm ev}_{1} : \overline{M}_{0,1}(Z,\eta) \to Z$ whose fiber $F_z$ over $z \in Z$ is a smooth projective variety (by Landsberg and Manivel \cite{LM}, in this paper the set of lines in a variety $X$ through a point $x$ is denoted by $Base|\FF\FF^2_{X,x}|$) of dimension $c_1(Z) - 2$. For $z \in Z'$, the subset of elements in $F_z$ that are contained in $\overline{M}_{0,1}(Z',\eta)$ is given by a linear section of codimension $r$ in $F_z$ (whose embedding is given by the projective space of lines through $z$). The fiber is non-empty as soon as $r \leq c_1(Z) - 2$. Furthermore, for $r < c_1(Z) - 2$, the fiber is connected by \cite{FH} and thus $\overline{M}_{0,1}(Z',\eta)$ is connected and since it is smooth by (1), we get the irreducibility.

  For (3) and (4), we first remark that $\overline{M}_{0,n}(Z',\eta)$ is non empty since it contains the locus of double maps to a line and since the space of lines is non-empty by (2). We compute the dimension $\dim \overline{M}_{\rm deg}(Z')$ of degenerate maps which is given by the dimension of double covers to lines in $Z'$. By (1), we get $\dim \overline{M}_{\rm deg}(Z') = \dim Z + c_1(Z) + n - 3 + 2 - 2r \leq {\rm expdim} \overline{M}_{0,n}(Z',\eta)$ with a strict inequality for $r < c_1(Z) - 2$. This proves the result, noting that in (4) the degenerate maps have a too small dimension to form an irreducible component, while in (3) for $r = c_1(Z) - 2$, the space $\overline{M}_{\rm deg}(Z')$ will form an irreducible component.
\end{proof}

\begin{remark}
  We discuss on the condition $r < c_1(Z) -2$ in (2) and (4) above.
  \begin{enumerate}
  \item The conditions $r \leq c_1(Z) -2$ and $r < c_1(Z) - 2$ are not sharp in (2) above. Indeed, consider $Z = \Gr(2,5)$ and $Z' \subset Z$ be a general linear section of codimension $4$. Then $Z'$ is a del Pezzo surface of degree $5$ and does contain lines. If we consider $Z' \subset Z$ a general linear section of codimension $3$, then $Z'$ is the Fano threefold of degree $5$ and $r = 3 = c_1(Z) - 2$. However, it is well known that the Fano variety of lines in $Z'$ is isomorphic to $\PP^2$ and therefore irreducible (see for instance \cite{KPS}).
\item The condition $r < c_1(Z) - 2$ is sharp in (4) above. Indeed an easy dimension count shows that $\dim \overline{M}_{\rm deg}(Z') = {\rm expdim} \overline{M}_{0,n}(Z',\eta)$ so that degenerate stable maps form an irreducible component. For example, consider again $Z = \Gr(2,5)$ and $Z' \subset Z$ the general linear section of codimension $3$ which is the Fano threefold of degree $5$. It was proved in \cite{FGP} that for $d = 2$, the space $\overline{M}_{0,n}(Z',\eta)$ has two irreducible components, both of expected dimension, one of which is formed by degenerate conics.
  \end{enumerate}
\end{remark}

\subsubsection{Lines in $Y$} We now apply the above results to a general hyperplane section $Y$ in an adjoint or a quasi-minuscule variety $X$. We start with basic results on lines. 

\begin{proposition}
  \label{prop:m(Y,1)}
  Let $X$ be an adjoint or a quasi-minuscule variety with $\Pic(X) = \ZZ$ and let $Y \subset X$ be a general hyperplane section. Let $\eta \in H_2(X,\ZZ)$ with $d = \langle \cO_X(1) , \eta \rangle = 1$.

Then $\overline{M}_{0,n}(Y,\eta)$ is smooth irreducible of expected dimension.
\end{proposition}

\begin{proof}  
If $X$ is a projective space or a quadric, then all the assertions are true, so we may assume that $X$ is neither a projective space nor a quadric.

If $X$ is not adjoint of type $G_2$, we have $c_1(X) \geq 4$ and the result follows from Corollary \ref{cor-dim-section}. For $X$ adjoint of type $G_2$, then $c_1(X) = 3$ and we need to use a different technique. Note that the Fano variety of lines is $F(X) = G_2/P_2$ and is a smooth quadric. Furthermore, our result on $T$-stable curves prove that all $T$-fixed points in $F(X)$ also lie in $F(Y)$, the Fano variety of lines in $Y$. Now since $\cF$ is globally generated on all Bia{\l}ynicki-Birula cells, the restriction of the open cell of $F(X)$ is open and dense in $F(Y)$ and since it contains a unique $T$-fixed point it must be connected, proving the result. Note that this proof works in all (co)adjoint cases. 
\end{proof}

We can give more precise results on lines. Let $F(X) $ denote the Fano variety of lines inside $X$, and let $F(Y)$ denote the Fano variety of lines inside $Y$. Note that we have $F(X) = \overline{M}_{0,0}(X,\eta)$ and $F(Y) = \overline{M}_{0,0}(Y,\eta)$ where $\langle \eta , \cO_X(1) \rangle = 1$ (in type $A$ there are two families of lines). Clearly $F(X)$ admits an action of $G$, induced by the action of $G$ on $X$. It turns out that $F(X)$ is a homogeneous variety when $X$ is an adjoint variety; in \cite{LM}, a complete description of $F(X)$ when $X$ is adjoint is given (the result in \cite{LM} also applies in a more general case than $X$ adjoint). In Table \ref{tableFanoadjoint} we report the list of Fano varieties $F(X)$ when $X$ is adjoint. From now on, we will leave aside the $C_n$ case because the geometry of $\PP^{2n-1}$ is well known. Notice that in the $A_n$ case we reported the two families of curves corresponding to the two generators of $\HHH_2(X,\ZZ)$.

\begin{table}[ht]
\centering
   \begin{tabular}{cccc}
      G & $\mathfrak{g}$ & Adjoint variety & Fano variety of lines \\
\hline
$A_n=\SL_{n+1}$ & $\mathfrak{sl_{n+1}}$ & ${\rm Fl}(1,n ; n+1)$ & ${\rm Fl}(2,n;n+1)\cup {\rm Fl}(1,n-1;n+1)$ \\
	  $B_n=\SO_{2n+1}$ & $\mathfrak{so_{2n+1}}$ & $\OGr(2,2n+1)$ & $\OF(1,3;2n+1)$ \\
	  $D_n=\Spin_{2n}$ & $\mathfrak{so_{2n}}$ & $\OGr(2,2n)$ & $\OF(1,3;2n)$ \\
	  $G_2$ & $\mathfrak{g_{2}}$ & $G_2/P_1$ & $G_2/P_2 = \QQ_5$ \\
	  $F_4$ & $\mathfrak{f_{4}}$ & $F_4/P_1$ & $F_4/P_2$ \\
	  $E_6$ & $\mathfrak{e_{6}}$ & $E_6/P_2$ & $E_6/P_4$ \\
	  $E_7$ & $\mathfrak{e_{7}}$ & $E_7/P_1$ & $E_7/P_3$ \\
	  $E_8$ & $\mathfrak{e_{8}}$ & $E_8/P_8$ & $E_8/P_7$ \\
   \end{tabular}
   \caption{\label{tableFanoadjoint} Fano varieties of lines of adjoint varieties}
\end{table}

In the coadjoint situation, the interesting cases are given by the $C_n$ and the $F_4$ groups, as already remarked in the previous sections; the Fano variety of lines, even though it admits an action of $G$, is not homogeneous in these two cases.
\begin{itemize}
\item $X=\IGr(2,2n)$: the Fano variety $F(X)$ is a subvariety of $F(1,3,2n)$, that is the Fano variety of lines of $\Gr(2,2n)$. As a subvariety of $F(1,3,2n)$, the variety $F(X)$ is the zero locus of a general section of the vector bundle $\cU_1^*\otimes (\cU_3/\cU_1)^*$; indeed, a general section of this bundle is a non degenerate skew-symmetric form, and the condition that defines the zero locus is exactly the one that ensures that the corresponding line is contained inside $\IGr(2,2n)$. Notice that $F(X)$ contains the homogeneous variety $\IF(1,3,2n)$ as its hyperplane section.

\item $X=F_4/P_4$: in this case, recall that $X$ is a hyperplane section of $E_6/P_6\subset \PP(J_3(\OO))$. Moreover, by \cite{LM}, the Fano variety of lines of $E_6/P_6$ is $F(E_6/P_6)=E_6/P_5\subset \Gr(2,J_3(\OO))$. Thus, $F(X)\subset E_6/P_5$ is the zero locus of a general section of $\cU_2^*|_{E_6/P_5}$.
\end{itemize}

\subsubsection{Conics in $Y$} We turn to the moduli space of stable maps of degree $2$.

\begin{theorem}  \label{thm:m(Y,2)}
  Let $X$ be an adjoint or a quasi-minuscule variety and let $Y \subset X$ be a general hyperplane section. Let $\eta \in H_2(X,\ZZ)$ with $d = \langle \cO_X(1) , \eta \rangle = 2$.
  \begin{enumerate}
  \item If $X$ is not the adjoint variety of type $G_2$, then $\overline{M}_{0,n}(Y,\eta)$ is irreducible with rational singularities outside of $\overline{M}_{\rm deg}(Y)$ and of expected dimension.
  \item If $X$ is the adjoint variety of type $G_2$, then $\overline{M}_{0,n}(Y,\eta)$ has two irreducible components, both of expected dimension, one of which is formed by degenerate conics.
  \end{enumerate}
\end{theorem}

\begin{proof}
  We may assume that $X$ is neither a projective space nor a quadric. In particular $c_1(X) \geq 3$ and Corollary \ref{cor-dim-section} implies that $\overline{M}_{0,n}(Y,\eta)$ is of expected dimension $\dim Y + 2c_1(Y) + n - 3$.

  Note that degenerate conics are given by double covers of lines. In particular $\dim \overline{M}_{\rm deg}(Y) = \dim Y + c_1(Y) + n - 1$ and in all cases except if $c_1(X) = 3$ degenerate curves will not form an irreducible component. The only case for which $c_1(X) = 3$ is the adjoint variety of type $G_2$.

  We are left to prove that $\overline{M}_{0,n}(Y,\eta) \setminus \overline{M}_{\rm deg}(Y)$ is irreducible. Since it has rational singularities by Corollary \ref{cor-dim-section}, it is enough to prove that it is connected. Set $M_Y = \overline{M}_{0,n}(Y,\eta) \setminus \overline{M}_{\rm deg}(Y)$ and $M_X = \overline{M}_{0,n}(X,\eta) \setminus \overline{M}_{\rm deg}(X)$ and recall that $M_Y$ is the zero locus of a rank $3$ vector bundle $\cF$ on $M_X$. Recall also (e.g. from \cite[Remark 6.2.18]{lazarsfeld2}) that $\cF$ is $k$-ample if $\cO_{\PP_{M_X}(\cF)}(m)$ is globally generated for some $m > 0$ and if the map $\PP_{M_X}(\cF) \to \PP(H^0(\PP_{M_X}(\cF),\cO_{\PP_{M_X}(\cF)}(m))$ has fibers of dimension at most $k$. A result of Tu (\cite{tu}, see also \cite[Remark 7.2.11]{lazarsfeld2}) asserts that if $\cF$ is $k$-ample, the zero locus of a general section of $\cF$ is connected if $\dim M_X > \rk(\cF) + k$. We claim that $\cF$ is $2c_1(X) + n - 2$ ample and since $\dim M_X - \rk(\cF) - (2c_1(X) + n - 2) = \dim X - 4 > 0$ the result will follow from this claim.

  To prove the claim, first note that by construction $\cF = \pi_*{\rm ev}^*\cO_{X}(1)$. For $f \in M_X$ let $C_f$ be the conic image of $f$ and $\langle C_f \rangle$ be the plane generated by this conic (this is indeed a plane since $f$ is non denegerate). We thus have $\PP_{M_X}(\cF) = \{ (f,x) \ | \ f \in M_X, x \in \langle C_x \rangle \}$, the line bundle $\cO_{\PP_{M_X}(\cF)}(1)$ is globally generated and the map $\Phi : \PP_{M_X}(\cF) \to \PP(H^0(\PP_{M_X}(\cF),\cO_{\PP_{M_X}(\cF)}(1))$ is simply given by $(f,x) \mapsto x$. Note that the map $\Phi$ is $G$-equivariant and that its image contains $X$ as unique closed $G$-orbit (in fact $X$ is the unique closed $G$-orbit in $\PP(V)$). In particular, the dimension of the fibers $\Phi^{-1}(x)$ is maximal for $x \in X$.  

  Let $x \in X$, then $\Phi^{-1}(x) = \{f \in M_X \ | \ x \in \langle C_f \rangle \}$. Let $f \in \Phi^{-1}(x)$ and $P = \langle C_f \rangle$. If $x \not\in C_f$, then $P\cap X$ contains $x$ and the conic $C_f$. But since $X$ is cut out by quadrics, this implies that $P \subset X$. In particular $\Phi^{-1}(x) = \{ f \in M_x \ | \ x \in C_f \} \cup \{f \in M_X \ | \ x \in \langle C_f \rangle \subset X\}$. By a result of Landsberg and Manivel \cite{LM} the dimension of planes contained in $X$ is known and the second part of this union has dimension smaller than $2c_1(X) + n - 2$. The first part is the fiber of the evaluation map and has dimension $2c_1(X) + n - 2$ proving the claim.
\end{proof}

\subsection{Comparing Gromov-Witten invariants in $X$ and $Y$}

In this section we compare Gromov-Witten invariants in $Y$ with Gromov-Witten invariants in $X$. Note that for Gromov-Witten invariant of degree $1$ or $2$, since the moduli space $\overline{M}_{0,n}(Y,\eta)$ has expected dimension, the virtual class is the fundamental class and therefore, for $\sigma_X,\sigma'_X,\sigma''_X \in H^*(X,\ZZ)$ and $\sigma_Y,\sigma'_Y,\sigma''_Y \in H^*(Y,\ZZ)$, we have by definition:
$$\langle \sigma_X , \sigma'_X , \sigma''_X \rangle_\eta^X = \int_{\overline{M}_{0,3}(X,\eta)} {\rm ev}_1^*\sigma_X \cup {\rm ev}_2^*\sigma'_X \cup {\rm ev}_3^*\sigma''_X \textrm{ and }$$
$$\langle \sigma_Y , \sigma'_Y , \sigma''_Y \rangle_\eta^Y = \int_{\overline{M}_{0,3}(Y,\eta)} {\rm ev}_1^*\sigma_Y \cup {\rm ev}_2^*\sigma'_Y \cup {\rm ev}_3^*\sigma''_Y.$$

\begin{lemma}
  \label{lem:comp-GW}
  Let $j : Y \to X$ be the inclusion and let $\eta \in H_2(X,\ZZ)$ with $\langle \eta , \cO_X(1) \rangle = 1$.
  
  Let $Z,Z' \subset Y$ be two closed subvarieties such that $Z \cap Z' = \emptyset$. Then for any $\tau_X \in H^*(X,\ZZ)$ we have $\langle [Z] , [Z'] , j^*\tau_X \rangle_\eta^Y = \langle j_*[Z] , j_*[Z'] , \tau_{X} \rangle_\eta^X$.
\end{lemma}

\begin{proof}
Consider the following diagram:
$$\xymatrix{\overline{M}_{0,2}(Y,\eta) \ar[d]_-{{\rm ev}_{1,2}}
    \ar[r]^-i & \overline{M}_{0,2}(X,\eta) \ar[d]^-{{\rm ev}_{1,2}} \\
    Y^2 \ar[r]^-j & X^2.}$$
Since any line in $X$ meeting $Y$ in two points is contained in $Y$, we have the equality $i({\rm ev}_{1,2}^{-1}(Z\times Z')) = {\rm ev}_{1,2}^{-1}(j(Z) \times j(Z'))$. Since furthermore both maps ${\rm ev}_{1,2}$ are flat outside the diagonal \emph{i.e.} over  $X^2 \setminus \Delta_X$ and $Y^2 \setminus \Delta_Y$ (in fact the above square is Cartesian away from the diagonals), we get the equalities:
$$i_*{\rm ev}_{1,2}^*([Z] \times [Z']) = [{\rm ev}_{1,2}^{-1}(Z\times Z')] = [{\rm ev}_{1,2}^{-1}(j(Z) \times j(Z'))] = ev_{1,2}^*(j_*[Z] \times j_*[Z']).$$
We thus have
$$\begin{array}{ll}
  \langle [Z] , [Z'] , j^*\tau_X \rangle_\eta^Y & = \int_{\overline{M}_{0,3}(Y,\eta)} {\rm ev}_{1,2}^*([Z] \times [Z']) \cup {\rm ev}_3^*j^*\tau_X \\
  & = \int_{\overline{M}_{0,3}(Y,\eta)} {\rm ev}_{1,2}^*([Z] \times [Z']) \cup i^*{\rm ev}_3^*\tau_X \\
  & = \int_{\overline{M}_{0,3}(X,\eta)} {\rm ev}_{1,2}^*j_*([Z] \times [Z']) \cup {\rm ev}_3^*\tau_X \\
  & = \langle j_*[Z] , j_*[Z'] , \tau_{X} \rangle_\eta^X,\\
\end{array}
$$
  proving the result.
\end{proof}

\begin{proposition}
  \label{prop:comp-GW}
  Let $\sigma_Y,\sigma'_Y \in H^*(Y,\ZZ)$ and $\tau_X \in H^*(X,\ZZ)$ with $\deg(\tau_X) < c_1(Y) = c_1(X) - 1$. Let $j : Y \to X$ be the inclusion and let $\eta \in H_2(X,\ZZ)$ with $\langle \eta , \cO_X(1) \rangle = 1$. Then we have $\langle \sigma_{Y} , \sigma'_{Y} , j^*\tau_X \rangle_\eta^Y = \langle j_*\sigma_{Y} , j_*\sigma'_{Y} , \tau_{X} \rangle_\eta^X$.
\end{proposition}

\begin{proof}
  First note that the degree conditions on both sides agree so that these invariants can only be simultaneously non-zero. Note also that if these invariants do not vanish, then $\deg(\sigma_Y) + \deg(\sigma'_Y) = \dim Y + c_1(Y) - \deg(\tau_X) > \dim Y$. In particular, for $d = \deg(\sigma_Y)$ and $d' = \deg(\sigma'_Y)$ fixed, the basis $([Y_\alpha])_\alpha$ of classes of Schubert varieties in $H^{2d}(Y,\ZZ)$ and the basis $([Y_\beta^-])_\beta$ of classes of Schubert varieties in $H^{2d'}(Y,\ZZ)$ are such that $Y_\alpha \cap Y_\beta^- = \emptyset$. We may therefore assume that $\sigma_Y = [Z]$ and $\sigma_Y' = [Z']$ are effective classes with $Z \cap Z'  = \emptyset$. The result now follows from Proposition \ref{prop:comp-GW}.
 \end{proof}

\subsection{Quantum Chevalley formula}

For simplicity of the exposition, we assume that $X$ is adjoint and quasi minuscule with $\Pic(X) = \ZZ$. The only case not treated is therefore the adjoint variety of type $A$. We will give a description of $\QH(Y)$ for $Y \subset X$ a general hyperplane section for $X$ adjoint of type $A$ in the appendix.

We first prove that the previous comparison result together with a unique invariant in degree $2$ for adjoint varieties gives a complete formula for the quantum multiplication by $h$, the hyperplane class in $Y$. Let $[{\rm pt}]$ be the class of a point in $Y$ and $[{\rm line}]$ be the class of a line in $Y$. 

\begin{lemma}
  \label{lem:deg2}
  Assume that $X$ is adjoint or quasi-minuscule and let $h$ be the hyperplane class in $Y$. Then the only non-vanishing Gromov-Witten invariants of the form $\langle \sigma_Y , \sigma'_Y , h \rangle_\eta^Y$ are obtained for $\langle \eta ,\cO_X(1) \rangle = 1$ or for the invariant $\langle [{\rm pt}] , [{\rm line}] , h \rangle_\eta^Y$ for $X$ adjoint and $\langle \eta ,\cO_X(1) \rangle = 2$.
\end{lemma}

\begin{proof}
  Set $d = \langle \eta , \cO_X(1) \rangle$. The Gromov-Witten invariant $\langle \sigma_Y , \sigma'_Y , h \rangle_\eta^Y$ vanishes unless we have $\deg \sigma_Y + \deg \sigma'_Y + \deg h = \dim Y + dc_1(Y)$. In particular $d c_1(Y) \leq \dim Y + 1$.

  For $X$ coadjoint non-adjoint, this implies $d \leq 1$. If $X$ is adjoint then $\dim Y = 2 c_1(Y)$, we get $dc_1(Y) \leq 2c_1(Y) +1$ and therefore $d \leq 2$. Furthermore, for $d = 2$, we get $\deg(\sigma_Y) + \deg(\sigma'_Y) = 2 \dim Y - 1$. The only possibilities (up to exchanging $\sigma_Y$ and $\sigma'_Y$) is therefore obtained for $\deg(\sigma_Y) = \dim Y$ and $\deg(\sigma'_Y) = \dim Y - 1$, proving the result.
\end{proof}

\begin{proposition}
  \label{prop:deg2}
  For $X$ adjoint and $\eta \in H_2(X,\ZZ)$ with $\langle \eta ,\cO_X(1) \rangle = 2$, we have
  $$\langle [{\rm pt}] , [{\rm line}] , h \rangle_\eta^Y = 2.$$
\end{proposition}

\begin{proof}
  Note that $\langle [{\rm pt}] , [{\rm line}] , h \rangle_\eta^Y = 2 \langle [{\rm pt}] , [{\rm line}] \rangle_\eta^Y$ and this last invariant is the number of conics passing through a given point and meeting a given line.
  
  Let $x = x_\Theta = [v_\Theta]$ and $y = x_{-\Theta} = [v_{-\Theta}]$. The unique conic in $X$ passing through $x$ and $y$ is explicitly given by $[s^2 v_\Theta + st [v_\Theta,v_{-\Theta}] + t^2 v_{-\Theta}]$ (this is also the closure of $\exp(s \ad_{v_\Theta}) \cdot x_{-\Theta}, s \in \CC$).

  Consider a general line $L$ passing through $y$ in $X$. Note that the space of all lines through $y$ is homogeneous under a Levi factor of $G$, so that $L = \SL_2 \cdot y$ for some subgroup $\SL_2 \subset G$. In particular there exists $z \in \fb$ with $\ad_z^2 = 0$ such that $[\exp(u \ad_z) \cdot v_{-\Theta}] \in L$ and $[\exp(u \ad_z) \cdot v_{\Theta}] = v_\Theta$ for all $u \in \CC$. The locus $S$ covered by all conics in $X$ passing through $x_\Theta$ and meeting $L$ is therefore the closure of $\{ [\exp(u \ad_z)\cdot (s^2v_\Theta + st [v_\Theta,v_{-\Theta}] + t^2 v_{-\Theta})] \ | \ s,t,u \in \CC \}$. A computation gives
  $$\exp(u \ad_z)\cdot (s^2v_\Theta + st [v_\Theta,v_{-\Theta}] + t^2 v_{-\Theta}) = s^2 v_\Theta + stu [z,[v_\Theta,v_{-\Theta}]] + st [v_\Theta,v_{-\Theta}] + ut^2 [z,v_{-\Theta}] + t^2 v_{-\Theta}.$$
  We work in the space $W = \langle v_\Theta , [z,[v_\Theta,v_{-\Theta}]] , [v_\Theta,v_{-\Theta}] , [z,v_{-\Theta}] , v_{-\Theta} \rangle$ and the above vector is given by $[s^2:stu:st:ut^2:t^2] \in \PP(W) = \PP^4$. The locus $S$ is therefore given by the following equations in $\PP(W)$: $x_0x_4 - x_2^2 = x_1x_4 - x_2x_3 = x_0x_3 - x_1x_2 = 0$. The locus covered by conics in $Y$ meeting $x$ and $L$ is therefore the intersection of $S$ with the hyperplane $(h = 0)$. We know that this hyperplane $(h = 0)$ has to contain $x$ and $L$ so its restriction to $\PP(W)$ has to contain $[1:0:0:0:0]$ and $[0:0:0:a:b]$ for all $[a:b] \in \PP^1$. In particular the restriction of the equation $(h = 0)$ to $\PP(W)$ is of the form $\lambda x_1 + \mu x_2 = 0$. The intersection with $S$ is therefore given by the union of the line $L$ with the conic in $\PP(W)$ given by the following equations:
  $$\lambda x_1 + \mu x_2 = \lambda x_3 + \mu x_4 = \lambda x_0x_3 + \mu x_2^2 = 0.$$
In particular, there is a unique conic joining $x$ and $L$ in $Y$, proving the result.  
\end{proof}

For $x = \sum_{\alpha \in \Phi} x_\alpha \alpha$, recall the definition of the support $\supp(x)$ of $x$ and recall from the discussion before Theorem \ref{thm:equi-chev} the definition of $|x| = |\supp(x)|$ and of $\|x\| = \max\{|x_\alpha| \ | \ \alpha \in \Phi\}$. If $\mathcal{P}$ is a boolean property, we write $\delta_{\mathcal{P}}$ the function such that
$$\delta_{\mathcal{P}} = \left\{\begin{array}{ll}
1 & \textrm{ If $\mathcal{P}$ is true } \\
0 & \textrm{ If $\mathcal{P}$ is false.} \\
\end{array}\right.$$
Recall that the quantum Chevalley formula is well known for homogeneous spaces.

\subsubsection{Adjoint case} We first consider the cases for which $X$ is adjoint not of type $A$.
Let $h_X$ be the hyperplane class in $X$ and recall the quantum Chevalley formula for $X$ (see \cite{fulton-woodward} or \cite{adjoint}). Recall also that $\alpha_0$ is the unique simple root with $\langle \Theta^\vee , \alpha_0 \rangle \neq 0$.

\begin{theorem}
\label{thm-qch-hom-ad} 
  Assume that $X$ is adjoint not of type $A$ and let $\alpha \in \aleph$.
  \begin{enumerate}
\item If $\alpha > 0$, then $h_X \star \sigma_{\alpha,X} = h_X \cup \sigma_{\alpha,X} + \delta_{\alpha = \alpha_0} q$.
\item If $\alpha < 0$ and $\alpha\neq -\Theta$, then $h_X \star \sigma_{\alpha,X} = h_X \cup \sigma_{\alpha,X} + |\langle \Theta^\vee , \alpha \rangle|
  q^{|\langle \Theta^\vee , \alpha \rangle|} \sigma_{s_\Theta(\alpha),X}$.
\item If $\alpha = -\Theta$, then $h_X \star \sigma_{\alpha,X} = q \sigma_{-\alpha_0,X} + 2q^2$.
  \end{enumerate}
\end{theorem}

We now prove the following result.

\begin{theorem}[Adjoint Quantum Chevalley formula]
   \label{thm-qch-ad}
  Assume that $X$ is adjoint not of type $A$ and let $Y \subset X$ be a general hyperplane section. Let $\alpha \in \aleph$.
  \begin{enumerate}
  \item If $\alpha > 0$ and $\deg(\sigma_\alpha) < c_1(Y) - 1$, then $h \star \sigma_\alpha = h \cup \sigma_\alpha$.
    \item If $\alpha > 0$ and $\deg(\sigma_\alpha) = c_1(Y) - 1$, then $h \star \sigma_\alpha = h \cup \sigma_\alpha + \|\alpha\| q\delta_{\alpha_0 \leq \alpha}$.
  \item If $\alpha$ is simple, then $h \star \sigma_\alpha = h \cup \sigma_\alpha + q h \delta_{\alpha_0 \leq \alpha}$.
  \item If $\alpha < 0$ and $|\Theta + \alpha| \geq 2$, then $h \star \sigma_\alpha = h \cup \sigma_\alpha + |\langle \Theta^\vee , \alpha \rangle| q^{|\langle \Theta^\vee , \alpha \rangle|} \sigma_{s_\Theta(\alpha)}.$
  \item  If $\alpha < 0$ and $|\Theta + \alpha| =1$, then $h \star \sigma_\alpha = h \cup \sigma_\alpha + q (\sigma_{\alpha_0} + \sigma_{-\alpha_0} + \sum_{\beta \in \Phi_\aleph, \langle \beta^\vee , \alpha_0 \rangle < 0} \sigma_{-\beta}) + 2q^2$.
    
  \item  If $\alpha = - \Theta$, then $h \star \sigma_\alpha = 2q^2 h + q \sum_{\gamma \in \Phi, \langle \gamma^\vee , \alpha_0 \rangle   < 0} |\langle \gamma^\vee , \alpha_0 \rangle| \sigma_{-s_\gamma(\alpha_0)}$.
  \end{enumerate}
\end{theorem}

\begin{remark}
  \label{rem-qch}
  In the above formulas, we want to emphasize the following:
  \begin{enumerate}
  \item For $\alpha$ simple, we have $h \star \sigma_\alpha = h \star \sigma_{-\alpha}$.
      \item If $\alpha = \alpha_0 - \Theta$, we have $h \star \sigma_\alpha = h \cup \sigma_\alpha + q j^*\sigma_{\alpha_0,X} + 2q^2$.
    \item If $\alpha = - \Theta$, the Chevalley formula in $X$ gives $h \star \sigma_\alpha = 2q^2h + q h \cup \sigma_{-\alpha_0}$.
  \end{enumerate}
\end{remark}

\begin{proof}
Note that the non quantum part of the product $h \star \sigma_\alpha$ is given by the classical Chevalley formula (see Corollary \ref{chevalley}). We may thus write
  $$h \star \sigma = h \cup \sigma_\alpha + \sum_{d > 0,\beta \in \aleph} \langle h , \sigma_\alpha , \sigma_\beta^\vee \rangle_d^Y q^d \sigma_\beta.$$

(1) There is no quantum correction for degree reasons and the result follows.

(2) Since $\deg(\sigma_\alpha) = c_1(Y) - 1$, the only non-trivial Gromov-Witten invariant is $\langle h , \sigma_\alpha , [{\rm pt}] \rangle_1^Y = \langle h_X , j_*\sigma_{\alpha} , [{\rm pt}] \rangle_1^X$. But because of the value of the degree, we have $\sigma_\alpha = j^*\sigma_{\alpha,X}$ therefore $j_*\sigma_\alpha = j_*j^*\sigma_{\alpha,X} = h_X \cup \sigma_{\alpha,X} = \sum_{\beta \in \Phi, \ \langle \beta^\vee , \alpha \rangle > 0} \langle \beta^\vee , \alpha \rangle \sigma_{s_\beta(\alpha),X}$. The roots $s_\beta(\alpha)$ appearing in the former sum are simple and the invariant $\langle h_X , \sigma_{s_\beta(\alpha),X} , [{\rm pt}] \rangle_1^X$ vanishes unless $s_\beta(\alpha) = \alpha_0$ and is equal to $1$ if $s_\beta(\alpha) = \alpha_0$. This proves the result.

(3) If $\alpha$ is simple, the only non-vanishing Gromov-Witten invariants are in degree $1$ and of the form $\langle h , \sigma_{\alpha} , [{\rm line}] \rangle_1^Y = \langle h_X , j_*\sigma_{\pm\alpha} , [{\rm line}] \rangle_1^X = \langle h_X , \sigma_{-\alpha,X} , [{\rm line}] \rangle_1^X$ (note that the same occurs for $\sigma_{-\alpha}$ which leads to the formula $h \star \sigma_\alpha = h \star \sigma_{-\alpha}$ for $\alpha$ simple given in Remark \ref{rem-qch}.(2)). But we have $\langle h_X , \sigma_{-\alpha,X} , [{\rm line}]  \rangle_1^X = \delta_{\alpha = \alpha_0}$, proving the result.

(4) Assume that $\alpha < 0$ and $|\Theta + \alpha| \geq 2$. Then $\deg(\sigma_\alpha) \in [c_1(Y) + 1,2c_1(Y) - 2]$. The only non vanishing Gromov-Witten invariants are in degree $1$ and of the form $\langle h , \sigma_\alpha , \sigma^\vee_\beta \rangle_1^Y$ with $\deg(\sigma_\beta) = \deg(\sigma_\alpha) + 1 - c_1(Y)$. In particular, for such invariants, we have $\sigma_\beta^\vee = \sigma_{\beta}^- = \sigma_{w_0(\beta)}$. Indeed, for $\beta > 0$, we have $j_*\sigma_\beta^\vee = j_*\sigma_\beta^- = \sigma_{\beta,X}^- = \sigma_{\beta,X}^\vee = \sigma_{w_0(\beta),X} = j_*\sigma_{w_0(\beta)}$. Therefore, we have
$\langle h , \sigma_\alpha , \sigma_\beta^\vee \rangle^Y_1 = \langle h , \sigma_\alpha , \sigma_{w_0(\beta)} \rangle^Y_1 = \langle h_X , \sigma_{\alpha,X} , \sigma_{w_0(\beta),X} \rangle^X_1 = \langle h_X , \sigma_{\alpha,X} , \sigma_{\beta,X}^\vee \rangle^X_1$ proving the result.

(5) If $\alpha < 0$ and $|\Theta + \alpha| = 1$, then $\deg(\sigma_\alpha) = 2c_1(Y) - 1$ and $\sigma_\alpha = [{\rm line}]$. The only non-trivial Gromov-Witten invariants in degree $1$ are $\langle h , \sigma_\alpha , \sigma_{\pm \beta} \rangle_1^Y = \langle h_X , j_*\sigma_{\alpha} , j_*\sigma_{\pm\beta} \rangle_1^X$ where $\beta$ is a simple root. We thus have $\langle h , \sigma_\alpha , \sigma_{\pm \beta} \rangle_1^Y = \langle h_X , [{\rm line}] , \sigma_{-\beta,X} \rangle_1^X = \delta_{\beta,\alpha_0}$, proving the formula $h \star \sigma_\alpha = h \cup \sigma_\alpha + q (\sigma_{\alpha_0}^\vee + \sigma_{-\alpha_0}^\vee) + 2q^2$. The result follows by applying Exampe \ref{ex:classe-duale}.

(6) If $\alpha = - \Theta$ then $\deg(\sigma_\alpha) = 2c_1(Y)$ and $\sigma_\alpha = [{\rm pt}]$. The only non-trivial Gromov-Witten invariants in degree $1$ are $\langle h , \sigma_\alpha , \sigma_{\beta} \rangle_1^Y = \langle h_X , j_*\sigma_{\alpha} , j_*\sigma_{\beta} \rangle_1^X = \langle h_X , [{\rm pt}] , j_*\sigma_{\beta} \rangle_1^X$ where $\deg(\sigma_\beta) = c_1(Y) - 1$. In particular we have $\sigma_\beta = j^*\sigma_{\beta,X}$ and thus $j_*\sigma_\beta = j_*j^*\sigma_{\beta,X} = h_X \cup \sigma_{\beta,X} = \sum_{\gamma \in \Phi, \ \langle \gamma^\vee , \beta \rangle > 0} \langle \gamma^\vee , \beta \rangle \sigma_{s_\gamma(\beta),X}$. In the former sum, the roots $s_\gamma(\beta)$ are simple. The invariant $\langle h_X , [{\rm pt}] , \sigma_{s_\gamma(\beta),X} \rangle_1^X$ vanishes unless $s_\gamma(\beta) = \alpha_0$ in which case the invariant is equal to $1$. Since this last condition is equivalent to $\beta = s_\gamma(\alpha_0)$, we thus get
$$h \star \sigma_\alpha = 2q^2 h + q \sum_{\gamma \in \Phi, \langle \gamma^\vee , \alpha_0 \rangle < 0} | \langle \gamma^\vee , \alpha_0 \rangle | \sigma_{s_\gamma(\alpha_0)}^\vee.$$
Now $\sigma_{s_\gamma(\alpha_0)}^\vee = \sigma_{w_0(s_\gamma(\alpha_0))}$ and $w_0(s_\gamma(\alpha_0)) = s_{w_0(\gamma)}(w_0(\alpha_0)) = - s_{-w_0(\gamma)}(\alpha_0)$ since $w_0(\alpha_0) = - \alpha_0$. Since $-w_0(\Phi) = \Phi$, and $\langle \gamma^\vee , \alpha_0 \rangle = \langle -w_0(\gamma)^\vee , - w_0(\alpha_0) \rangle = \langle -w_0(\gamma)^\vee , \alpha_0 \rangle$, we get
$$h \star \sigma_\alpha = 2q^2 h + q \sum_{\gamma \in \Phi, \langle \gamma^\vee , \alpha_0 \rangle
  < 0} |\langle \gamma^\vee , \alpha_0 \rangle| \sigma_{-s_\gamma(\alpha_0)},$$
finishing the proof.
\end{proof}

\begin{corollary} Assume that $X$ is adjoint not of type $A$. Then we have, for all $\alpha \in \Phi_\aleph$, the equalities $h \star ([{\rm pt}] - q \sigma_{\alpha_0} - q^2) 
 = h \star ([{\rm pt}] - q \sigma_{-\alpha_0} - q^2) = 0$.
\end{corollary}

Recall the definiton of $\Hna = \Ker(j_*\vert_{\HHH^{\dim Y}(Y)})$, the non ambient part of the middle cohomology (see Definition \ref{def-hna}). Then $(\Gamma_\alpha)_{\alpha \in \Phi_\aleph}$ is a basis of $\Hna$ with $\Gamma_\alpha = \sigma_\alpha - \sigma_{-\alpha}$.

\begin{corollary} Assume that $X$ is adjoint not of type $A$.
We have $h \star \Gamma_\alpha = 0$ for all $\alpha \in \Phi_\aleph$.
\end{corollary}

\begin{proof}
Indeed for $\alpha \in \Phi_\aleph$, we have $h \star \sigma_\alpha = h \star \sigma_{-\alpha}$.
\end{proof}

We now consider the endomorphisms $E_X = (h_X \star -)$ on $\QH(X)$ and $E_Y = (h \star -)$ on $\QH(Y)$. For $k \in [0,c_1(X) - 1]$, we set $\QH^k(X) = \oplus_{i \geq 0}H^{2k+2i c_1(X)}(X,\QQQ)$ and for $k \in [0,c_1(Y) - 1]$, we set $\QH^k(Y) = \oplus_{i \geq 0} H^{2k + 2ic_1(Y)}(X,\QQQ)$.

Set $W = \langle (\Gamma_\alpha)_{\alpha \in \aleph} , [\pt] - q \sigma_{-\alpha_0} - q^2 \rangle$, $\QH(Y)_{\rm res} = \QH(Y)/W$ and $\pires : \QH(Y) \to \QH(Y)_{\rm res}$ the projection. Since $W \subset \Ker E_Y$, the map $E_Y$ restricts to a map $\bar E_Y : \QH(Y)_{\rm res} \to \QH(Y)_{\rm res}$. Note that since $W \subset \QH^0(Y)$, the $\ZZ/c_1(Y)\ZZ$-grading of $\QH(Y)$ descends to a $\ZZ/c_1(Y)\ZZ$-grading on $\QH(Y)_{\rm res}$ and since $h$ is homogeneous of degree $1$, the map $\bar E_Y$ is homogeneous of degree $1$. Set
$$\QH(X)_{\rm res} = \bigoplus_{k = 0}^{c_1(Y) - 1} \QH^k(X).$$
We define the following $\QQQ$-linear map $\mathbb{j} : \QH(X)_{\rm res} \to \QH(Y)_{\rm res}$ via
$$\mathbb{j}\vert_{H^{2k}(X,\QQQ)} = \left\{\begin{array}{ll}
\pires \circ j^* & \textrm{ for $k \in [0,c_1(Y)-1]$} \\
\pires \circ (j_*)^{-1} & \textrm{ for $k \in [c_1(Y)+1,\dim Y]$} \\
\end{array}\right.$$
and $\mathbb{j}(q_X)=q$. Note that this is well defined since for $k \in [c_1(Y)+1,\dim Y]$, the map $j_*$ is an isomorphism.

\begin{lemma}
  The map $\mathbb{j} : \QH(X)_{\rm res} \to \QH(Y)_{\rm res}$ is an isomorphism of $\QQQ$-vector spaces.
\end{lemma}

\begin{proof}
This easily follows from the fact that the maps $j^*$ for $k \in [0,c_1(Y)-1]$ and $j_*$ for $k \in [c_1(Y)+1,\dim Y]$ are isomorphisms and the definitions of $\QH(X)_{\rm res}$ and $\QH(Y)_{\rm res}$.
\end{proof}

Define the map $\bar E_X : \QH(X)_{\rm res} \to \QH(X)_{\rm res}$ by
$$\bar E_X \vert_{QH^{k}(X)} =
\left\{\begin{array}{ll}
E_X & \textrm{ for $k\neq c_1(Y)-1$} \\
E_X^2 & \textrm{ for $k = c_1(Y)-1$}. \\
  \end{array}\right.$$
Note that $\QH(X)_{\rm res}$ has a $\ZZ/c_1(Y)\ZZ$-grading induced by the grading on $\QH(Y)_{\rm res}$ via the isomorphism $\mathbb{j}$. The map $\bar E_X$ is of degree $1$ for this grading.

\begin{proposition}
  We have the relation $\bar E_Y \circ \mathbb{j} = \mathbb{j} \circ \bar E_X$.
\end{proposition}

\begin{proof}
Let $\tau \in H^{2k}(X,\QQQ)$ with $k \in [0,\dim Y] \setminus \{c_1(Y)\}$ and choose $\sigma \in \QH(Y)$ with $\pires(\sigma) = \mathbb{j}(\tau)$. 

If $k \in [0,c_1(Y) - 2]$, then we may choose $\sigma$ such that $\sigma = j^*\tau$ and we have $E_Y(j^*\tau) = E_Y(\sigma) = h \cup \sigma = j^*(h_X \cup \tau) = j^*E_X(\tau)$ proving the result in this case.

If $k = c_1(Y) - 1$, then we may assume that $\tau = \sigma_{\alpha,X}$ for some positive root $\alpha$ and choose $\sigma = \sigma_\alpha = j^*\tau$. We have $E_Y(\sigma) = h \cup \sigma + \| \alpha \| q \delta_{\alpha_0 \leq \alpha}$ and $\bar E_X(\tau) = h_X^2\star \tau = h_X \star (h_X \cup \tau) = h_X \star (\sum_{\beta \in \Phi, \langle \beta^\vee,\alpha \rangle < 0}|\langle \beta^\vee , \alpha \rangle | \sigma_{s_{\beta}(\alpha)}) = h_X \cup h_X \cup \tau + q \sum_{\beta \in \Phi, \langle \beta^\vee,\alpha \rangle < 0}|\langle \beta^\vee , \alpha \rangle | \delta_{\alpha_0 = s_{\beta}(\alpha)} = h_X \cup h_X \cup \tau + \| \alpha \| q \delta_{\alpha_0 \leq \alpha}$. The result is true in this case using the following identities: $j_*(h \cup \sigma) = h_X \cup j_*\sigma = h_X \cup j_*j^*\tau = h_X \cup h_X \cup \tau$.
  
If $k \in [c_1(Y) + 1,\dim Y - 1]$, then we may assume that $\tau = \sigma_{\alpha,X}$ and choose $\sigma = \sigma_\alpha$ with $\alpha < 0$ such that $|\Theta + \alpha| \geq 2$. We have $j_*\sigma = \tau$ and $E_Y(\sigma) = h \cup \sigma + |\langle \Theta^\vee , \alpha \rangle | q \sigma_{s_{\Theta}(\alpha)} = j_*^{-1}(h_X \cup \tau) + q |\langle \Theta^\vee,\alpha \rangle| j^*( \sigma_{s_{\Theta}(\alpha),X} ) = \mathbb{j}( E_X(\tau) )$ proving the result in this case.

Finally assume that $k = \dim Y = \dim X - 1$. We may assume that $\tau = \sigma_{\alpha_0 - \Theta,X}$ and $\sigma = \sigma_{\alpha_0 - \Theta}$. We have $j_*\sigma = \tau$ and $E_Y(\sigma) = [\pt] + q j^*\sigma_{\alpha_0,X} + 2q^2$, therefore $\bar E_Y(\mathbb{j}(\tau)) = \bar E_Y(\pires(\sigma)) = q\sigma_{-\alpha_0} + q j^*\sigma_{\alpha_0,X} + 3q^2$. On the other hand, we have $\bar E_X(\tau) = h_X^2 \star \tau = h_X \star (\sigma_{-\Theta,X} + q \sigma_{\alpha_0,X}) = 3q^2 + q \sigma_{-\alpha_0,X} + q h_X \cup \sigma_{\alpha_0,X}$. The result follows since $j_*j^*\sigma_{\alpha_0,X} = h_X \cup \sigma_{\alpha_0,X}$. 
\end{proof}

\begin{corollary}
We have $\Ker E_Y = \pires^{-1}(\mathbb{j}(\Ker \bar E_X))$.
\end{corollary}

Note that $E_X^{c_1(X)}$ stabilises $\QH(X)_{\rm res}$. We set $\hat E_X^{c_1(X)} = E_X^{c_1(X)}\vert_{\QH(X)_{\rm res}}$

\begin{corollary}
We have $\bar E_Y^{c_1(Y)} = \mathbb{j} \circ \hat E_X^{c_1(X)} \circ \mathbb{j}^{-1}$.
\end{corollary}

\begin{proof}
We have $\bar E_X^{c_1(Y)} = \hat E_X^{c_1(X)}$.
\end{proof}

\begin{proposition}
  Let $X$ be adjoint not of type $A$.
  \begin{enumerate}
  \item The radical of $\QH(X)$ is contained in $\Ker E_X$.
    \item The non-zero eigenvalues of $E_X$ have multiplicity one.
\item If $X$ is quasi-minuscule, the minimal polynomial $\mu_{E_X}$ of $E_X$ is of the form $\mu_{E_X}(T) = T P(T^{c_1(X)})$ with $P$ having non-zero simple roots.
\item If $X$ is not quasi-minuscule, the minimal polynomial $\mu_{E_X}$ of $E_X$ is of the form $\mu_{E_X}(T) = P(T^{c_1(X)})$ with $P$ having non-zero simple roots.
  \end{enumerate}
\end{proposition}

\begin{proof}
  (1) See \cite[Theorem 1.4]{maxim} for quasi-minuscule cases. For $X$ adjoint not quasi minuscule, the radical is trivial (see \cite[Theorem 6]{adjoint}). 

  (2) For exceptional groups, the result follows from the explicit presentation given in \cite{adjoint}. In type $C_n$ we have $X = \PP^{2n-1}$ and $\QH(X) = \QQQ[h_X,q]/(h_X^{2n} - q)$ proving the result. The types $B_n$ and $D_n$ are slightly more complicated. Let
  $$\Sigma_{2k} = \sum_{i = 0}^k x_1^{2i}x_1^{2k-2i} \textrm{ and } F = \sum_{i = 0}^{2n-2}x_1^kx_2^{2n-2-k}.$$
  It was proved in \cite{adjoint} on page 325 that in type $B_n$ we have $\QH(X) = \QQQ[x_1,x_2,q]^{S_2}/(\Sigma_{2n-2} + 4q,x_1x_2F)$ where $S_2 = \ZZ/2\ZZ$ acts by exchanging the variables $x_1$ and $x_2$. Furthermore, we have $h_X = x_1 + x_2$. It was proved in \cite[Corollary 5.8]{maxim} that in type $D_n$, we have $\QH(X) = \QQQ[x_1,x_2,\gamma,q]^{S_2}/(x_1x_2\gamma,\gamma^2 + (-1)^n\Sigma_{2n-4}, (x_1^2 + x_1x_2 + x_2^2)\Sigma_{2n-4} -x_1^2x_2^2 \Sigma_{2n-6} + 4q(x_1+x_2))$. Again $S_2 = \ZZ/2\ZZ$ acts by exchanging the variables $x_1$ and $x_2$ and $h_X = x_1 + x_2$. Resolving these equations in $x_1$ and $x_2$ (this is done in \cite[Page 325]{adjoint} and \cite[Proposition 5.12]{maxim}), the result follows.

(3) and (4) follow from (1) and (2) and an explicit computation showing that the exponents of $T$ appearing in $\mu_{E_X}(T)$ are multiples of $c_1(X)$ if $X$ is not quasi-minuscule, and of type $kc_1(X)+1$ for $k\in \NN$ if $X$ is quasi-minuscule.
\end{proof}

\begin{corollary}
\label{cor:min_pol_Y}
  The minimal polynomial $\mu_{E_Y}$ of $E_Y$ is $\mu_{E_Y} = T P(T^{c_1(Y)})$ and the non-zero eigenvalues of $E_Y$ have multiplicity one. In particular, the radical of $\QH(Y)$ is contained in $\Ker E_Y$. 
\end{corollary}

\begin{proof}
  Note that $Q(T) = T P(T^{c_1(Y)})$ has simple roots so to prove the assertion on the minimal polynomial, it is enough to prove that $Q(E_Y) = 0$. Let $\sigma \in \QH(Y)$, we have $\pires(P(E_Y^{c_1(Y)})(\sigma)) = P(\bar E_Y^{c_1(Y)})(\pires(\sigma)) = \mathbb{j} P(\hat E_X^{c_1(X)})(\mathbb{j}^{-1}(\sigma))$. Because of the form of $\mu_{E_X}$, we have the inclusion $P(\hat E_X^{c_1(X)})(\mathbb{j}^{-1}(\sigma)) \in \Ker \bar E_X$, thus $\pires(P(E_Y^{c_1(Y)})(\sigma)) \in \Ker \bar E_Y$ and $P(E_Y^{c_1(Y)})(\sigma) \in \Ker E_Y$ proving the vanishing result.

  Furthermore, as pairs of a vector space with an endormorphism, the two pairs $(\QH(Y)_{\rm res} , \bar E_Y^{c_1})$ and $(\QH(X)_{\rm res}, \hat E_X^{c_1(X)})$ are isomorphic and the second has non-zero eigenvalues of multiplicity one. The same is therefore true for $(\QH(Y),E_Y)$ since $W \subset \Ker E_Y$. The last assertion follows from this.
\end{proof}

\begin{corollary}
\label{cor:kerY}
Let $X$ be adjoint not quasi-minuscule, then $\Ker E_Y = W$ is of dimension $|\Phi_\aleph| + 1$.
\end{corollary}

\begin{proof}
Follows from the fact that $\Ker E_X = 0$.
\end{proof}

The following result reproduces the analogous statement concerning the quantum cohomology of quasi-minuscule varieties.

\begin{proposition}
  \label{prop-non-ss-qmin+adj}
  If $X$ is adjoint and quasi-minuscule (but not minuscule) not of type $D_n$ nor $A_n$, then $\QH(Y)$ is not semi-simple.
\end{proposition}

\begin{proof}
  We find an element $\sigma \in \Ker h \setminus \{0\}$ such that $\sigma^2 = 0$. We proceed type by type. By assumption $G$ is simply laced.

  In type $E_6$, it is proved in \cite[Theorem 1.4]{maxim} that the radical of $\QH(X)$ is located in degrees $3$, $4$, $6$, $7$ and $10$. Pick $\sigma_X \in \Ker h_X$ of degree $4$, then there exists $\sigma \in \QH(Y)$ such that $\pires(\sigma) = \mathbb{j}(\sigma_X)$. We have $\sigma \in \Ker h \setminus \{0\}$ and $\deg(\sigma) = 4$. We deduce that $\sigma^2 \in \Ker h$ and $\deg(\sigma^2) = 8$. But our results imply that the degrees for which $\Ker h$ is non trivial are the same as for $\Ker h_X$. In particular, there is no non-trivial element in $\Ker h$ in degree $8$ and $\sigma^2 = 0$.

  In type $E_7$, it is proved in \cite[Theorem 1.4]{maxim} that the radical of $\QH(X)$ is located in degrees $4$, $6$, $8$, $10$, $12$ and $16$. Pick $\sigma_X \in \Ker h_X$ of degree $6$, then there exists $\sigma \in \QH(Y)$ such that $\pires(\sigma) = \mathbb{j}(\sigma_X)$. We have $\sigma \in \Ker h \setminus \{0\}$ and $\deg(\sigma) = 6$. We deduce that $\sigma^3 \in \Ker h$ and $\deg(\sigma^3) = 18 \equiv 2$ modulo $c_1(Y) = 16$. But our results imply that the degrees for which $\Ker h$ is non trivial are the same as for $\Ker h_X$. In particular, there is no non-trivial element in $\Ker h$ in degree $2$ and $\sigma^3 = 0$.

  In type $E_8$, it is proved in \cite[Theorem 1.4]{maxim} that the radical of $\QH(X)$ is located in degrees $6$, $10$, $12$, $16$, $18$, $22$ and $28$. Pick $\sigma_X \in \Ker h_X$ of degree $10$, then there exists $\sigma \in \QH(Y)$ such that $\pires(\sigma) = \mathbb{j}(\sigma_X)$. We have $\sigma \in \Ker h \setminus \{0\}$ and $\deg(\sigma) = 10$. We deduce that $\sigma^2 \in \Ker h$ and $\deg(\sigma^3) = 20$. But our results imply that the degrees for which $\Ker h$ is non trivial are the same as for $\Ker h_X$. In particular, there is no non-trivial element in $\Ker h$ in degree $10$ and $\sigma^2 = 0$.
\end{proof}

\begin{remark}
  In type $D_n$ with $n \geq 4$, the above technique does not work since it was proved in \cite[Lemma 5.13]{maxim} that the radical of $\QH(X)$ is contained in $\Ker h_X$ and has non trivial elements in degrees $n-2$ and $2k$ for $k \in [1,n-2]$ (for $n$ even there are two linearly independent element of degree $n-2$). In particular for $n-2$ even, any power of an element in $\Ker h$ will have an even degree, therefore occuring as a non trivial degree in $\Ker h$ (recall that $c_1(Y) = 2n - 4$ is also even).
  \end{remark}

On the other hand, recall the following result about adjoint not quasi-minuscule varieties.

\begin{proposition}[\cite{adjoint}]
  Let $X$ be adjoint not quasi-minuscule, then $\QH(X)$ is semi-simple.
\end{proposition}

In the last part of this section we will prove the analogous of the previous result for hyperplane sections $Y\subset X$.

\begin{theorem}
  \label{thm:ss}
  Let $X$ be adjoint not quasi-minuscule, then $\QH(Y)$ is semi-simple.
\end{theorem}

According to a conjecture of Kuznetsov and Smirnov \cite{KS}, we also expect the following result.

\begin{conjecture}
  \label{conj:db}
  Let $X$ be adjoint not quasi-minuscule.
  \begin{enumerate}
\item $D^b(X)$ has a full rectangular Lefschetz collection of size $c_1(X) \times (|\Phi_\aleph| +1)$;
  \item $D^b(Y)$ has a rectangular Lefschetz collection of size $c_1(Y) \times (|\Phi_\aleph| +1)$ and its residual category has a completely orthogonal exceptional collection of size $|\Phi_\aleph| +1$.
  \end{enumerate}
\end{conjecture}

\begin{remark}
In \cite{Kuz_G2} Kuznetsov has already shown that a hyperplane section of the adjoint variety of type $G_2$ admits a rectangular collection of size $2\times 3$ and its residual category has a completely orthogonal exceptional collection of size $3$, as expected. This is done by showing that the homological projective dual variety of $G_2/P_2$ exists and can be realized as a double cover of $\PP^{13}$ ramified along a sextic (the classical projective dual variety of $G_2/P_2$). The fact that the cover is double allows to recover, as the residual category of $Y\subset G_2/P_2$, the derived category of two points, \emph{i.e.} a completely orthogonal exceptional collection of size $2$. 

In view of Theorem \ref{thm:ss} it may be argued that a similar picture may hold for the other adjoint not quasi-minuscule varieties. For instance, if Conjecture \ref{conj:db} holds and the homological projective dual variety of $F_4/P_1$ exists, it should be realized as a triple cover of $\PP^{51}$; if the analogy is pushed further, this triple cover should be ramified along the classical dual projective variety of $F_4/P_1$, which is a hypersurface of degree $24$.
\end{remark}

Let us now give a proof of Theorem \ref{thm:ss}. We start with some intermediate results. Let $A \subset \QH(Y)$ be the subalgebra generated by $h$. Recall the definition of $\Ha = j^*\HHH(X)$ the ambient cohomology and recall the following result from \cite[Corollary 3.2]{xu} (the statement is written for complete intersections in $\PP^n$ but the proof adapts verbatim).

\begin{lemma}
  \label{lemm:amb-qh}
  The subspace $\Ha \subset \QH(Y)$ is stable under the quantum multiplication. In particular $A \subset \Ha$.
\end{lemma}

\begin{lemma}
  \label{lemm:pas-de-pt}
Let $\sigma,\tau \in \HHH(Y)$ and expand $\sigma \star \tau$ in the Schubert basis. Then the coefficient of $[\pt]$ in $\sigma \star \tau$ is $\sigma \cup \tau$.
\end{lemma}

\begin{proof}
Any other term is of the form $\langle \sigma , \tau , 1\rangle^Y_d$ with $d > 0$ and it vanishes by definition of Gromov-Witten invariants.
\end{proof}

\begin{proposition}
  \label{prop:sigma}
  Let $X$ be adjoint non quasi-minuscule and let $Y \subset X$ be a general hyperplane section.
  \begin{enumerate} 
  \item  The intersection $A \cap \Ker E_Y$ has dimension $1$. 
  \item  There exists a unique $\sigma \in A \cap \Ker E_Y$ such that, in the expansion of $\sigma$ in the Schubert basis, the coefficient of the class of the point is $1$.
  \item For $\sigma$ as above, there exists $\lambda_0 \in \QQQ$ such that $\sigma^2 = \lambda_0 q^2 \sigma$ and $\sigma \Gamma = \lambda_0q^2 \Gamma$ for all $\Gamma \in \Hna$.
    \item The rational $\lambda_0$ satifies the following:
    \begin{enumerate}
  \item For $G$ of type $B_n$ or $F_4$, we have $\lambda_0 < 0$.
    \item For $G$ of type $G_2$, we have $-\lambda_0 = \frac{3}{2}$ and is not a quare in $\QQQ$.
  \end{enumerate}
\end{enumerate}
\end{proposition}

\begin{proof}
(1) Recall from Corollary \ref{cor:min_pol_Y} that the non-zero eigenvalues of $E_Y$ have multiplicities $1$ and that its minimal polynomial has simple roots. If $\mu$ is the minimal polynomial of $E_Y$ and $\chi$ is its characteristic polynomial, this implies that $\deg(\mu) - 1 = \deg(\chi) - \dim \Ker E_Y$. Thus $\dim A = \deg(\mu) = \deg(\chi) + 1 - \dim \Ker E_Y = \dim \QH(Y) + 1 - \dim \Ker E_Y$ proving that $\dim A \cap \Ker E_Y \geq 1$.

  Recall from Corollary \ref{cor:kerY} that $\Ker E_Y = W = \langle [\pt] - q \sigma_{-\alpha_0} - q^2 \rangle + \Hna$. So if we have $\dim A \cap \Ker E_Y \geq 2$, then $A \cap \Hna$ is non trivial, contradicting Lemma \ref{lemm:amb-qh}. 

  (2) Let $\sigma$ be non trivial in $A\cap \Ker E_Y$, then $\sigma = \lambda ([\pt] - q \sigma_{-\alpha_0} - q^2) + q \Gamma$ with $\Gamma \in \Hna$ and by the above argument $\lambda \neq 0$; dividing by $\lambda$ we get the existence of $\sigma$. The uniqueness follows from (1).

  (3) Write $\sigma$ as a polynomial in $h$. This is possible since $\sigma \in A$. We have $\sigma = \sum_{k = 0}^Na_kh^k$ for some $N$ and some $a_i \in \QQQ[q,q^{-1}]$. Since $\sigma \in \Ker E_Y$ and $\Gamma \in \Hna \subset \Ker E_Y$, we get $\sigma^2 = \sum_ka_kh^k\sigma = a_0 \sigma$ and $\sigma\Gamma = a_0 \Gamma$. Computing the degrees proves that $a_0 = \lambda_0 q^2$ for some $\lambda_0 \in \QQQ$.

  (4.a) We compute the coefficient of $[\pt]$ in $\sigma^2$.  We have $\sigma = [\pt] - q^2 - q \sigma_{-\alpha_0} + q\Gamma$ with $\Gamma \in \Hna$. Set $\sigma_t = t([\pt] - q^2 - q \sigma_{-\alpha_0}) + q\Gamma$ so that $\sigma = \sigma_1$. By Lemma \ref{lemm:pas-de-pt}, the coefficient of $[\pt]$ in the expansion of $\sigma_t^2$ in the Schubert basis is $-2t^2q^2 - 2q^2 t \sigma_{-\alpha_0} \cup \Gamma + q^2 \Gamma \cup \Gamma$. On the other hand, the coefficient of $[\pt]$ of $\lambda_0 q^2 \sigma$ is $q^2 \lambda_0$. We get $\lambda_0 = - 2 -  2\sigma_{-\alpha_0} \cup \Gamma + \Gamma \cup \Gamma$.

  We want to study the values of the quadratic form $q(t,\Gamma) = -2t^2 - 2t \sigma_{-\alpha_0} \cup \Gamma + \Gamma \cup \Gamma$. Write $\Gamma = \sum_{\alpha \in \aleph} t_\alpha \Gamma_\alpha$

 For $G$ of type $B_n$ or $F_4$, we have $\sigma_{-\alpha_0} \cup \Gamma = t_{\alpha_0}$ and the matrix of the intersection form $\Gamma \cup \Gamma$ is $C_\aleph - 4I_\aleph$. This proves that the matrix of the quadratic form $q(t,\Gamma)$ (up to reordering the roots) is equal to the matrix $C_\aleph-4I_\aleph$ with root system $D_n$ in type $B_n$ and root system $A_3$ in type $F_4$. Since $4I_\aleph-C_\aleph$ is positive definite we get the result.

  (4.b) In type $G_2$, the minimal polynomial of $h_X$ is $P(T) =  T^2 - 18q T - 27q^2$ (see \cite{adjoint}), therefore $\sigma$ is a multiple of $h^2 - 18 qh - 27q^2$. By Chevalley formula, we have $h^2 - 18 qh - 27 q^2 = -18q^2 - 9q \sigma_{\alpha_1} - 9q \sigma_{-\alpha_1} + 18 [\pt]$, therefore we have
 $$\sigma = [\pt] - \frac{1}{2}q(\sigma_{\alpha_1} + \sigma_{\alpha_2}) - q^2 = \frac{1}{18} h^2 - qh - \frac{3}{2} q^2$$
 proving that $\lambda_0 = -\frac{3}{2}$.
 
\end{proof}

Now we can prove Theorem \ref{thm:ss}.

\begin{proof}[Proof of Theorem \ref{thm:ss}]
  Let $x \in \QH(Y)$ be a nilpotent element. Recall from Corollary \ref{cor:min_pol_Y} that the radical of $\QH(Y)$ is contained in $\Ker E_Y$, thus $x \in \Ker E_Y$. We may assume that $x$ is homogeneous of degree $\dim Y$, thus it can be written as $x = \lambda \sigma + q\Gamma$ with $\lambda \in \QQQ$ and $\Gamma \in \Hna$. If $\lambda = 0$, then $x = q\Gamma$ and $\Gamma$ would be nilpotent in $\QH(Y)$. By Lemma \ref{lemm:pas-de-pt}, the coefficient of $[\pt]$ in the Schubert expansion of $\Gamma^2$ is $\Gamma \cup \Gamma$ so we would then have $\Gamma \cup \Gamma = 0$. Since the cup product is non-degenerate (positive definite or negative definite) on $\Hna$, this is not possible. We may therefore assume that $\lambda \neq 0$.

  Since $\Gamma \in \ker E_Y$, we have $\Gamma^2 \in \Ker E_Y$ thus there exists $\mu \in \QQQ$ and $\Gamma' \in \Hna$ such that $\Gamma^2 = \mu \sigma + q \Gamma'$. We compute $x^2 = \lambda^2\sigma^2 + 2\lambda q\sigma \Gamma + q^2 \Gamma^2 = \lambda_0 \lambda^2 q^2 \sigma + 2\lambda_0 \lambda q^3 \Gamma + \mu q^2 \sigma + q^3 \Gamma' = (\lambda_0 \lambda^2 + 
  \mu ) q^2 \sigma + (2\lambda_0 \lambda q^3 \Gamma + q^3 \Gamma')$.
  Since $\sigma \not\in \Hna$, we get from $x^2 = 0$ the equalities $\lambda_0\lambda^2 + \mu = 0$ and $2\lambda_0\lambda \Gamma + \Gamma' = 0$.

  Note that to compute $\mu$, it is enough to compute the coefficient of $[\pt]$ in the expansion of $\Gamma^2$ in the Schubert basis. By Lemma \ref{lemm:pas-de-pt}, we thus have $\mu = \Gamma \cup \Gamma$. In particular, we have
  $$\lambda_0 = - \frac{\mu}{\lambda^2} = -\frac{\Gamma \cup \Gamma}{\lambda^2}.$$
  In types $B_n$ and $F_4$ the cup-product is negative definite on $\Hna$, thus $\Gamma \cup \Gamma < 0$ and we get $\lambda_0 > 0$ contradicting Proposition \ref{prop:sigma}.4.a. In type $G_2$, the cup-product $\Gamma \cup \Gamma$ is twice a square since $\Hna$ has dimension $1$ and from Proposition \ref{prop:sigma}.4.b we would get that $\frac{3}{4} = -\frac{\lambda_0}{2}$ is a square, a contradiction.
\end{proof}

\subsection{Quasi-minuscule case} We now consider the cases for which $X$ is quasi-minuscule not adjoint. 
Recall the quantum Chevalley formula for $X$ in this case.

\begin{theorem}[Quantum Chevalley formula for $X$, \cite{fulton-woodward,adjoint}]
\label{thm-qch-hom-qmin}
Assume that $X$ is quasi-minuscule not adjoint, let $h_X$ be the hyperplane class and let $\alpha \in \aleph$. We have
$$h_X \star \sigma_{\alpha,X} = h_X \cup \sigma_{\alpha,X} + q \delta_{\alpha \leq \theta - \Theta} \sigma_{\alpha + \Theta}.$$
\end{theorem}

We now prove the following result.

\begin{theorem}[Quasi-minuscule Quantum Chevalley formula]
  \label{thm-qch-qm}
  Let $X$ be quasi-minuscule not adjoint and let $Y \subset X$ be a general hyperplane section. Let $\alpha \in \aleph$.
  \begin{enumerate}
  \item If $\alpha > 0$ is not simple, then $h \star \sigma_\alpha = h \cup \sigma_\alpha$.
  \item If $\alpha$ is simple, then $h \star \sigma_\alpha = h \star \sigma_{-\alpha}$.
  \item If $\alpha < 0$, then $h \star \sigma_\alpha = h \cup \sigma_\alpha + \delta_{\alpha \leq \theta - \Theta} q j^*\sigma_{\alpha + \Theta,X}$.
  \end{enumerate}
\end{theorem}

\begin{proof}
(1) If $\alpha > 0$ not simple, then $\deg(\sigma_\alpha) < c_1(Y) - 1$ and the result follows.

  (2) If $\alpha$ is simple, then $j_*\sigma_\alpha = j_*\sigma_{-\alpha}$ and the only non-vanishing Gromov-Witten invariants are obtained for lines and are of the form $\langle h , \sigma_\alpha , \sigma_\beta \rangle_1^Y$. But we have $\langle h , \sigma_\alpha , \sigma_\beta \rangle_1^Y = \langle h_X , j_*\sigma_\alpha , j_*\sigma_\beta \rangle_1^X = \langle h_X , j_*\sigma_{-\alpha} , j_*\sigma_\beta \rangle_1^X = \langle h , \sigma_{-\alpha} , \sigma_\beta \rangle_1^Y$, proving the result.

  (3) The only non-vanishing Gromov-Witten invariants are obtained for lines and are of the form $\langle h , \sigma_\alpha , \sigma_\beta \rangle_1^Y = \langle h_X , j_*\sigma_\alpha , j_*\sigma_\beta \rangle_1^X$. Note that computing the degrees, we see that $\deg(\sigma_\beta) \geq c_1(Y) - 1$. Note that this implies $\deg(\sigma_\beta) > \dim Y/2$ except maybe in type $C$ for which $\pm\beta$ is simple. If $X$ is not of type $C$ or $\pm\beta$ is not simple we thus have $\langle h_X , j_*\sigma_\alpha , j_*\sigma_\beta \rangle_1^X = \langle h_X , \sigma_{\alpha,X} , \sigma_{\beta,X} \rangle_1^X = \delta_{\alpha + \beta =- \Theta}$ and the result follows from the fact that $\sigma_\beta^\vee  = \sigma_{-\beta} = j^*\sigma_{-\beta,X}$. The case where $X$ is of type $C$ and $\pm\beta$ is simple only occurs if $\alpha=-\theta$. Since $j_*\sigma_\beta=j_*\sigma_{-\beta}$, we may assume that $\beta$ is simple and deduce that $\langle h_X , j_*\sigma_\alpha , j_*\sigma_{\beta} \rangle_1^X = \langle h_X , j_*\sigma_\alpha , j_*\sigma_{-\beta} \rangle_1^X = \langle h_X , \sigma_{\alpha,X} , \sigma_{-\beta,X} \rangle_1^X = \delta_{\alpha - \beta =- \Theta}$. Thus we obtain $h\star \sigma_{-\theta}=h\cup \sigma_{-\theta}+q(\sigma_{\Theta-\theta}^\vee+\sigma_{\theta-\Theta}^\vee)=h\cup \sigma_{-\theta}+q(j^*\sigma_{\Theta-\theta,X})$ by Example \ref{ex:classe-duale}.
\end{proof}

\begin{corollary} Assume that $X$ is quasi-minuscule not adjoint, then $h \star \Gamma_\alpha = 0$ for all $\alpha \in \Phi_\aleph$.
\end{corollary}

We now consider the endomorphisms $E_X = (h_X \star -)$ on $\QH(X)$ and $E_Y = (h \star -)$ on $\QH(Y)$. For $k \in [0,c_1(X) - 1]$, we set $\QH^k(X) = \oplus_{i \geq 0}H^{2k+2i c_1(X)}(X,\QQQ)$ and for $k \in [0,c_1(Y) - 1]$, we set $\QH^k(Y) = \oplus_{i \geq 0} H^{2k + 2ic_1(Y)}(X,\QQQ)$.

Note that the map $j_* : H^{\dim Y}(Y,\QQQ) \to H^{\dim Y + 2}(X,\QQQ)$ is surjective and define a right inverse $(j_*)^{-1} : H^{\dim Y + 2}(X,\QQQ) \to H^{\dim Y}(Y,\QQQ)$ by $(j_*)^{-1}(\sigma_{-\alpha,X}) = \sigma_{-\alpha}$ for $\alpha \in \Phi_\aleph$. We define the following $\QQQ$-linear map $\mathbb{j} : \QH(X) \to \QH(Y)$ via
$$\mathbb{j}\vert_{H^{2k}(X,\QQQ)} = \left\{\begin{array}{ll}
j^* & \textrm{ for $k \in [0,\dim Y/2]$} \\
(j_*)^{-1} & \textrm{ for $k \in [\dim Y/2+1,\dim X]$} \\
\end{array}\right.$$

\begin{lemma}
  The map $\mathbb{j} : \QH(X) \to \QH(Y)$ is an isomorphism of $\QQQ$-vector spaces.
\end{lemma}

\begin{proof}
For $k \neq \dim Y/2$, this easily follows from the fact that the maps $j^*$ for $k \in [0,\dim Y/2 - 1]$ and $j_*$ for $k \in [\dim Y/2 + 1,\dim Y]$ are isomorphisms. For $k = \dim Y/2$ the result follows from Corollary \ref{cor-j}.
\end{proof}

Define the map $\bar E_X : \QH(X) \to \QH(X)$ by
$$\bar E_X \vert_{H^{2k}(X,\QQQ)} =
\left\{\begin{array}{ll}
E_X & \textrm{ for $k\neq \dim Y/2$} \\
E_X^2 & \textrm{ for $k = \dim Y/2$}. \\
  \end{array}\right.$$

\begin{proposition}
  We have the relation $E_Y \circ \mathbb{j} = \mathbb{j} \circ \bar E_X$.
\end{proposition}

\begin{proof}
Let $\alpha \in \aleph$. If $\alpha > 0$ not simple, then $h_X \star \sigma_{\alpha,X} = h_X \cup \sigma_{\alpha,X}$ and $h \star \mathbb{j}(\sigma_{\alpha,X}) = h \star \sigma_\alpha = h \cup \sigma_\alpha = j^*(h_X \cup \sigma_{\alpha,X})$ proving the result in this case.

If $\alpha$ is simple then $\deg(\sigma_{\alpha,X}) = \dim Y/2$. We have $h_X^2 \star \sigma_{\alpha,X} = h_X \star (h_X \cup \sigma_{\alpha,X}) = h_X \star (2 \sigma_{-\alpha,X} + \sum_{\beta \in \Phi_\aleph, \langle \beta^\vee,\alpha \rangle < 0} \sigma_{-\beta,X}) = h_X \cup h_X \cup \sigma_{\alpha,X} + q_X \sum_{\beta \in \Phi_\aleph, \langle \beta^\vee,\alpha \rangle < 0} \delta_{-\beta \leq \theta - \Theta} \sigma_{\Theta-\beta,X} + 2q_X\delta_{-\alpha\leq \theta-\Theta}\sigma_{\Theta-\alpha,X}$. On the other hand, we have $h \star \mathbb{j}(\sigma_{\alpha,X}) = h \star j^*\sigma_{\alpha,X} = h \star (\sigma_\alpha + \sigma_{-\alpha} + \sum_{\beta \in \Phi_\aleph, \langle \beta^\vee,\alpha \rangle < 0} \sigma_{-\beta}) = h \cup j^*\sigma_{\alpha,X} + q\sum_{\beta \in \Phi_\aleph, \langle \beta^\vee,\alpha \rangle < 0} \delta_{-\beta \leq \theta - \Theta} \sigma_{\Theta-\beta} + 2q\delta_{-\alpha\leq \theta-\Theta}\sigma_{\Theta-\alpha}$. The result follows from this and the equality $j_*(h \cup j^* \sigma_{\alpha,X}) = h_X \cup h_X \cup \sigma_{\alpha,X}$.

If $\alpha < 0$, then $h_X \star \sigma_{\alpha,X} = h_X \cup \sigma_{\alpha,X} + q \delta_{\alpha \leq \theta - \Theta} \sigma_{\alpha + \Theta}$. On the other hand, we have $j_*\sigma_\alpha = \sigma_{\alpha,X}$ and $h \star \sigma_\alpha = h \cup \sigma_\alpha + q \delta_{\alpha \leq \theta - \Theta} \sigma_{\alpha + \Theta}$ and the result follows since $j_*(h \cup \sigma_\alpha) = h_X \cup j_*\sigma_{\alpha,X}$ and $\sigma_{\alpha + \Theta} = j^*\sigma_{\alpha + \Theta,X}$.
\end{proof}

\begin{corollary}
We have $\Ker E_Y = \mathbb{j}(\Ker \bar E_X)$.
\end{corollary}

\begin{corollary}
We have $E_Y^{c_1(Y)} = \mathbb{j} \circ E_X^{c_1(X)} \circ \mathbb{j}^{-1}$.
\end{corollary}

\begin{proof}
Note that $c_1(X) > (\dim X + 1)/2$ thus $c_1(Y) = c_1(X) - 1 > \dim Y/2 + 1$ and $\bar E_X^{c_1(Y)} = E_X^{c_1(X)}$; the result follows.
\end{proof}

\begin{proposition}
  Let $X$ be quasi-minuscule not adjoint.
  \begin{enumerate}
  \item The radical of $\QH(X)$ is contained in $\Ker E_X$.
  \item The non-zero eigenvalues of $E_X$ have multiplicity one.
\item The minimal polynomial $\mu_{E_X}$ of $E_X$ is of the form $\mu_{E_X}(T) = T P(T^{c_1(X)})$ with $P$ having non-zero simple roots.
  \end{enumerate}
\end{proposition}

\begin{proof}
  (1) This is proved in \cite[Theorem 1.4]{maxim}.

  (2) There are only two cases to consider: $X = \IGr(2,2n)$ and $X = F_4/P_4$. The latter case is easily obtained from the explicit presentation given in \cite{adjoint}. For $X= \IGr(2,2n)$, we have a presentation $\QH(X) = \QQQ[a_1,a_2,b_1,\cdots,b_{n-2},q]/I$ where $I$ is the ideal generated by the relation
  $$(1+(2a_2 - a_1^2)x^2) + a_2^2x^4)(1 + b_1x^2 + \cdots + b_{n_2}x^{2n-4}) = 1 - qa_1x^{2n},$$
  see \cite[Corallary 4.2]{ig22n}. We have $h_X = a_1$. Resolving this equation in $a_1$ (see \cite[Proposition 4.3]{ig22n}) proves the result.

  (3) Follows from (1) and (2) and an explicit computation showing that the exponents of $T$ appearing in $\mu_{E_X}(T)$ are of the form $kc_1(X)+1$ for $k\in \NN$ if $X$ is quasi-minuscule.
\end{proof}

\begin{corollary}
  The minimal polynomial $\mu_{E_Y}$ of $E_Y$ is $\mu_{E_Y} = T P(T^{c_1(Y)})$ and the non-zero eigenvalues of $E_Y$ have multiplicity one. In particular, the radical of $\QH(Y)$ is contained in $\Ker E_Y$. 
\end{corollary}

\begin{proof}
  Note that $Q(T) = T P(T^{c_1(Y)})$ has simple roots so to prove the assertion on the minimal polynomial it is enough to prove that $Q(E_Y) = 0$. Let $\sigma \in \QH(Y)$, we have $P(E_Y^{c_1(Y)})(\sigma) = \mathbb{j} P(E_X^{c_1(X)})(\mathbb{j}^{-1}(\sigma))$. Because of the form of $\mu_{E_X}$, we have the inclusion $P(E_X^{c_1(X)})(\mathbb{j}^{-1}(\sigma)) \in \Ker E_X$ thus we have $P(E_Y^{c_1(Y)})(\sigma) \in \Ker E_Y$ proving the vanishing result.

  Furthermore, as pairs of a vector space with an endormorphism, the two pairs $(\QH(Y) , E_Y^{c_1})$ and $(\QH(X) , E_X^{c_1(X)})$ are isomorphic and the second has non-zero eigenvalues of multiplicity one. The same is therefore true for $(\QH(Y),E_Y)$. The last assertion follows from this.
\end{proof}

\begin{corollary}
 \label{prop-non-ss-qmin-non-adj}
  Let $X$ be quasi-minuscule not of type $C_n$ or $D_n$, then $\QH(Y)$ is not semi-simple.
\end{corollary}

\begin{proof}
  The case where $X$ is quasi-minuscule and adjoint follows from Proposition \ref{prop-non-ss-qmin+adj}. We thus prove the result for $X$ quasi-minuscule not adjoint not of type $C_n$. The only case is $X = F_4/P_4$. There is an element $\sigma \in \Ker h \setminus \{0\}$ with $\sigma \in \QH^4(Y)$. Then $\sigma^2 \in \QH^8(Y) \cap \Ker h = 0$.
\end{proof}

\begin{remark}
We believe that $\QH(Y)$ is also not semi-simple in types $C_n$ and $D_n$ but we are not able to prove this using our technique. 
\end{remark}

\bibliographystyle{alpha}
\bibliography{bibliiotbisymGrass}

\section{Appendix: the adjoint variety of type $A_n$}

In this appendix, we compute quantum Chevalley formulas for $Y \subset X$ a general hyperplane section of $X$, the adjoint variety of type $A_n$. Note that $\Pic(X) = \ZZ^2$: we have that $X$ is a general hyperplane section of the Segre embedding $\PP^{n} \times {\PP^n}^\vee \subset \PP(\mathfrak{sl}_{n+1})$. We have two maps $p_1 : X \to \PP^n$ and $p_2 : X \to {\PP^n}^\vee$. For $i = 1,2$, let $h_{i,X}$ be the inverse image by $p_i$ of the hyperplane class. We have $\omega_X = \cO_X(-n)$ and $\cO_X(1) = \cO_X(h_{1,X} + h_{2,X})$. Let $j : Y \subset X$ be a general hyperplane section and let $h_i = j^*h_{i,X}$.

Roots are indexed by pairs of integers $1 \leq i < j \leq n+1$ with $\alpha_{i,j} = \alpha_i + \cdots + \alpha_{j-1}$. Let $[{\rm line}_1] = \sigma_{-\alpha_{2,n+1}}$ and $[{\rm line}_2] = \sigma_{-\alpha_{1,n}}$. For $\eta \in \HHH_2(X,\ZZ)$ we may identify $\eta$ with the pair $(d_1,d_2)$ with $d_i = \langle \eta , \cO_X(h_{i,X}) \rangle$ via $\eta = d_1[{\rm line}_1] + d_2[{\rm line}_2]$.

Note that Lemma \ref{lem:comp-GW} and Proposition \ref{prop:comp-GW} are true in type $A_n$. We extend Lemma \ref{lem:deg2} and Proposition \ref{prop:deg2}. 

\begin{proposition}
  \label{prop:deg2an}
  The only non-vanishing Gromov-Witten invariants of the form $\langle \sigma_Y , \sigma'_Y , h \rangle_\eta^Y$ are obtained for $\langle \eta ,\cO_X(1) \rangle = 1$ or for the invariants
  $$\langle [{\rm pt}] , [{\rm line}_i] , h_j \rangle_{(1,1)}^Y = 1 \textrm{ for all } i,j \in [1,2].$$
\end{proposition}

\begin{proof}
  The fact that all invariants vanish except for $\langle \eta ,\cO_X(1) \rangle \leq 2$ is a degree computation and similar to the proof of Lemma \ref{lem:deg2}. Furthermore, if $\langle \eta ,\cO_X(1) \rangle = 2$ then up to reordering we have $\sigma_Y = [{\rm pt}]$ and $\sigma'_Y = [{\rm line}_1]$ or $\sigma'_Y = [{\rm line}_2]$. Note also that since there is no curve in $X$ of degree $(2,0)$ or $(0,2)$ meeting a general point, a general line and a general representative of $h_{1,X}$ or $h_{2,X}$, we must have the vanishings $\langle [{\rm pt}] , [{\rm line}_i] , h_j \rangle_{(2,0)}^Y = 0 = \langle [{\rm pt}] , [{\rm line}_i] , h_j \rangle_{(0,2)}^Y$ for all $i,j \in [1,2]$.

  As in the proof of Proposition \ref{prop:deg2}, the locus in $X$ covered by conics passing through $x_\Theta$ and meeting the line $(x_{-\Theta},x_{\alpha_n - \Theta})$ is the closure in $\PP(\mathfrak{sl}_{n+1})$ of the matrices of the form
  $$\left(\begin{array}{ccccc}
  -st & 0 & \cdots & 0 & s^2 \\
  0 & \vdots & \vdots & \vdots & 0 \\
  0 & 0 & \cdots & 0 & 0 \\
  -ut^2 & 0 & \cdots & 0 & ust \\
  -t^2 & 0 & \cdots & 0 & st \\
  \end{array}
  \right).$$
In particular, the same computation as in the proof of Proposition \ref{prop:deg2} shows that there is a unique curve of class $(1,1)$ meeting a general representative of $[{\rm pt}]$ and $[{\rm line 1}]$ giving the formulas $\langle [{\rm pt}] , [{\rm line}_1] , h_1 \rangle_{(1,1)}^Y = 1 = \langle [{\rm pt}] , [{\rm line}_1] , h_2 \rangle_{(1,1)}^Y$. Formulas involving $[{\rm line}_2]$ follow by symmetry.
\end{proof}

In the following result, we set $\sigma_\alpha = 0$ if $\alpha$ is not a root.

\begin{theorem}[Quantum Chevalley formula]
\label{thm:qChevalley_An}
  Let $X$ be adjoint of type $A_n$ and $Y \subset X$ be a general hyperplane section.
\begin{enumerate}
\item If $j - i \geq 3$, then $h_1 \star \sigma_{\alpha_{i,j}} = \sigma_{\alpha_{i,j-1}}$ and $h_2 \star \sigma_{\alpha_{i,j}} = \sigma_{\alpha_{i+1,j}}$.
\item We have $h_1 \star \sigma_{\alpha_{i,i+2}} = \sigma_{\alpha_{i}} + \sigma_{-\alpha_{i}} + \sigma_{-\alpha_{i-1}} + \sigma_{-\alpha_{i+1}} + q_1 \delta_{i,1}$.\\
 We have $h_2 \star \sigma_{\alpha_{i,i+2}} = \sigma_{\alpha_{i+1}} + \sigma_{-\alpha_{i+1}} + \sigma_{-\alpha_{i}} + \sigma_{-\alpha_{i+2}} + q_2 \delta_{i,n-1}.$
\item We have $h_1 \star \sigma_{\alpha_i} = h_1 \star \sigma_{-\alpha_i}$ and
  $h_2 \star \sigma_{\alpha_i} = h_2 \star \sigma_{-\alpha_i}$.
\item If $j - i \leq n - 2$, then $h_1 \star \sigma_{-\alpha_{i,j}} = \sigma_{-\alpha_{i-1,j}} + q_1 \delta_{i,1} \sigma_{\Theta - \alpha_{i,j}}$ and 
  $h_2 \star \sigma_{-\alpha_{i,j}} = \sigma_{-\alpha_{i,j+1}} + q_2 \delta_{j,n+1} \sigma_{\Theta - \alpha_{i,j}}$.
  \item We have $h_1 \star \sigma_{-\alpha_{1,n}} = q_1(\sigma_{\alpha_n} + \sigma_{-\alpha_n} + \sigma_{-\alpha_{n-1}}) + q_1q_2$ and $h_1 \star \sigma_{-\alpha_{2,n+1}} = \sigma_{-\alpha_{1,n+1}} + q_1q_2$.\\
    We have $h_2 \star \sigma_{-\alpha_{2,n+1}} = q_2(\sigma_{\alpha_1} + \sigma_{-\alpha_1} + \sigma_{-\alpha_{2}}) + q_1q_2$ and $h_2 \star \sigma_{-\alpha_{1,n}} = \sigma_{-\alpha_{1,n+1}} + q_1q_2$.
  \item We have $h_1 \star [{\rm pt}] = q_1 \sigma_{-\alpha_{n-1,n+1}} + q_1q_2(h_1 + h_2)$ and $h_2 \star [{\rm pt}] = q_2 \sigma_{-\alpha_{1,3}} + q_1q_2(h_1 + h_2)$.
  \end{enumerate}
\end{theorem}

\begin{proof}
  We compute the mutiplication with $h_1$, a similar argument works for the multiplication with $h_2$. For $\sigma \in \HHH^*(Y)$, we write ${\rm Coef}_{\sigma_\alpha}(\sigma)$ for the coefficient of $\sigma_\alpha$ in the expansion of $\sigma$ in the basis $(\sigma_\alpha)_{\alpha \in \aleph}$. For $\sigma_X \in \HHH^*(X)$, we write ${\rm Coef}_{\sigma_{\alpha,X}}(\sigma_X)$ for the coefficient of $\sigma_{\alpha,X}$ in the expansion of $\sigma_X$ in the basis $(\sigma_{\alpha,X})_{\alpha \in \aleph}$. We thus have
  $$\sigma = \sum_{\alpha \in \aleph} {\rm Coef}_{\sigma_{\alpha}}(\sigma) \sigma_\alpha \textrm{ and } \sigma_X = \sum_{\alpha \in \aleph} {\rm Coef}_{\sigma_{\alpha,X}}(\sigma_X) \sigma_{\alpha,X}.$$
  
  (1) For $j - i \geq 3$, we have $\sigma_{\alpha_{i,j}} = j^*\sigma_{\alpha_{i,j},X}$ and there is no quantum correction. In particular $h_1 \star\sigma_{\alpha_{i,j}} = h_1 \cup \sigma_{\alpha_{i,j}} = j^*(h_{1,X} \cup \sigma_{\alpha_{i,j},X}) = j^*\sigma_{\alpha_{i,j-1},X} = \sigma_{\alpha_{i,j-1}}$. 

  (2) We have $\sigma_{\alpha_{i,i+2}} = j^*\sigma_{\alpha_{i,i+2},X}$. In particular $h_1 \cup \sigma_{\alpha_{i,i+2}} = j^*(h_{1,X} \cup \sigma_{\alpha_{i,i+2},X}) = j^*\sigma_{\alpha_{i,i+1},X} = j^*\sigma_{\alpha_{i},X} = \sigma_{\alpha_{i}} + \sigma_{-\alpha_{i}} + \sigma_{-\alpha_{i-1}} + \sigma_{-\alpha_{i+1}}$. The only non vanishing Gromov-Witten invariants are $\langle h_1 , \sigma_{\alpha_{i,i+2}} , \sigma \rangle_{\eta}^Y$ with $\eta = (1,0)$. We have $\langle h_1 , \sigma_{\alpha_{i,i+2}} , \sigma \rangle_{\eta}^Y = \langle h_{1,X} , j_*\sigma_{\alpha_{i,i+2}} , j_*\sigma \rangle_{\eta}^X = \langle h_{1,X} , \sigma_{\alpha_i,X} + \sigma_{\alpha_{i+1},X} , j_*\sigma \rangle_{\eta}^X = \delta_{i,1}{\rm Coef}_{[{\rm pt}]}(j_*\sigma)$ proving the first formula. 

  (3) This follows from the projection formula $j_*(h_1 \cup \sigma_{\alpha_i}) = h_{1,X} \cup j_*\sigma_{\alpha_i} = h_{1,X} \cup j_*\sigma_{-\alpha_i} = j_*(h_1 \cup \sigma_{-\alpha_i})$, the fact that $j_*$ is injective for $\deg(\sigma_{\alpha_i}) + 1$ and the fact that $\langle h_1 , \sigma_{\alpha_{i}} , \sigma \rangle_{\eta}^Y = \langle h_{1,X} , j_*\sigma_{\alpha_{i}} , j_*\sigma \rangle_{\eta}^X = \langle h_{1,X} , j_*\sigma_{-\alpha_{i}} , j_*\sigma \rangle_{\eta}^X = \langle h_1 , \sigma_{-\alpha_{i}} , \sigma \rangle_{\eta}^Y$ with $\eta = (1,0)$ are the only non-vanishing Gromov-Witten invariants. 

  (4) We have $j_*(h_1 \cup \sigma_{-\alpha_{i,j}}) = h_{1,X} \cup j_*\sigma_{-\alpha_{i,j}} = j_*\sigma_{-\alpha_{i-1,j}}$ so the classical part follows from the injectivity of $j_*$ in the corresponding degree. The only non-vanishing Gromov-Witten invariants are of the form $\langle h_1 , \sigma_{-\alpha_{i,j}} , \sigma \rangle_{\eta}^Y = \langle h_{1,X} , j_*\sigma_{-\alpha_{i,j}} , j_*\sigma \rangle_{\eta}^X = \langle h_{1,X} , \sigma_{-\alpha_{i,j}, X} , j_*\sigma \rangle_{\eta}^X$ with $\eta = (1,0)$. We have $\langle h_{1,X} , \sigma_{-\alpha_{i,j},X} , j_*\sigma \rangle_{\eta}^X = \delta_{i,1}{\rm Coef}_{\sigma_{\Theta - \alpha_{i,j},X}}((j_*\sigma)^\vee)$. The condition $(j_*\sigma)^\vee = \sigma_{\Theta - \alpha_{i,j},X}$ is equivalent to $j_*\sigma = \sigma_{w_0(\Theta - \alpha_{i,j}),X}$ and in turn to $\sigma = \sigma_{w_0(\Theta - \alpha_{i,j})}$ and thus to $\sigma^\vee = \sigma_{\Theta - \alpha_{i,j}}$, proving the first formula.

  (5) The classical part follows from $h_{1,X} \cup \sigma_{-\alpha_{1,n}} = 0$  and $h_{1,X} \cup \sigma_{-\alpha_{2,n+1}} = \sigma_{-\alpha_{1,n+1}}$. The term in $q_1q_2$ follows from Proposition \ref{prop:deg2an} and we only need to compute degree $1$ Gromov-Witten invariants. The only non-vanishing Gromov-Witten invariants are of the form $\langle h_1 , \sigma_{-\alpha_{1,n}} , \sigma \rangle_{\eta}^Y = \langle h_{1,X} , j_*\sigma_{-\alpha_{1,n}} , j_*\sigma \rangle_{\eta}^X = \langle h_{1,X} , \sigma_{-\alpha_{1,n},X} , j_*\sigma \rangle_{\eta}^X = {\rm Coef}_{\sigma_{\alpha_n,X}^\vee}(j_*\sigma)$ and also $\langle h_1 , \sigma_{-\alpha_{2,n+1}} , \sigma \rangle_{\eta}^Y = \langle h_{1,X} , j_*\sigma_{-\alpha_{2,n+1}} , j_*\sigma \rangle_{\eta}^X = \langle h_{1,X} , \sigma_{-\alpha_{2,n+1},X} , j_*\sigma \rangle_{\eta}^X = 0$ with $\eta = (1,0)$. For $\sigma$ an element in the basis $(\sigma_\alpha)_{\alpha \in \Phi_\aleph \cup -\Phi_\aleph}$, the condition $j_*\sigma = \sigma_{\alpha_n,X}^\vee = \sigma_{-\alpha_1,X}$ is equivalent to $\sigma = \sigma_{\alpha_1}$ or $\sigma = \sigma_{-\alpha_1}$. We thus get $h_1 \star \sigma_{-\alpha_{1,n}} = q_1(\sigma_{\alpha_1}^\vee + \sigma_{-\alpha_1}^\vee) + q_1q_2$ and $h_1 \star \sigma_{-\alpha_{2,n+1}} = \sigma_{-\alpha_{1,n+1}} + q_1q_2$. The result follows from Example \ref{ex:classe-duale}.

  (6) We have $h_1 \cup \sigma_{-\alpha_{1,n+1}} = 0$ for degree reasons. The $q_1q_2$ terms follows from Proposition \ref{prop:deg2an} as well as the fact that the only other non-vanishing Gromov-Witten invariants are of the form $\langle h_1 , \sigma_{-\alpha_{1,n+1}} , \sigma \rangle_{\eta}^Y = \langle h_{1,X} , j_*\sigma_{-\alpha_{1,n+1}} , j_*\sigma \rangle_{\eta}^X = \langle h_{1,X} , \sigma_{-\alpha_{1,n+1},X} , j_*\sigma \rangle_{\eta}^X = {\rm Coef}_{\sigma_{-\alpha_n,X}^\vee}(j_*\sigma)$. Note that there exists $\tau_X \in \HHH^*(X)$ such that $\sigma = j^*\tau_X$ and we have ${\rm Coef}_{\sigma_{-\alpha_n,X}^\vee}(j_*\sigma) = {\rm Coef}_{\sigma_{-\alpha_n,X}^\vee}(h_X \cup \tau_X)$. We thus get $h_1 \star [{\rm pt}] = q_1 \sigma_{\alpha_{1,3}}^\vee + q_1q_2(h_1 + h_2) = q_1 \sigma_{-\alpha_{n-1,n+1}} + q_1q_2(h_1 + h_2)$.
\end{proof}

\begin{remark}
Using the above formula, we can check that setting $q_1 = q_2 = 1$ the algebra $\QH(Y)/(q_1 - 1,q_2 - 1)$ is semi-simple for $n = 2$ and not semi-simple for $n = 3$. This agrees with Conjecture \ref{conjecture:ks-an}.
\end{remark}

\begin{proposition}
The quantum cohomology $\QH(Y)/(q_1-q_2)$ is semi-simple when $n$ is even.
\end{proposition}

\begin{proof}
The proof follows the same lines as the proof of Theorem \ref{thm:ss}. First of all notice that, since $n$ is even, $q_1+(-1)^{n}q_2\neq 0$ in general. From the results in \cite{adjoint}, the algebra $\QH(X)/(q_1-q_2)$ is semi-simple, and from the presentation of the cohomology given in \cite{adjoint} one can check that the eigenvalues of $h_X$ are all non-zero and with multiplicity one. Moreover the minimal polynomial of $E_X$ (the multiplication-by-$h_X$ endomorphism) can be written as $P(T^{c_1(X)})$. Along the same lines of Corollaries \ref{cor:min_pol_Y} and \ref{cor:kerY} and using Theorem \ref{thm:qChevalley_An}, one can check that the non-zero eigenvalues of $E_Y$ (the multiplication-by-$h$ endomorphism) have multiplicity one and $\Ker E_Y=W:=\Hna + \QQQ([\pt]-q(\sigma_{-\alpha_1}+\sigma_{-\alpha_n})-q^2)$. Thus the nilpotent elements of $\QH(Y)/(q_1-q_2)$ are contained in $W$.

As in Proposition \ref{prop:sigma}, if $A$ denotes the subalgebra generated by $h$, there exists a unique $\sigma=([\pt]-q(\sigma_{-\alpha_1}+\sigma_{-\alpha_n})-q^2)+q\Gamma\in A\cap \Ker E_Y$ with $\Gamma\in \Hna$ such that $\sigma^2=\lambda_0 q\sigma$ and $\sigma \Gamma=\lambda_0 q^2 \Gamma$, where $\lambda_0$ can be shown to be equal to $-2 - 2(\sigma_{-\alpha_1}+\sigma_{-\alpha_n})\cup \Gamma + \Gamma \cup \Gamma$. The matrix of the quadratic form $q(t,\Gamma)=-2t^2 - 2t(\sigma_{-\alpha_1}+\sigma_{-\alpha_n})\cup \Gamma + \Gamma \cup \Gamma$ is equal to the matrix $C_\aleph-4I_\aleph$ with root system the affine root system of type $\hat{A}_n$. Note that $C_\aleph$ is positive with a unique zero eigenvalue (see \cite[Lemma 4.5]{kac}) so that the eigenvalues of $C_\aleph$ are real numbers $\lambda \in [0,4[$ (because the non-zero eigenvalues are those of type $A_n$). Therefore the symmetric matrix $4I_\aleph-C_\aleph$ is positive definite thus $\lambda_0<0$.

Now, a nilpotent element $x$ can always be written as $x=\lambda \sigma + q\Gamma'$ with $\lambda\neq 0$. As in the proof of Theorem \ref{thm:ss}, the rational number $\lambda_0$ satisfies $\lambda_0=-\frac{\Gamma'\cup \Gamma'}{\lambda^2}$. But $n$ being even, the dimension $\dim Y/2 = n - 1$ is odd and the intersection product restricted to $\Hna$ is negative definite. We obtain that $\lambda_0>0$, a contradiction.
\end{proof}

\end{document}